%% file: GL3GL2.tex
\documentclass[10pt]{amsart}
\input{amssymbol}

\usepackage{color}
\usepackage{amsmath}
\usepackage{hyperref}

\def\G{{\rm GL_2}}
\def\SU{{\rm SU}}

\def\x{\times}
\def\b{\bar}
\def\h{\hat}
\def\t{\tilde}
\def\bt{\boxtimes}
\def\mf{\mathfrak}
\def\itPi{\mathit \Pi}
\def\cF{K}

\newtheorem*{propa}{Proposition A}
\newtheorem*{lma}{Lemma A}
\newtheorem*{cora}{Corollary A}

\title[On Deligne's conjecture for certain automorphic L-functions for $\GL(3)\times \GL(2)$]
{On Deligne's conjecture for central values of certain automorphic L-functions on 
$\GL(3)\times \GL(2)$}
\subjclass[2010]{11F41, 11F67}

\author{SHIH-YU CHEN AND YAO CHENG}
\address{Department of Mathematics~\\National Taiwan University ~ \\
No. 1, Sec. 4, Roosevelt Road, Taipei 10617, Taiwan~}
\email{r98221018@ntu.edu.tw}
\address{Department of Mathematics~\\National Taiwan University ~ \\
No. 1, Sec. 4, Roosevelt Road, Taipei 10617, Taiwan~}
\email{ai852obrian@gmail.com}
\date{\today}

\date{\today}
\begin{document}

\begin{abstract}
We prove Deligne's conjecture for central critical values of certain automorphic $L$-functions for 
$\GL(3)\times\GL(2)$. The proof is base on rationality results for central critical values of triple 
product $L$-functions, which follow from establishing explicit Ichino's formulae for trilinear 
period integrals for Hilbert cusp forms on totally real \etale cubic algebras over $\Q$.
\end{abstract}

\maketitle
\tableofcontents

\section{Introduction}
The purpose of this paper is to establish the explicit Ichino formula for twisted triple product 
$L$-functions. As an application of our formula, we establish new cases on the algebraicity of central critical 
values of certain class of automorphic $L$-functions for $\GL(3)\times \GL(2)$ divided by the 
associated Deligne's periods. To begin with, let $f$ and $g$ be elliptic newforms of weights 
$\kappa'$ and $\kappa$, level $\Gamma_0(N_1)$ and $\Gamma_0(N_2)$, respectively. 
We let $L(s,{\rm Sym}^2(g)\otimes f)$ be the motivic $L$-function associated with
${\rm Sym}^2(g)\otimes f $. Put 
\[
w=2\kappa+\kappa'-3;\quad \epsilon=(-1)^{\kappa'/2-1}.
\] 
Denote by $\Omega_f^{\pm}$ the Shimura's periods of $f$ in \cite{Shimura1977} and define the
Deligne's period $\Omega_{f,g} \in \C^{\times}$ by
\begin{align*}
\Omega_{f,g}=\left \{ \begin{array}{ll} (2\pi\sqrt{-1})^{3-3\kappa}(\sqrt{-1})^{1-\kappa'}
\langle f,f\rangle \Omega_f^{\epsilon} & \mbox{ if }
2\kappa\leq \kappa',\\
(2\pi\sqrt{-1})^{2-\kappa-\kappa'}
\langle g,g\rangle^2
\Omega_f^{\epsilon} & \mbox{ if }2\kappa>\kappa'. 
\end{array} \right .
\end{align*}
Our main result is as follows. 
\begin{thma} (Cor. \ref{C:algebraicity for GL_2 times GL_3 unbalanced case} and 
\ref{C:algebraicity for GL_2 times GL_3 balanced case} ) 
Suppose that $N_1$ and $N_2$ are square-free.  
%\beqcd{*}N_1>1\text{ if }2\kappa>\kappa'.\eeqcd For every $\sigma \in {\rm Aut}(\C)$, 
We have
\begin{align*}
\left( 
\frac{L((w+1)/2,{\rm Sym}^2(g)\otimes f)}
{(2\pi\sqrt{-1})^{3(w+1)/2} \Omega_{f,g}^{\epsilon}}
\right )^{\sigma}
&=
\frac{L((w+1)/2,{\rm Sym}^2(g^{\sigma})\otimes f^{\sigma})}{(2\pi\sqrt{-1})^{3(w+1)/2}
\Omega_{f^{\sigma},g^{\sigma}}^{\epsilon} }.
\end{align*}
\end{thma}

\begin{Remark}
% \noindent\begin{enumerate} 
If $N_1=1$ and $2\kappa>\kappa'$, then the above algebraicity result was obtained by Ichino
\cite[Corollary 2.6]{Ichino2005} via the explicit pullback formula for Saito-Kurokawa lifts 
($N_2=1$ and $\kappa=\kappa'/2+1$) and by Xue ($N_2=1$) using a different but closely related
approach \cite{Xuepreprint}. The first author has generalized Ichino's pullback formula of 
Saito-Kurokawa lifts in \cite{Chen2018} if $N_2$ is furthter assumed to be odd and cubic-free. 

Our result covers the remaining cases and thus settles down Deligne's conjecture for the central
value of the $L$-functions for $\Sym^2(g)\ot f$ at least when the levels of $f$ and $g$ are 
square-free.
%\item The assumption on the square-freeness of $N_1$ and $N_2$ is not serious and definitely can 
%be removed later,
%For instance, we can calculate the local period integral at $p\nmid N_1$ regardless how small 
%$|N_2|_{\Q_p}$ is. If $p^2 \mid N_1$, the local period integral at $p$ is also computable when 
%$|N_1|_{\Q_p}\leq |N_2|_{\Q_p}^2$ (see \cite{Hu2017}).
%but the hypothesis \eqref{*} is crucial for our method. \end{enumerate}
\end{Remark}

We remark that Raghuram has proved the algebraicity of the central critical values of 
the Rankin-Selberg $L$-functions attached to regular algebraic cuspidal automorphic representations
on $\GL(n)\times \GL(n-1)$ in a quite general setting \cite{Raghuram2009}. His method is based on 
a cohomological interpretation of the Rankin-Selberg zeta integral, and specializing the result of 
Raghuram to $n=3$, one also obtains the algebraicity of the central critical value of 
$L(s,{\rm Sym}^2(g)\otimes f)$ divided by certain cohomological period for $\GL(3)\times\GL(2)$ in 
the case $2\kappa>\kappa'$. However, our result in this case is not covered by \cite{Raghuram2009} 
in the sense that the periods in both results are quite different. More precisely, the periods in 
our main theorem coincide with Deligne's period described by Blasius in the appendix to \cite{Orloff1987} while Raghuram uses the period $p^\pm(\Pi)$ obtained from the comparison between deRham and Betti cohomology groups for $\GL(3)$ \cite[\S 3.2.1]{Raghuram2009}. It seems 
a difficult problem to study directly the relation between Deligne's period and Raghuram 's cohomological period. Our result combined with the non-vanishing hypothesis of central $L$-values would give a comparison between these two periods.

Our approach also offers the algebraicity of the central critical value of symmetric cube 
$L$-functions with the assumption on the non-vanishing of $L$-values with cubic twist. 

\begin{thmb}[\corref{C:algebraicity for symmetric cubic}]
Suppose that $N_1>1$ and there exist a cubic Dirichlet character $\chi$ such that 
$L\left(\frac{\kappa'}{2} ,f\ot\chi\right)\neq 0$. 
For $\sigma \in {\rm Aut}(\C)$, we have
\begin{align*}
\left(
\frac{L((w+1)/2,{\rm Sym}^3(f))}{\pi^{2\kappa'-1}(\sqrt{-1})^{\kappa'}
\langle f,f\rangle (\Omega_f^{\epsilon})^2} 
\right )^{\sigma}
=
\frac{L((w+1)/2,{\rm Sym}^3(f^{\sigma}))}{\pi^{2\kappa'-1}(\sqrt{-1})^{\kappa'}
\langle f^{\sigma},f^{\sigma}\rangle (\Omega_{f^{\sigma}}^{\epsilon})^2}.
\end{align*}
\end{thmb}
The hypothesis on the non-vanishing of cubic twists of $L$-values is expected to hold in general 
but seems unfortunately a far-reaching problem at this moment. So far this hypothesis is only 
known to be satisfied for cuspidal automorphic representations on $\GL_2(\A_{K})$ when $\Q(\sqrt{-3})\subset K$ 
in \cite{BFH2005}.\\

%The central critical value of $L(s,{\rm Sym}^2(g)\otimes f)$ is related to the central critical 
%value of certain twisted triple product $L$-function of some automorphic representations on $\GL_2(\A_E)$. 
%Here $E=\cF \times \Q$ for some quadratic extension $\cF$ of $\Q$. Theorem A is a consequence of the 
%algebraicity of the central critical value of the triple product $L$-function. We prove explicit 
%central value formula (Theorems \ref{T:central value formula unbalanced case} and \ref{T:central value formula for balanced case}), which might be of independent interest, for the triple product $L$-function. 
%In particular, we carry out the explicit calculation of certain archimedean local period integral 
%defined by Ichino in \cite{Ichino2008}.
%The algebraicity of the central value of the triple product $L$-function is a corollary of the 
%central value formula. Here for our purpose, we take $\cF$ to be a real quadratic field. In 
%\cite{Ichino2005}, a central value formula is given in the case that $\cF$ is an imaginary quadratic field. 

Our proof of Theorem A is based on an explicit Ichino's central value formula for the twisted 
triple product $L$-functions. Let $K$ be a real quadratic field and let $g_{\cF}$ be the Hilbert 
modular newform over $\cF$ associated to $g$ obtained by the base change lift. 
Let $L(s,g_{\cF}\otimes f)$ be the triple product $L$-function associated to $g_{\cF}\otimes f$. 
Let $\tau_K$ be the quadratic Dirichlet character associated with $K/\Q$.  
From the following factorization of $L$-functions \begin{align*}
L(s,g_{\cF}\otimes f)
=
L(s,{\rm Sym}^2(g)\otimes f)L\left( s-\kappa+1,f\otimes \tau_K \right),
\end{align*}
one can deduce easily the algebraicity of $L\left(\frac{w+1}{2},{\rm Sym}^2(g)\otimes f \right )$ 
(divided by the associated Deligne's period) from that of the central value 
$L\left (\frac{w+1}{2},g_{\cF}\otimes f \right)$ of the twisted triple product and that of 
the central value $L\left ( \frac{\kappa'}{2},f\otimes \tau_\cF \right)$ of elliptic modular forms 
whenever $L(\frac{\kappa'}{2},f\ot\tau_\cF)$ does not vanish. The algebraicity of critical $L$-values 
of elliptic modular forms with Dirichlet twists is a classical result due to Shimura, so the main 
task is to choose a nice real quadratic field $K$ with $L(\frac{\kappa'}{2},f\ot\tau_K)\not =0$ and 
show the algebraicity of the central value $L(\frac{w+1}{2},g_\cF\ot f)$, for which one appeals to 
Ichinos's formula in \cite{Ichino2008}. 
More precisely,  if the global sign in the functional equation of the automorphic $L$-function for 
the twisted triple product $g_\cF\ot f$ is $+1$, then Ichino's formua alluded to above asserts that 
there exists a quaternion algebra $D$ over $\Q$ so that the central critical value 
$L\left (\frac{w+1}{2},g_{\cF}\otimes f \right)$ is the ratio between the square of the 
global trilinear period integral of an automorphic form on $D^\x(\A_\cF)\times D^\x(\A)$ and a 
product of certain local zeta integrals. Taking into account the functional equation and the Galois 
invariance of the global sign, we may assume the global sign of $\Sym^2(g)\ot f$ is $+1$. Then by 
using a result of Friedberg and Hofffstein \cite{FriedbergHoffstein1995}, we can choose a real 
quadratic field $K$ such that (i) $L(\frac{\kappa'}{2},f\ot\tau_K)\not =0$, (ii) the sign of 
$g_\cF\ot f$ is $+1$ and (iii) the quaternion algebra $D$ in Ichino's formula is the matrix algebra 
(resp. a definite quaternion algebra) over $\Q$ in the case $2\kappa \leq \kappa'$ 
(resp. $2\kappa>\kappa'$ if we assume further that $N_1>1$). To obtain the explicit Ichino's central value 
formula, we calculate the local period integral at each place (Theorems \ref{T:central value formula 
unbalanced case} and  \ref{T:central value formula for balanced case}) in terms of global period 
integral, and as a consequence, we obtain the algebraicity of the central value 
$L\left (\frac{w+1}{2},g_{\cF}\otimes f \right)$ 
(Corollaries  \ref{C:algebraicity for triple, unbalanced case} and 
\ref{C:algebraicity for triple for balanced case}) by a standard argument. 
%The requirment  on $D$ according to 
%$\kappa'\geq 2\kappa$ or $2\kappa >\kappa'$ would give some conditions on $\cD_0$. 
%The non-vanishing of $L\left ( \frac{\kappa'}{2},f\otimes \tau_\cF \right)$ also give some 
%conditions on $\cD_0$. If the two sets of conditions are incompatible, then we can prove that the
%root number of $L(s,{\rm Sym}^2(g)\otimes f)$ is equal to $-1$. As the local root number of 
%$L(s,{\rm Sym}^2(g)\otimes f)$ are invariant under Galois action in our case, we conclude that 
%$L\left (\frac{w+1}{2}, {\rm Sym}^2(g^{\sigma})\otimes f^{\sigma} \right)=0$ for all 
%$\sigma \in {\rm Aut}(\C)$. 

The idea of the proof for Theorem B is similar. 
Assume $\chi$ is a cubic Dirichlet character such that 
$L\left (\frac{\kappa'}{2} ,f\otimes \chi \right) \neq 0$. 
Let $E$ be the totally real cubic Galois extension over $\Q$ cut out by $\chi$ and 
let $f_E$ be the Hilbert modular newform over $E$ associated to $f$ via the base change lift. 
Consider the degree eight triple product $L$-function $L(s,f_E)$ associated to $f_E$. 
Then we have the factorization of 
$L$-functions : 
\[
L(s,f_E)
=
L\left(s,{\rm Sym}^3(f)\right)L(s-\kappa'+1,f\otimes \chi)L(s-\kappa'+1,f\otimes \chi^2).
\]
Thus the algebraicity of $L\left (\frac{w+1}{2},{\rm Sym}^3(f)\right)$ is a consequence 
of the algebraicity of  $L\left (\frac{w+1}{2},f_E\right)$, which again can be deduced from 
the explicit Ichino central value formula in this case.

This paper is organized as follows. We first study the local zeta integrals in Ichino's formula. In in \S2, we introduce the local zeta integrals and fix the test vectors used in the subsequent local calculation. After recalling basic properties of local matrix coefficients for $\GL(2)$ in \S 3, we carry out the calculations of local zeta integrals in the cases of the matrix algebra and the 
division algebra in \S4 and \S5 respectively. In particular, we compute the archimedean zeta integrals explicitly. 
In \S6, we recall Ichino's formula and in \S7, we prove its explicit version in 
\thmref{T:central value formula unbalanced case} and 
\thmref{T:central value formula for balanced case} 
as well as the main result  in \S8. In the appendix, we 
follow a computation of Bhagwat \cite{Bhagwat2014} to determine Deligne's period for 
$\Sym^2(g)\ot f$. This is well-known to experts, but we include it here for the sake of completeness.

\section{Local zeta integrals}
The purpose of this section is to fix the test vectors and to define the local zeta integrals 
in our local calculation. These local zeta integrals are used to establish explicit Ichino's
formulae. 
 
\subsection{Notation and assumptions}
Let $F$ be a local field of characteristic zero. When $F$ is non-archimedean, denote $\mathcal{O}_F$ the ring of integers of F, $\varpi_F$ a prime element, and ${\rm ord}_F$ be the valuation on $F$ normalized so that ${\rm ord}_F(\varpi_F)=1$.  Let $|\cdot|_F$ be the absolute value on $F$ normalized so that $|\varpi_F|_F^{-1}$ is equal to the cardinality of $\mathcal{O}_F / \varpi_F\mathcal{O}_F$. When $F$ is archimedean, let $|\cdot|_{\R}$ be the usual absolute value on $\R$ and $|z|_\C=z\overline{z}$ on $\C$.

Let $E$ be an \etale cubic algebra over a local field $F$ of characteristic zero. Then $E$ is (i) $F\x F\x F$ three copies of $F$, or 
(ii) $\cF\x F$, where $\cF$ is a quadratic extension of $F$, or (iii) is a cubic field extension 
of $F$. 
Let $D$ be a quaternion algebra over $F$. If $L$ is a $F$-algebra, let $D^\x(L):=(D\ot_F L)^\x$. Let $\itPi$ be a unitary irreducible admissible 
representation of $D^\x(E)$ whose central character is trivial on $F^{\x}$. 
Let $\itPi'$ be the unitary irreducible admissible representation of $\G(E)$ associated to $\itPi$
via the Jacquet-Langlands correspondence. Therefore $\itPi'=\itPi$ if $D={\rm M}_2(F)$ is the matrix
algebra. Notice that $\itPi'=\pi_1\bt\pi_2\bt\pi_3$ (if $E=F\x F\x F$), or $\itPi'=\pi'\bt\pi$ 
(if $E=K\x F$), where $\pi_j$ ($j=1,2,3$) and $\pi$ are unitary irreducible admissible
representations of $\G(F)$, and $\pi'$ is a unitary irreducible admissible representation of $\G(K)$.
We make the following assumptions on the triplet $(F,E,\itPi)$ in this section and \S4,
\S5. 

\begin{itemize}
\item
If $F$ is archimedean, then $F=\R$ and $E=\R\x\R\x\R$.
\item
When $F=\R$, $\itPi'$ is a (limit of) discrete series with the minimal weight 
$\ul{k}=(k_1,k_2,k_3)$ and the central character 
${\rm Sgn}^{k_1}\bt{\rm Sgn}^{k_2}\bt{\rm Sgn}^{k_3}$ for some positive integers $k_1, k_2, k_3$. 
\item
When $F$ is non-archimedean, we assume $\pi_1, \pi_2, \pi_3$  
and $\pi', \pi$ and $\itPi'$ (when $E$ is a field) are unramified (special) representations whose
central characters are trivial.  
\item
We assume $\Lambda(\itPi')<1/2$, where $\Lambda(\itPi')$ is defined in 
\cite[pp. 284-285]{Ichino2008}.  
\item
We assume ${\rm Hom}_{D^{\x}(F)}(\itPi,\C)\neq\stt{0}$.
\end{itemize}

\begin{Remark}
By the results of Prasad \cite{Prasad1990} and \cite{Prasad1992}, we have
\[
{\rm dim}_\C{\rm Hom}_{D^{\x}(F)}(\itPi,\C)\leq 1.
\]
When $F=\R$, it follows from \cite[Theorem 9.5]{Prasad1990} that 
${\rm Hom}_{D^{\x}(\R)}(\itPi,\C)\neq\stt{0}$ precisely when (i) $D={\rm M}_2(F)$ and 
$2\cdot {\rm max}\stt{k_1,k_2,k_3}\geq k_1+k_2+k_3$; (ii)  $D$ is the division algebra and
$2\cdot {\rm max}\stt{k_1,k_2,k_3}<k_1+k_2+k_3$. The first case is called the $unbalanced$ case, 
while the second case is called the $balanced$ case.
\end{Remark}

\subsection{The new line}\label{SS:specialine} 
Denote by $V_\itPi$ the representation space of $\itPi$. In what follows, we shall introduce a special one-dimensional
subspace in $V_{\itPi}$, which is called \emph{the new line} $V_{\itPi}^{\rm new}$ of $V_{\itPi}$ . If $F$ is non-archimedean and $\fraka$ is an ideal of $\cO_E$, let 
\[
\cU_0(\fraka)=\stt{g=\pMX{a}{b}{c}{d}\in\GL_2(\cO_E)\mid c\in \fraka }.
\]
Suppose that $D={\rm M}_2(F)$. If $F$ is non-archimedean, then by \cite{Casselman1973}, there is a 
unique ideal $c(\itPi)$ of $\cO_E$ such that
\[
\dim_\C V_\itPi^{\cU_0(c(\itPi))}=1.
\]
The ideal $c(\itPi)$ is called the \emph{conductor} of $\itPi$, and define the new line
$V_{\itPi}^{\rm new}:=V_\itPi^{\cU_0(c(\itPi))}$.
If $F=\R$ and $E=\R\x \R\x \R$, the new line $V_\itPi^{\rm new}$ is the one dimensional subspace of the minimal weight under the $\SO_2(E)$-action.

Suppose that $D$ is division. If $F$ is non-archimedean, and $E\neq K\x F$, then $V_\itPi$ is already 
one-dimensional according to our assumption. In this case, we put $V^{\rm new}_\itPi=V_\itPi$. 
When $E=K\x F$, we have $\itPi=\pi'\bt\pi$, where $\pi$ (resp. $\pi'$) is a unitary irreducible
admissible (resp. generic) representation of $D^{\x}(F)$ (resp. $\G(K)$). In this case, we have 
$V_\itPi=V_{\pi'}\ot V_\pi$ and define the new line $V^{\rm new}_{\itPi}:=V^{\rm new}_{\pi'}\ot V_\pi$.  Finally, if $F=\R$ and $2\,{\rm max}\stt{k_1,k_2,k_3}<k_1+k_2+k_3$, 
we define the new line $V^{\rm new}_\itPi$ to be the one-dimensional subspace
$V^{D^{\x}(\R)}_\itPi$ of $V_\itPi$ \cite[Theorem 9.3]{Prasad1990}.

\subsection{}\label{SS:raising element}
%Define 
%\[\cH(\GL_2(E)):=\begin{cases}\Z[\GL_2(E)]&\text{ if $F$ is non-archimedean},\\[1em]
%\frakg_\C\times \SO(2,E) &\text{ if $F=\R$, $\frakg_\C=\Lie(\GL_2(E))\ot_\R\C$}.
%\end{cases}\]
%We shall introduce a special element $\bft\in\cH(\GL_2(E))$ for the representation $\itPi$.
Let $\frak{g}=\Lie(\GL_2(\R))\ot_\R\C$ and $\frak{U}$ be the universal enveloping algebra of 
$\frak{g}$. We put $\frak{U}_E=\frak{U}\ot\frak{U}\ot\frak{U}$. Let
\begin{align*}
V_+
=&
\begin{pmatrix} 1&0\\0&-1 \end{pmatrix}\otimes 1+
\begin{pmatrix} 0&1\\1&0 \end{pmatrix}\otimes \sqrt{-1}\in\frak{U}
\end{align*}
be the weight raising operator in \cite[Lemma 5.6]{JLbook}. Define the normalized operator by
\[\t{V}_+:=(-\frac{1}{8\pi})\cdot V_+.\]Let $\tau_F\in\G(F)$ be 
given by
\begin{equation}\label{E:tau_F}
\tau_F
=
\begin{cases}
\begin{pmatrix}
-1&0\\0&1
\end{pmatrix}\quad&\text{if $F=\R$},\\
1\quad&\text{if $F$ is nonarchimedean}.
\end{cases}
\end{equation}

Define a special element $\bft\in\frak{U}_E\x{\rm SO}(2,E)$ or $D^{\x}(E)$ attached to $\itPi$ as
follows:
\begin{itemize}
\item 
$F=\R$, $D={\rm M}_2(F)$ and $2\,{\rm max}\stt{k_1,k_2,k_3}\geq k_1+k_2+k_3$. Suppose that $k_3={\rm max}\stt{k_1,k_2,k_3}$. Then 
\[
\bft
=
\left( 1\ot\wtd V_+^\frac{k_3-k_1-k_2}{2}\ot 1, (1,1,\tau_{\R})
\right)
\in 
\frak{U}_E\x \SO(2,E).
\]
\item 
$F$ non-archimedean, $E=F\x F\x F$, $D={\rm M}_2(F)$, $\itPi=\pi_1\bt\pi_2\bt\pi_3$ and precisely
one of $\pi_j$ is unramified special, say $\pi_1$:
\[
\bft
=
\left(
1,
\begin{pmatrix}
\varpi_F^{-1}&0\\0&1
\end{pmatrix}
,1
\right) 
\in\G(F)\x\G(F)\x\G(F).
\]  
\item 
$F$ non-archimedean, $E=\cF\x F$, $\cF/F$ is ramified, $D={\rm M}_2(F)$,
$\itPi=\pi'\bt\pi$ with $\pi'$ spherical and $\pi$ unramified special:
\[
\bft
=
\left( 
\begin{pmatrix}
\varpi^{-1}_{\cF}&0\\0&1
\end{pmatrix},
1
\right)
\in 
\G(\cF)\x\G(F).
\]

\item 
$F$ non-archimedean, $E=\cF\x F$, $\cF/F$ is ramified, $D={\rm M}_2(F)$, $\itPi=\pi'\bt\pi$ with
$\pi'$ unramified special and $\pi$ spherical:
\[
\bft
=
\left(
\begin{pmatrix}
\varpi_{F}&0\\0&1
\end{pmatrix}
,1
\right)
\in 
\G(\cF)\x\G(F).
\] 

\item 
$F$ non-archimedean, $E$ ramified cubic extension, $D={\rm M}_2(F)$, $\itPi$ unramified special:
\[
\bft
=
\begin{pmatrix}
\varpi^{-1}_{E}&0\\0&1
\end{pmatrix}
\in\G(E).
\]
\item 
For all other cases:
\[
\bft
=
1\in D^{\x}(E).
%(1,1,1)\in\G(F)\x\G(F)\x\G(F).
\] 
\end{itemize}

\subsection{Definition of local zeta integrals}\label{SS:local period integral}
We are going to define the local zeta integrals in our local computation except for the balanced 
case, which will be defined by equation \eqref{E:local period integral for the balabced case}.
 
Let $\cJ\in D^{\x}(E)$ be given as follows:
\[
\cJ
=
\begin{cases}
(\tau_{\R},\tau_{\R},\tau_{\R})\in\G(\R)\x\G(\R)\x\G(\R)
\quad&\text{if $F=\R$ and $D={\rm M}_2(\R)$},\\
1\quad&\text{otherwise}.
\end{cases}
\]
Let $\zeta_F(s)$ denote the local zeta function. Therefore,
\[
\zeta_F(s)
=
\begin{cases}
2(2\pi)^{-s}\Gamma(s)\quad&\text{if $K=\C$},\\
\pi^{-s/2}\Gamma(s/2)\quad&\text{if $K=\R$},\\
(1-q_F^{-s})^{-1}\quad&\text{if $K$ is nonarchimedean},
\end{cases}
\]
where $q_F$ is the cardinality of the residue field of $F$ when $F$ is non-archimedean.

\begin{defn}\label{D: definition of zeta integral}
Fix a nonzero $D^{\x}(E)$-invariant pairing $\cB_\itPi\colon V_{\itPi}\times V_{\itPi}\to\C$. 
Let $\phi_{\itPi}\in V^{\rm new}_{\itPi}$ be a non-zero vector in the new line. The normalized local zeta integral is defined by
\begin{align*}
I(\itPi,\bft)
&=
\int_{F^{\x}\backslash D^{\x}(F)}
\frac{\cB_{\itPi}(\itPi(h\bft)\phi_{\itPi},\itPi(\bft)\phi_{\itPi})}
{\cB_{\itPi}(\itPi(\cJ)\phi_{\itPi},\phi_{\itPi})}\,dh,\\
I^*(\itPi,\bft)
&=
\frac{\zeta_F(2)}{\zeta_E(2)}\cdot
\frac{L(1,\itPi',{\rm Ad})}{L(1/2,\itPi', r)}\cdot
I(\itPi,\bft).
\end{align*}
%We devoted to compute the following local zeta integrals in 
%\secref{S:local zeta integral: matrix algebra} and 
%\secref{S:local zeta integral: division algebra}
%\[
%I^*(\itPi,\bft)
%=
%\frac{\zeta_F(2)}{\zeta_E(2)}\cdot
%\frac{L(1,\itPi',{\rm Ad})}{L(1/2,\itPi', r)}\cdot
%I(\itPi,\bft).
%\]
Here the $L$-factors are defined in \cite[pp. 282-283]{Ichino2008}.
\end{defn}

%We have some remarks.
\begin{Remark}\noindent
\begin{itemize}
\item[(1)] 
Since the central character of $\itPi$ is trivial on $F^{\x}$, the integrals are well-defined. 
Moreover, our assumption $\Lambda(\itPi')<1/2$ implies these integrals converge absolutely 
\cite[lemma 2.1]{Ichino2008}.
\item[(2)] 
We note that $\phi_{\itPi}$ is unique up to a constant as well as $\cB_{\itPi}$. Thus 
$I(\itPi,\bft)$ is independence of the choice of $\phi_{\itPi}$ and $\cB_{\itPi}$.  
But it does depend on the choice of the measure $dh$.
%\item[(3)] 
%The integral 
%\[
%\int_{F^{\x}\backslash D^{\x}(F)}\cB_{\itPi}(\itPi(h)\phi,\phi')dh,\quad \phi,\,\phi'\in V_{\itPi},
%\]
%defines an element in ${\rm Hom}_{D^{\x}(F)\x D^{\x}(F)}\left(\itPi\bt\itPi,\C\right)$, 
%whose dimension is at most one
%\cite{Prasad1990}, \cite{Prasad1992} and \cite{Loke2001}. For more details, see 
%\subsubsecref{SSS:local root numbers}.
\end{itemize}
\end{Remark}

\section{Matrix coefficients for $\GL(2)$}
Let $F$ be either $\R$ or a non-archimdean local field. Let $\pi$ be a unitary irreducible admissible 
generic representation of $\G(F)$. Let $\phi_\pi$ be a non-zero vector in the new line $V_\pi^{\rm new}$. Fix a non-zero $\G(F)$-invariant bilinear 
pairing $\cB_\pi\colon \pi\x\t{\pi}\to\C$, where $\t{\pi}$ is the admissible dual of $\pi$.
\begin{defn}\label{D:matrix coefficient}
We define the matrix coefficient associate with an element 
$t\in\frak{U}\x{\rm O}(2)$ or $t\in\G(F)$ by
\[
\Phi_{\pi}(h;t)
=
\frac{\cB_{\pi}(\pi(ht)\phi_{\pi},\t{\pi}(t)\phi_{\t{\pi}})}
{\cB_{\pi}(\pi(\tau_F)\phi_{\pi}, \phi_{\t{\pi}})},\quad h\in\G(F).           
\]
Recall that $\tau_F$ is given by \eqref{E:tau_F}. When $t=1$, we simply denote 
$\Phi_{\pi}(h)$ for $\Phi_{\pi}(h;t)$.
\end{defn}

\begin{Remark}
Note that $\Phi_{\pi}(h;t)$ is independent of the choice of elements $\phi_{\pi}$ and 
$\phi_{\t{\pi}}$ in the one dimensional subspaces of $\pi$ and $\t{\pi}$ which consisting of
either weight $k$ elements or newforms, respectively. 
Moreover, it is also independent of $\cB_{\pi}$ and the models for which we used to realize $\pi$ 
and $\t{\pi}$.  
\end{Remark}

\subsection{A formulas of $\Phi_\pi(h,t)$: the archimedean case}
\label{SS:matrix coefficient: archimedean case}
Let $\pi$ be a (limit of) discrete series representation of $\G(\R)$ with minimal weight $k\geq 1$
and the central character ${\rm sgn}^k$. Note that $\pi\cong\t{\pi}$. 
Let $\psi$ be the additive character of $\R$ defined by 
$\psi(x)=e^{2\pi\sqrt{-1}x}$. Let $\cW(\pi,\psi)$ be the 
Whittaker model of $\pi$ with respect to $\psi$. 
Let $\cB_{\pi}:\cW(\pi,\psi)\times\cW(\pi,\psi)\rightarrow\C$ be the $\G(\R)$-invariant 
bilinear pairing given by 
\begin{equation}\label{E:bilinear pairing}
\cB_{\pi}(W, W')
=
\int_{\R^{\x}}W\left(\pMX t001\right)W'\left(\begin{pmatrix}-t&0\\0&1\end{pmatrix}\right)d^{\x}t,
\end{equation}
for $W, W'\in\cW(\pi,\psi)$. Here $d^{\x}t=|t|_\R^{-1}dt$, and $dt$ is 
the usual Lebesgue measure on $\R$.

Let $W_{\pi}\in\cW(\pi,\psi)$ be the weight $k$ element characterized by
\begin{equation}\label{E:Whittaker function for weight k}
W_{\pi}\left(\pMX a001\right)=a^{\frac{k}{2}} e^{-2\pi a}\cdot\bbI_{\R_+}(a),\quad a\in\R^{\x}.
\end{equation} 
For each $m\in\Z_{\geq 0}$, we put
\[
 W^m_{\pi}=\rho\left(V_+^m\right)W_{\pi}.
\]
Here $\rho$ denotes the right translation. In particular, we have $W^0_{\pi}=W_{\pi}$. We note that $W^m_{\pi}$ has weight $k+2m$.
The following recursive formula can be deduced from the proof of \cite[Lemma 5.6]{JLbook}
\begin{equation}\label{E:recursive formula}
W_{\pi}^{m+1}\left(\pMX a001\right)
=
2a\cdot\frac{d}{da}W_{\pi}^m\left(\pMX a001\right)+(k+2m-4\pi a)\cdot 
W_{\pi}^m\left(\pMX a001\right).
\end{equation}
\begin{lm}\label{L:whittaker function for weight k+2m}
We have
\[
W_{\pi}^m\left(\pMX a001\right)
=
2^m P^m_{\pi}(a) e^{-2\pi a}\cdot \bbI_{\R_+}(a),
\]
where $P^m_\pi$ is the polynomial given by
\[
P^m_{\pi}(a)
=
\sum_{j=0}^m (-4\pi)^j\begin{pmatrix} m\\j \end{pmatrix}
\frac{\Gamma(k+m)}{\Gamma(k+j)}\cdot a^{\frac{k}{2}+j}.
\]
\end{lm}
\begin{proof}
This follows from \eqref{E:Whittaker function for weight k}, \eqref{E:recursive formula} and the induction on $m$.
\end{proof}
\begin{lm}\label{L:matrix coefficient for weight k+2m}
Let $a\in\R^{\x}$ and $x\in\R$. Then 
\[
%\Phi_{\pi}^m(\begin{pmatrix} 1&x\\0&1 \end{pmatrix} \begin{pmatrix} a&0\\0&1 \end{pmatrix})
\cB_{\pi}\left(\rho\left(\pMX 1x01\pMX a001\right)W^m_{\pi}, W^m_{\pi}\right)
\]
is equal to
\[
2^{-k+2m}\pi^{-k}\Gamma(k+m)^2\,\bbI_{\R_-}(a)
\sum_{i,j=0}^m(-2)^{i+j}
\begin{pmatrix}m\\i\end{pmatrix}\begin{pmatrix}m\\j\end{pmatrix}
\frac{\Gamma(k+i+j)}{\Gamma(k+i)\Gamma(k+j)}
\frac{(-a)^{\frac{k}{2}+i}}{\left[(1-a)+\sqrt{-1}\,x\right]^{k+i+j}}.
\]
\end{lm}
\begin{proof}
By \eqref{E:bilinear pairing} and \lmref{L:whittaker function for weight k+2m} we have
\begin{align*}
&\cB_{\pi}\left(\rho\left(\pMX 1x01\pMX a001\right)W^m_{\pi}, W^m_{\pi}\right)\\
&=
\int_{\R^{\x}}W_{\pi}^m\left(\begin{pmatrix}at&0\\0&1\end{pmatrix}\right)
W_{\pi}^m\left(\begin{pmatrix}-t&0\\0&1\end{pmatrix}\right)\psi(xt)d^{\x}t\\
&=
2^{2m}\int_{\R^{\x}}
P^m_{\pi}(at)P^m_{\pi}(-t) e^{-2\pi\left\lbrace(1-a)+\sqrt{-1}\,x\right\rbrace(-t)}
\cdot\bbI_{\R_+}(at)\bbI_{\R_+}(-t)d^{\x}t\\
&=
2^{2m}\bbI_{\R_-}(a)\sum_{i,j=0}^m(-4\pi)^{i+j}
\begin{pmatrix}m\\i\end{pmatrix}\begin{pmatrix}m\\j\end{pmatrix}
\frac{\Gamma(k+m)^2}{\Gamma(k+i)\Gamma(k+j)}\cdot a^{\frac{k}{2}+i}\cdot I_{ij},
\end{align*}
where 
\begin{align*}
I_{ij}
&=
(-1)^{\frac{k}{2}+i}\int_0^{\infty}t^{k+i+j}
 e^{-2\pi\left[(1-a)+\sqrt{-1}\,x\right] t}d^{\x}t\\
&=
(-1)^{\frac{k}{2}+i}\left(2\pi\left[(1-a)+\sqrt{-1}\,x\right]\right)^{-(k+i+j)}
\Gamma(k+i+j). 
\end{align*}
This proves the lemma.  
\end{proof}
%As a corollary, we obtain the norm of $W^m_{\pi}$, which is 
%\[
%\cB_{\pi}\left(\rho\left(\tau_{\R}\right) W^m_{\pi},W^m_{\pi}\right)
%=
%\cB_{\pi}\left(\rho\left(\begin{pmatrix}-1&0\\0&1\end{pmatrix}\right) W^m_{\pi}, W^m_{\pi}\right).
%=
%\Phi_{\pi}^m(\begin{pmatrix}-1&0\\0&1\end{pmatrix}).
%\]
\begin{lm}\label{L:combintorial identity 1}
Let $N$ be a nonnegative integer. We have the following identity
\[
\sum_{i=0}^N (-1)^i\begin{pmatrix}N\\i\end{pmatrix}\frac{\Gamma(z+i)}{\Gamma(w+i)}
=
\frac{\Gamma(z)}{\Gamma(w-z)}\cdot\frac{\Gamma(w-z+N)}{\Gamma(w+N)},
\]
for every $z,w\in\C$.
\end{lm}
\begin{proof}
This is  \cite[Lemma 2.1]{Ikeda1998}.
\end{proof}
%Recall that the gamma function has a simple pole at every non-positive integer $-n$ with residue $\frac{(-1)^n}{\Gamma(n+1)}$.
\begin{lm}\label{L:norm for weight k+2m}
We have
\[
\cB_{\pi}\left(\rho\left(\begin{pmatrix}-1&0\\0&1\end{pmatrix}\right) W^m_{\pi}, W^m_{\pi}\right)
=
4^{-k+m}\pi^{-k}\Gamma(k+m)\Gamma(m+1).
\]
\end{lm}
\begin{proof}
By \lmref{L:matrix coefficient for weight k+2m} we have
\[
%\Phi_{\pi}^m(\begin{pmatrix}-1&0\\0&1\end{pmatrix})
\cB_{\pi}\left(\rho\left(\begin{pmatrix}-1&0\\0&1\end{pmatrix}\right) W^m_{\pi}, W^m_{\pi}\right)
=
4^{-k+m}\pi^{-k}\Gamma(k+m)^2
\sum_{i,j=0}^m(-1)^{i+j}
\begin{pmatrix}m\\i\end{pmatrix}\begin{pmatrix}m\\j\end{pmatrix}
\frac{\Gamma(k+i+j)}{\Gamma(k+i)\Gamma(k+j)}.
\]
Applying \lmref{L:combintorial identity 1}, we find that
\begin{align*}
\sum_{i,j=0}^m(-1)^{i+j}
\begin{pmatrix}m\\i\end{pmatrix}\begin{pmatrix}m\\j\end{pmatrix}
\frac{\Gamma(k+i+j)}{\Gamma(k+i)\Gamma(k+j)}
&=
\sum_{i=0}^m\begin{pmatrix}m\\i\end{pmatrix}\frac{(-1)^i}{\Gamma(k+i)}
\sum_{j=0}^m (-1)^j\begin{pmatrix}m\\j\end{pmatrix}\frac{\Gamma(k+i+j)}{\Gamma(k+j)}\\
&=
\sum_{i=0}^m (-1)^i\begin{pmatrix}m\\i\end{pmatrix}\frac{\Gamma(m-i)}{\Gamma(-i)\Gamma(k+m)}\\
&=
(-1)^m\frac{\Gamma(0)}{\Gamma(-m)\Gamma(k+m)}
=
\frac{\Gamma(m+1)}{\Gamma(k+m)}.
\end{align*}
This proves the lemma.
\end{proof}
Combining the above results, we obtain the following corollary.
\begin{cor}\label{C:mc for real with V_+ and tau}
Let $m\in\Z_{\geq 0}$, $x\in\R$ and $a\in\R^{\x}$. We have
\begin{itemize}
\item[(1)]
\begin{align*}
&\Phi_{\pi}\left(\pMX ax01;V^m_+\right)\\
=&
2^{k+2m}\,\frac{\Gamma(k+m)^2}{\Gamma(k)}\,\bbI_{\R_-}(a)
\sum_{i,j=0}^m(-2)^{i+j}
\begin{pmatrix}m\\i\end{pmatrix}\begin{pmatrix}m\\j\end{pmatrix}
\frac{\Gamma(k+i+j)}{\Gamma(k+i)\Gamma(k+j)} 
\frac{(-a)^{\frac{k}{2}+i}}{\left[(1-a)+\sqrt{-1}\,x\right]^{k+i+j}}.
\end{align*}
\item[(2)]
\[
\Phi_{\pi}\left(\pMX ax01;\tau_\R\right)
=
2^k\,
\frac{(-a)^{\frac{k}{2}}}{\left[(1-a)-\sqrt{-1}\,x\right]^{k}}\,\bbI_{\R_-}(a).
\]
\end{itemize}
\end{cor}

\subsection{A formula of $\Phi_\pi(h,t)$: the non-archimedean case}
Let $F$ be a non-archimdean local field. Let $B(F)$ be the subgroup of upper triangular matrices in $\G(F)$. Denote by ${\rm St}_F$ the Steinberg representation of $\G(F)$. Namely, ${\rm St}_F$ is the unique irreducible subrepresentation in the induced 
representation 
\[
{\rm Ind}_{B(F)}^{\G(F)}(|\cdot|_F^{1/2}\bt|\cdot|_F^{-1/2}).
\] 

\begin{lm}\label{L:macdonald formula}
Suppose that $\pi={\rm Ind}_{B(F)}^{\G(F)}\left(|\cdot|^{\lambda}_F\bt|\cdot|_F^{-\lambda}\right)$ 
is spherical. Let $\alpha=|\varpi|^\lambda_F$. Then for $n\in\Z$, we have
\[
\Phi_{\pi}\left(\begin{pmatrix}\varpi_F^n&0\\0&1\end{pmatrix}\right)
=
\frac{q_F^{-|n|/2}}{1+q_F^{-1}}
\left( 
\alpha^{|n|}\cdot\frac{1-\alpha^{-2}q_F^{-1}}{1-\alpha^{-2}}
+
\alpha^{-|n|}\cdot\frac{1-\alpha^2q_F^{-1}}{1-\alpha^2}
\right) 
\]  
\end{lm}
\begin{proof}This is Macdonald's formula. For example, see \cite[Theorem 4.6.6]{Bump1998}.
\end{proof}
\begin{lm}\label{L:mc for unramified special}
Suppose $\pi={\rm St}_F\ot\chi$, where $\chi$ is a unramified quadratic character of $F^{\x}$. 
Then for $n\in\Z$, we have 
\[
\Phi_{\pi}\left(\begin{pmatrix}\varpi_F^n&0\\0&1\end{pmatrix}\right)
=
\chi(\varpi_F^n)q_F^{-|n|}
\quad\text{and}\quad
\Phi_{\pi}\left(\pMX 0110\begin{pmatrix}\varpi_F^n&0\\0&1\end{pmatrix}\right)
=
-\chi(\varpi_F^n)q_F^{-|n-1|}.
\]
\end{lm}
\begin{proof}
For the ease of notation, we put $\varpi=\varpi_F$ and $q=q_F$. 
Let $\psi$ be an additive character of $F$ of order zero. Let $\cW(\pi,\psi)$ be the Whittaker model
of $\pi$ with respect to $\psi$. Since $\pi$ is self-dual, we have $\cW(\t{\pi},\psi)=\cW(\pi,\psi)$
by the uniqueness of the Whittaker model. Let $W_{\pi}\in\cW(\pi,\psi)$ be the newform with 
$W_{\pi}(1)=1$.
By \cite[Summary]{Schmidt2002} we have
\[
W_{\pi}\left(\pMX a001\right)
=
\chi(a)|a|_F\cdot\bbI_{\cO_F}(a),\quad a\in F^{\x}.
\]
Using \cite[Prop. 3.1.2]{Schmidt2002}, one can deduce that
\begin{equation}\label{E:action of Weyl element for unramified special}
W_{\pi}\left(\pMX a001\pMX0110\right)
=
-\chi(\varpi)\rho\left(\begin{pmatrix}-\varpi&0\\0&1\end{pmatrix}\right)
W_{\pi}\left(\pMX a001\right)
=
-q^{-1}\chi(a)|a|_F\cdot\bbI_{\varpi^{-1}\cO_F}(a).
\end{equation}

%Let $\psi$ be an additive character of $F$ with order zero. 
Let $\cB_{\pi}:\cW(\pi,\psi)\x\cW(\t{\pi},\psi)\rightarrow\C$ be the $\G(F)$-invariant bilinear 
pairing given by \eqref{E:bilinear pairing} with $\R^{\x}$ replaced by $F^{\x}$. The Haar measure 
$d^{\x}t$ on $F^{\x}$ is determined by ${\rm Vol}\left(\cO_F^{\x}, d^{\x}t\right)=1$.   
Let $n\in\Z$. 
\begin{align*}
\cB_{\pi}\left(\rho\left(\begin{pmatrix}\varpi^n&0\\0&1\end{pmatrix}\right) W_{\pi}, W_{\pi}\right)
&=
\int_{F^{\x}}W_{\pi}\left(\begin{pmatrix}\varpi^nt&0\\0&1\end{pmatrix}\right)
W_{\pi}\left(\begin{pmatrix}-t&0\\0&1\end{pmatrix}\right) d^{\x}t\\
&=
\chi(\varpi^n)q^{-n}\int_{\varpi^{\delta(n)}\cO_F}|t|^2d^{\x}t\\
&=
\chi(\varpi^n)q^{-n-2\delta(n)}\cdot\zeta_F(2),
\end{align*}
where
\[
\delta(n)
=
\begin{cases}
0\quad&\text{if $n\geq 0$},\\
-n\quad&\text{if $n<0$}.
\end{cases}
\]
In particular, we have
\begin{equation}\label{E:norm for unramified special}
\cB_{\pi}(W_{\pi}, W_{\pi})=\zeta_F(2),
\end{equation}
and hence
\[
\Phi_{\pi}\left(\begin{pmatrix}\varpi^n&0\\0&1\end{pmatrix}\right)
=
\chi(\varpi^n)q^{-n-2\delta(n)}=\chi(\varpi^n)q^{|n|}.
\]
Using \eqref{E:action of Weyl element for unramified special} and by similar calculations,
we find that 
\begin{align*}
\cB_{\pi}\left(\rho\left(\pMX 0110\begin{pmatrix}\varpi^n&0\\0&1\end{pmatrix}\right) 
W_{\pi}, W_{\pi}\right)
=
-\chi(\varpi^n)q^{-n-1-2\delta'(n)}\cdot\zeta_F(2),
\end{align*}
where
\[
\delta'(n)
=
\begin{cases}
-1\quad&\text{if $n\geq 1$},\\
-n\quad&\text{if $n<1$}.
\end{cases}
\]
Since $\cB_{\pi}(W_{\pi}, W_{\pi})=\zeta_F(2)$ by \eqref{E:norm for unramified special}, 
we obtain
\[
\Phi_{\pi}\left(\pMX 0110\begin{pmatrix}\varpi^n&0\\0&1\end{pmatrix}\right)
=
-\chi(\varpi^n)q^{-n-1-2\delta'(n)}
=
-\chi(\varpi^n)q^{-|n-1|}.
\]
This finishes the computation.
\end{proof}

\section{The calculation of local zeta integral (I)}
\label{S:local zeta integral: matrix algebra}
In this section, let $D={\rm M}_2(F)$. We compute the normalized local zeta integral $I^*(\itPi,\bft)$ in \defref{D: definition of zeta integral}.
\subsection{Haar measures}\label{SS:measure for matrix algebra}
If $F=\R$, let $dx$ be the usual Lebesgue measure on $\R$, and the Haar measure $d^{\x}x$ on 
$\R^{\x}$ is given by $|x|_\R^{-1}d^{\x}x$. The Haar measure $dh$ on $\G(\R)$ is given by
\[
dh=\frac{dz}{|z|_\R}\frac{dxdy}{|y|_\R^2}dk
\]
for 
$
h
= 
z
\begin{pmatrix}
1&x\\0&1
\end{pmatrix}
\begin{pmatrix}
y&0\\0&1
\end{pmatrix}k
$ 
with $x\in\R, y\in\R^{\x}, z\in\R_+^{\x}, k\in{\rm SO}(2)$, where $dx, dy, dz$ are the usual Lebesgue
measures and $dk$ is the Haar measure on ${\rm SO}(2)$ such that ${\rm Vol}({\rm SO}(2), dk)=1$. 

If $F$ is non-archimedean, let $dx$ be the Haar measure on $F$ so that the total volume of $\cO_F$ is equal to $1$ and let $d^{\x}x$ on $F^{\x}$ be the Haar measure on $F^\x$ so that $\cO^{\x}_F$ also has volume 
$1$. On $\G(F)$, we let $dh$ be the Haar measure determined by 
${\rm Vol}\left(\G(\cO_F),dh\right)=1$.

The measure on the quotient space $F^{\x}\backslash \G(F)$ is the unique 
quotient measure induced from the measure $dh$ on $\G(F)$ and the measure $d^{\x}x$ on $F^{\x}$.
\subsection{The archimedean case}
\label{SS:unbalanced case}
Let $\pi_j\,(j=1,2,3)$ be a (limit of) discrete series representation of $\G(\R)$ with minimal 
weight $k_j\geq 1$ and central character ${\rm sgn}^{k_j}$ such that
\[
2\,{\rm max}\stt{k_1,k_2,k_3}\geq k_1+k_2+k_3.
\]
We may assume $k_3={\rm max}\stt{k_1,k_2,k_3}$ and let $2m=k_3-k_1-k_2$ for some integer $m\geq 0$.
 
\begin{prop}\label{P:archimedean unbalanced case}
%In the case $D={\rm M}_2(\R)$, 
We have
\[
I^*\left(\itPi,\bft\right)=2^{k_1+k_2-k_3+1}.
\] 
\end{prop}
\begin{proof}Note that the $L$-factor given by 
\begin{align*}
L(s,\itPi,r)
=
&\zeta_{\C}(s+(k_3+k_2+k_1-3)/2))\zeta_{\C}(s+(k_3-k_2-k_1+1)/2)\\
&\x\zeta_{\C}(s+(k_3-k_2+k_1-1)/2)\zeta_{\C}(s+(k_3+k_2-k_1-1)/2)).
\end{align*}
We proceed to compute $I(\itPi,\bft)$. By definition
\begin{align*}
I\left(\itPi,\bft\right)
&=
\int_{\R^{\x}\backslash\G(\R)}
\Phi_{\pi_1}(h)\Phi_{\pi_2}\left(h;\t{V}^m_+\right)\Phi_{\pi_3}(h;\tau_{\R})dh\\
&=
\left(\frac{1}{8\pi}\right)^{2m}
\int_{\R^{\x}\backslash\G(\R)}
\Phi_{\pi_1}(h)\Phi_{\pi_2}\left(h;V^m_+\right)\Phi_{\pi_3}(h;\tau_{\R})dh.
\end{align*}
Put
\[
\Phi(h)
=
\Phi_{\pi_1}(h)\Phi_{\pi_2}\left(h;V^m_+\right)\Phi_{\pi_3}(h;\tau_{\R}),\quad h\in\G(\R).
\]
We now focus our attention to compute the following integral:
\[
I:=\int_{\R^{\x}\backslash\G(\R)}
\Phi(h)dh.
\]
Note that $\Phi(h)$ is right ${\rm SO}(2)$-invariant. Since the total volume of ${\rm SO}(2)$ is 1,
it follows that
\[
\int_{\R^{\x}\backslash\G(\R)}\Phi(h)dh
=
\int_{\R}\int_{\R_+}\left[\Phi\left(\begin{pmatrix}1&x\\0&1\end{pmatrix}
\begin{pmatrix}a&0\\0&1\end{pmatrix}\right)
+
\Phi\left(\begin{pmatrix}1&x\\0&1\end{pmatrix}\begin{pmatrix}-a&0\\0&1\end{pmatrix}\right)\right]
\,\frac{d^{\x}a}{a}dx, 
\]
by the Iwasawa decomposition.
Since $\Phi\left(\begin{pmatrix}1&x\\0&1\end{pmatrix}\begin{pmatrix}a&0\\0&1\end{pmatrix}\right)$ 
vanishes when $a\in\R_+$, we find that
\[
I
=
\int_{\R^{\x}\backslash\G(\R)}\Phi(h)dh
=
\int_{\R}\int_{\R_+}
\Phi\left(\begin{pmatrix}1&x\\0&1\end{pmatrix}\begin{pmatrix}-a&0\\0&1\end{pmatrix}\right)
\,\frac{d^{\x}a}{a}\,dx.
\]
By \corref{C:mc for real with V_+ and tau}, we have
$\Phi\left(\begin{pmatrix}1&x\\0&1\end{pmatrix}\begin{pmatrix}a&0\\0&1\end{pmatrix}\right)$ 
is equal to $2^{2k_3}\frac{\Gamma(k_2+m)^2}{\Gamma(k_2)}$ times
\begin{equation}\label{E:equation for Phi for real}
\bbI_{\R_-}(a)
\sum_{i,j=0}^m(-2)^{i+j}\begin{pmatrix}m\\i\end{pmatrix}\begin{pmatrix}m\\j\end{pmatrix}
\frac{\Gamma(k_2+i+j)}{\Gamma(k_2+i)\Gamma(k_2+j)}\cdot
\frac{(-a)^{k_3-m+i}}{[(1-a)-\sqrt{-1}\,x]^{k_3}[(1-a)+\sqrt{-1}\,x]^{k_3-2m+i+j}}.
\end{equation}
By \eqref{E:equation for Phi for real} we have
\[
I
=
2^{2k_3}\frac{\Gamma(k_2+m)^2}{\Gamma(k_2)}
\sum_{i,j=0}^m(-2)^{i+j}\begin{pmatrix}m\\i\end{pmatrix}\begin{pmatrix}m\\j\end{pmatrix}
\frac{\Gamma(k_2+i+j)}{\Gamma(k_2+i)\Gamma(k_2+j)}\cdot I_{i,j},
\]
where for $0\leq i,j\leq m$, 
\begin{align*}
I_{i,j}
&:=
\int_{\R}\int_{\R_+}\frac{a^{k_3-m+i-1}}{[(1+a)-\sqrt{-1}\,x]^{k_3}[(1+a)+\sqrt{-1}\,x]^{k_3-2m+i+j}}
\,d^{\x}a\,dx\\
&=
\left(\int_{\R_+}\frac{a^{k_3-m+i-1}}{(1+a)^{2k_3-2m+i+j-1}}\,d^{\x}a\right)
\left(\int_{\R}\frac{dx}{[1+\sqrt{-1}\,x]^{k_3-2m+i+j}[1-\sqrt{-1}\,x]^{k_3}}\right)\\
&=
2^{2-(2k_3-2m+i+j)}\pi
\frac{\Gamma(k_3-m+i-1)\Gamma(k_3-m+j)}{\Gamma(k_3-2m+i+j)\Gamma(k_3)}.
\end{align*}
The last equality follows from the following lemma.
\begin{lm}\label{L:integral formula 1}
For $|arg\,z|<\pi$, $0<Re(\beta)<Re(\alpha)$, we have
\[
\int_{\R_+}\frac{t^{\beta}}{(t+z)^{\alpha}}d^{\x}t
=
z^{\beta-\alpha}\cdot
\frac{\Gamma(\alpha-\beta)\Gamma(\beta)}{\Gamma(\alpha)}.
\]
%\begin{lm}\label{L:integral formula 2}
For $Re(\alpha+\beta)>1$, we have
\[
\int_{\R}\frac{dx}{(1+\sqrt{-1}\,x)^{\alpha}(1-\sqrt{-1}\,x)^{\beta}}
=
2^{2-\alpha-\beta}\cdot\pi\cdot
\frac{\Gamma(\alpha+\beta-1)}{\Gamma(\alpha)\Gamma(\beta)}.
\]
\end{lm}
\begin{proof}
These are  \cite[Lemma 2.4 and 2.5]{Ikeda1998}
\end{proof}
Thus we obtain
\[
I
=
2^{2+2m}\pi\frac{\Gamma(k_2+m)^2}{\Gamma(k_2)\Gamma(k_3)}
\sum_{i,j=0}^m(-1)^{i+j}\begin{pmatrix}m\\i\end{pmatrix}\begin{pmatrix}m\\j\end{pmatrix}
\frac{\Gamma(k_2+i+j)}{\Gamma(k_2+i)\Gamma(k_2+j)}\cdot
\frac{\Gamma(k_3-m+i-1)\Gamma(k_3-m+j)}{\Gamma(k_3-2m+i+j)}.
\]
To simply the above expression of $I$, we need one more combinatorial identity from \cite[Lemma 3]{Orloff1987}.
\begin{lm}\label{L:combintorial identity 2}
Let $N\in\Z_{\geq 0}$ and $t,\alpha,\beta\in\C$. Then
\begin{align*}
\Gamma(\alpha+N)&\sum_{i=0}^{N}(-1)^i\begin{pmatrix}N\\i\end{pmatrix}
\frac{\Gamma(t+i)}{\Gamma(\alpha+i)}\cdot
\frac{\Gamma(t+\beta+\alpha+N-1+i)}{\Gamma(2t+\beta+i)}\\
&=(-1)^N\frac{\Gamma(t)\Gamma(t+\beta+\alpha+N-1)}{\Gamma(2t+\beta+N)}\cdot
\frac{\Gamma(t+\beta+N)}{\Gamma(t+\beta)}\cdot
\frac{\Gamma(t-\alpha+1)}{\Gamma(t-\alpha-N+1)}.
\end{align*}
\end{lm} 
Now we write
\[
I=
2^{2+2m}\pi\frac{\Gamma(k_2+m)}{\Gamma(k_1)\Gamma(k_2)}
\sum_{j=0}^m(-1)^j \begin{pmatrix}m\\j\end{pmatrix}\frac{\Gamma(k_3-m+j)}{\Gamma(k_2+j)}\cdot I', 
\]
where
\[
I'
=
\Gamma(k_2+m)\sum_{i=0}^m(-1)^i\begin{pmatrix}m\\i\end{pmatrix}
\frac{\Gamma(k_2+j+i)}{\Gamma(k_2+i)}\cdot
\frac{\Gamma(k_3-m-1+i)}{\Gamma(k_3-2m+j+i)}. 
\]
Applying \lmref{L:combintorial identity 2} to $I'$ with $t=k_2+j$, $\alpha=k_2$ and 
$\beta=k_3-2m-2k_2-j$, we find that
\[
I'
=
(-1)^m\cdot\frac{\Gamma(k_2+j)\Gamma(k_3-m-1)}{\Gamma(k_3-m+j)}\cdot
\frac{\Gamma(k_3-k_2-m)}{\Gamma(k_3-k_2-2m)}\cdot
\frac{\Gamma(j+1)}{\Gamma(j-m+1)}.
\]It follows that
\[
I
= 
(-1)^m2^{2+2m}\pi
\frac{\Gamma(k_3-m-1)\Gamma(k_2+m)\Gamma(k_1+m)}{\Gamma(k_3)\Gamma(k_2)\Gamma(k_1)}
\sum_{j=0}^m(-1)^j\begin{pmatrix}m\\j\end{pmatrix}\frac{\Gamma(1+j)}{\Gamma(1-m+j)}.
\]
Applying \lmref{L:combintorial identity 1}, we obtain
\[
I
=
2^{2+2m}\pi
\frac{\Gamma(k_3-m-1)\Gamma(k_2+m)\Gamma(k_1+m)\Gamma(m+1)}{\Gamma(k_1)\Gamma(k_2)\Gamma(k_3)}.
\]
Therefore we find that
\[
I\left(\itPi,\bft\right)
=
\left(\frac{1}{8\pi}\right)^{2m}I
=
2^{2-4m}\pi^{1-2m}
\frac{\Gamma(k_3-m-1)\Gamma(k_2+m)\Gamma(k_1+m)\Gamma(m+1)}{\Gamma(k_1)\Gamma(k_2)\Gamma(k_3)},
\]
and the proposition follows.
\end{proof}

We deduce a consequence from \propref{P:archimedean unbalanced case}. Let $m_1,m_2$ be two non-negative
integers such that $m_1+m_2=m$. Put
\[
\bft_{m_1,m_2}
=
\left(\t{V}^{m_1}_+\ot\t{V}^{m_2}_+\ot 1, (1,1,\tau_{\R})\right)
\in
\mathfrak{U}_{E}\x {\rm O}(2,E).
\]
Then our original element $\bft$ is $\bft_{0,m}$. Put
\[
I^*\left(\itPi;\bft_{m_1,m_2}\right)
=
\frac{L(1,\itPi,{\rm Ad})}{\zeta_{\R}(2)^2 L(1/2,\itPi,r)}\cdot 
I\left(\itPi,\bft_{m_1,m_2}\right),
\]
where
\[
I\left(\itPi,\bft_{m_1,m_2}\right)
=
\int_{\R^{\x}\backslash\G(\R)}
\Phi_{\pi_1}\left(h;\t{V}^{m_1}_+\right)\Phi_{\pi_2}\left(h;\t{V}^{m_2}_+\right)
\Phi_{\pi_3}(h;\tau_{\R})dh
\]
Then $I^*\left(\itPi,\bft\right)$ in \defref{D: definition of zeta integral} 
is nothing but $I^*\left(\itPi,\bft_{0,m}\right)$.
\begin{cor}
Notation is as above. We have
\[
I^*\left(\itPi,\bft_{m_1,m_2}\right)=I^*\left(\itPi,\bft\right)
\]
for every non-negative integers $m_1,m_2$ such that $m_1+m_2=m$.
\end{cor}
\begin{proof}
This is in fact an easy consequence form the multiplicity one result of local trilinear forms,
\propref{P:archimedean unbalanced case} together with the local Rankin-Selberg integral.
More precisely, let $\mu_2=|\cdot|_\R^{(k_2-1)/2}$ and $\nu_2=|\cdot|_\R^{(1-k_2)/2}{\rm sgn}^{k_2}$
be two characters of $\R^{\x}$. Then $\pi_2$ can be realized as the unique irreducible 
subrepresentation of ${\rm Ind}_{B(\R)}^{\G(\R)}(\mu_2\bt\nu_2)$ which we denote by 
${\rm Ind}_{B(\R)}^{\G(\R)}(\mu_2\bt\nu_2)_0$. For every non-negative integer $n$, we let
$f^{n}_{\pi_2}\in{\rm Ind}_{B(\R)}^{\G(\R)}(\mu_2\bt\nu_2)_0$ be the element characterized
by requiring
\[
f^{n}_{\pi_2}
\left(\begin{pmatrix}
{\rm cos}\theta&{\rm sin}\theta\\-{\rm sin}\theta&{\rm cos}\theta
\end{pmatrix}
\right)
=
e^{i(k_2+2n)\theta}.
\]
We have the following relation, which can be found in \cite[Lemma 5.6 (iii)]{JLbook}
\[
\rho\left(\t{V}_+\right)f_{\pi_2}^{n}
=
2(k_2+n)f^{n+1}_{\pi_2}.
\]
Inductively we find that 
\[
\rho\left(\t{V}^{\ell}_+\right)f_{\pi_2}^{n}
=
c(\pi_2,n,\ell)f^{n+\ell}_{\pi_2},
\]
where
\begin{equation}\label{E:raising constant}
c(\pi_2,n,\ell)
=
2^{\ell}\frac{\Gamma(k_2+n+\ell)}{\Gamma(k_2+n)},
\end{equation}
for every $\ell\geq 0$.

Let $\Psi:\cW(\pi_1,\psi)\bt{\rm Ind}_{B(\R)}^{\G(\R)}(\mu_2\bt\nu_2)_0\bt\cW(\pi_3,\psi)\to\C$
be the local Rankin-Selberg integral defined by 
\[
\Psi(W_1\ot f_2\ot W_3)
=
\int_{\R^{\x}N(\R)\backslash\G(\R)}W_1(\tau_{\R}g)W_3(g)f_2(g)dg,
\] 
for $W_1\in\cW(\pi_1,\psi)$, $f_2\in{\rm Ind}_{B(\R)}^{\G(\R)}(\mu_2\bt\nu_2)_0$ and
$W_3\in\cW(\pi_3,\psi)$. Here
\[
N
=
\stt{\pMX 1*01\in\G}.
\]
One check easily that this integral converges absolutely and
certainly it defines a $\G(\R)$-invariant trilinear form. From the multiplicity one result of
such trilinear form and the fact that $I^*(\itPi,\bft)\neq 0$, one can deduce that following
equality easily
\begin{equation}\label{E:ichino integral and rankin-selberg integral}
\frac{I^*\left(\itPi,\bft_{m_1,m_2}\right)}{I^*\left(\itPi,\bft\right)}
=
\left(\frac{c(\pi_2,m_1,m_2)}{c(\pi_2,0,m)}\right)^2
\left(
\frac{\Psi\left(W_{\pi_1}^{m_1}\ot f_{\pi_2}^{m_2}\ot\rho(\tau_{\R})W_{\pi_3}\right)}
{\Psi\left(W_{\pi_1}\ot f_{\pi_2}^{m}\ot\rho(\tau_{\R})W_{\pi_3}\right)}
\right)^2.
\end{equation} 
Recall that $W_{\pi_1}^n=\rho\left(\t{V}^n_+\right)W_{\pi_1}$ for every $n\geq 0$. 
Our task now is to compute the ratio of these two Rankin-Selberg integrals.
Since we can let $m_1,m_2$ vary, it suffices to compute the numerator.
Applying \lmref{L:whittaker function for weight k+2m} and \lmref{L:combintorial identity 1}, we find that the numerator is
\begin{align*}
&\Psi\left(W_{\pi_1}^{m_1}\ot f_{\pi_2}^{m_2}\ot\rho(\tau_{\R})W_{\pi_3}\right)\\
&=
\int_{{\rm SO}(2)}\int_{\R^{\x}}
W^{m_1}_{\pi_1}\left(\begin{pmatrix}-a&0\\0&1\end{pmatrix}k\right)
W_{\pi_3}\left(\begin{pmatrix}-a&0\\0&1\end{pmatrix}k\right)
f^{m_2}_{\pi_2}\left(\begin{pmatrix}a&0\\0&1\end{pmatrix}k\right)
\frac{d^{\x}a}{|a|_\R}dk\\
&=
\int_{\R^{\x}}
W^{m_1}_{\pi_1}\left(\begin{pmatrix}-a&0\\0&1\end{pmatrix}\right)
W_{\pi_3}\left(\begin{pmatrix}-a&0\\0&1\end{pmatrix}\right)|a|_\R^{\frac{k_2}{2}-1}d^{\x}a\\
&=
2^{k_1+k_3+m_1}\sum_{j=0}^{m_1}(-4\pi)^j\frac{\Gamma(k_1+m_1)}{\Gamma(k_1+j)}
\begin{pmatrix}m_1\\j\end{pmatrix}\int_0^\infty
a^{\frac{k_1+k_2+k_3}{2}+j-1}e^{-4\pi a}d^{\x}a\\
&= 
2^{k_1+k_3+m_1}(4\pi)^{1-\frac{k_1+k_2+k_3}{2}}\Gamma(k_1+m_1)
\sum_{j=0}^{m_1}(-1)^j\begin{pmatrix}m_1\\j\end{pmatrix}
\frac{\Gamma\left(\frac{k_1+k_2+k_3}{2}+j-1\right)}{\Gamma(k_1+j)}\\
&=
(-1)^{m_1}2^{k_1+k_3+m_1}(4\pi)^{1-\frac{k_1+k_2+k_3}{2}}
\frac{\Gamma\left(\frac{k_1+k_2+k_3}{2}-1\right)\Gamma(k_2+m_1+m_2)}{\Gamma(k_2+m_2)}.
\end{align*}
By letting $m_1=0$ and $m_2=m$, we obtain the value of denominator. Combining with equation
\eqref{E:raising constant}, we find that the right hand side of the equation 
\eqref{E:ichino integral and rankin-selberg integral} is equal to $1$. The corollary 
follows.
\end{proof}

\subsection{The non-archimedean case}
Let $F$ be a non-archimedean local field.  
Write $\varpi=\varpi_F$ and $q=q_F$ for simplicity. Recall that we have assumed  
\begin{equation}\label{E:nonzero hom}
{\rm Hom}_{\G(F)}(\itPi,\C)\neq\stt{0}.
\end{equation}
According to the results of Prasad \cite{Prasad1990} and \cite{Prasad1992} and our assumption on
$\itPi$, \eqref{E:nonzero hom} holds for following cases.
(i) Suppose $E=F\x F\x F$ so that $\itPi=\pi_1\bt\pi_2\bt\pi_3$. Then (i-a) one of 
$\pi_1, \pi_2, \pi_3$ is spherical; (i-b) $\pi_j={\rm St}_F\ot\chi_j$ are unramified
special representations for $j=1,2,3$ with $\chi_1\chi_2\chi_3(\varpi)=-1$.
(ii) Suppose $E=K\x F$ so that $\itPi=\pi'\bt\pi$. Then (ii-a) $\pi$ is spherical; (ii-b) 
$\pi'$ is spherical, $\pi={\rm St}_F\ot\chi$ is a unramified special 
representation, $K/F$ is ramified and $\chi(\varpi)=-1$; 
(ii-c) $\pi'={\rm St}_K\ot\chi'$, $\pi={\rm St}_F\ot\chi$ are unramified special representations, 
$K/F$ is ramified or $K/F$ is unramified and $\chi'\chi(\varpi)=1$. (iii) Suppose $E$ is a field.
Then (iii-a) $\itPi$ is spherical; (iii-b) $\itPi={\rm St}_E\ot\chi$ is a unramified special
representation with $\chi(\varpi)=-1$.

We say that $E$ is unramified over $F$ if either $E=F\x F\x F$,
or $E=\cF\x F$, where $\cF$ is the unramified extension over $F$, or $E$ is the unramified 
cubic extension over $F$. The evaluation of $I^*(\itPi,\bft)$ has been carried out in the following cases.

\begin{prop}\label{P:known results}
\noindent 
\begin{itemize}
\item[(1)]
Suppose $E$ is unramified over $F$ and $\itPi$ is spherical. Then we have
\[
I^*\left(\itPi,\bft\right)=1.
\]
\item[(2)]
Suppose $E=F\x F\x F$ and $\pi_j={\rm St}_F\ot \chi_j$, where $\chi_j$ are unramified quadratic
characters of $F^{\x}$ for $j=1,2,3$. Then we have
\[
I^*\left(\itPi,\bft\right)
=
2q^{-1}(1+q^{-1}).
\]
\item[(3)]
Suppose $E=F\x F\x F$ and one of $\pi_j\,(j=1,2,3)$ is spherical and the other two are unramified
special. Then we have
\[
I^*\left(\itPi,\bft\right)
=
q^{-1}.
\]
\end{itemize}
Here $I^*\left(\itPi,\bft\right)$ is defined in \subsecref{SS:local period integral}.
\end{prop}

\begin{proof}
Part $(1)$ is \cite[Lemma 2.2]{Ichino2008}, $(2)$ is in \cite[Section 7]{IchinoIkeda2010} and
$(3)$ is a result of \cite[Lemma 4.4]{Nelson2011}.
\end{proof}

We proceed to compute $I^*(\itPi,\bft)$ in the remaining cases. For $\Phi\in L^1(F^{\x}\backslash\G(F))$ such that $\Phi(khk')=\Phi(h)$ for every $h\in\G(F)$
and $k,k'\in K_0(\varpi)$, where
\[
K_0(\varpi)
=
\stt{
\begin{pmatrix}
a&b\\c&d
\end{pmatrix}
\in\G(\cO_F)
\mid c\in\varpi\cO_F,
}
\]
we have the integration formula  
\begin{equation}\label{E:integral formula for iwahori}
\int_{F^{\x}\backslash\G(F)}\Phi(h)dh
=
(1+q)^{-1}
\left\lbrace 
\sum_{n\in\Z}
\Phi \left(
\begin{pmatrix}
\varpi^n&0\\0&1
\end{pmatrix}
\right )q^{|n|}
+
\sum_{n\in\Z}
\Phi \left(
\pMX 0110
\begin{pmatrix}
\varpi^n&0\\0&1
\end{pmatrix}
\right )q^{|n-1|}
\right\rbrace
\end{equation}    
(\cf \cite[Section 7]{IchinoIkeda2010}).

\begin{prop}\label{P:(1,1,1)one unramified special,two spherical}
Let $E=F\x F\x F$. Suppose one of $\pi_j$ is unramified special and the other two are 
spherical. Then we have
\[
I^*\left(\itPi,\bft\right)
=
q^{-1}(1+q^{-1})^{-1}.
\]
\end{prop}
\begin{proof}In this case, the $L$-factor is given by
\begin{align*}
L(s,\Pi,r)
&=
(1-\chi(\varpi)\alpha\beta q^{-s-1/2})^{-1}
(1-\chi(\varpi)\alpha\beta^{-1}q^{-s-1/2})^{-1}\\
&\x(1-\chi(\varpi)\alpha^{-1}\beta q^{-s-1/2})^{-1}
(1-\chi(\varpi)\alpha^{-1}\beta^{-1}q^{-s-1/2})^{-1}.
\end{align*}
We continue to compute $I(\itPi,\bft)$. Assume $\pi_1={\rm St}_F\ot\chi$ for some unramified quadratic character $\chi$ of $F^{\x}$, 
and
\[
\pi_j
=
{\rm Ind}_{B(F)}^{\G(F)}\left(|\cdot|^{\lambda_j}_F\bt|\cdot|_F^{-\lambda_j}\right)
\]
for $j=2,3$. Let $\alpha=|\varpi|^{\lambda_2}_F$ and $\beta=|\varpi|^{\lambda_3}_F$. Then we have
\[
I\left(\itPi,\bft\right)
=
\int_{F^{\x}\backslash\G(F)}
\Phi(h)dh,
\]
where
\[
\Phi(h)
=
\Phi_{\pi_1}(h)\Phi_{\pi_2}(h)
\Phi_{\pi_3}
\left( 
h;\begin{pmatrix}\varpi^{-1}&0\\0&1\end{pmatrix}
\right),\quad h\in\G(F).
\]
%Let $\alpha=|\varpi|^{\lambda_2}_F$ and $\beta=|\varpi|^{\lambda_3}_F$ and put 
%\[
%A_1
%=
%\frac{1}{1+q^{-1}}\cdot\frac{1-\alpha^{-2}q^{-1}}{1-\alpha^{-2}};
%\quad
%A_2
%=
%\frac{1}{1+q^{-1}}\cdot\frac{1-\alpha^{2}q^{-1}}{1-\alpha^{2}},
%\]
%and
%\[
%B_1
%=
%\frac{1}{1+q^{-1}}\cdot\frac{1-\beta^{-2} q^{-1}}{1-\beta^{-2}};
%\quad
%B_2
%=
%\frac{1}{1+q^{-1}}\cdot\frac{1-\beta^{2} q^{-1}}{1-\beta^{2}}.
%\]
By \eqref{E:integral formula for iwahori}, \lmref{L:macdonald formula} and
\lmref{L:mc for unramified special}, we find that
\begin{align*}
&I(\itPi,\bft)\\
&=
(1+q)^{-1}
\left\lbrace 
\sum_{n=-\infty}^\infty
\Phi
\left( 
\begin{pmatrix}\varpi^n&0\\0&1\end{pmatrix}
\right)q^{|n|}
+
\sum_{n=-\infty}^\infty
\Phi
\left( 
\begin{pmatrix}0&1\\1&0\end{pmatrix}\begin{pmatrix}\varpi^n&0\\0&1\end{pmatrix}
\right)q^{|n-1|}
\right\rbrace\\
&=
(1+q)^{-1}\cdot\frac{(1-q^{-1})}{(1+q^{-1})}\cdot
\frac{(1-\alpha^2 q^{-1})(1-\alpha^{-2}q^{-1})(1-\beta^2 q^{-1})(1-\beta^{-2}q^{-1})}
     {(1-\chi(\varpi)\alpha\beta q^{-1})(1-\chi(\varpi)\alpha\beta^{-1}q^{-1})
      (1-\chi(\varpi)\alpha^{-1}\beta q^{-1})(1-\chi(\varpi)\alpha^{-1}\beta^{-1}q^{-1})}.
%A_1B_1
%\frac{(1-\alpha^2 q^{-1})(1-\beta^2 q^{-1})}{(1-\chi(\varpi)\alpha\beta q^{-1})}
%+
%A_1B_2
%\frac{(1-\alpha^2 q^{-1})(1-\beta^{-2} q^{-1})}{(1-\chi(\varpi)\alpha\beta^{-1}q^{-1})}\\
%&+
%A_2B_1
%\frac{(1-\alpha^{-2}q^{-1})(1-\beta^2 q^{-1})}{(1-\chi(\varpi)\alpha^{-1}\beta q^{-1})}
%+
%A_2B_2
%\frac{(1-\alpha^{-2}q^{-1})(1-\beta^{-2}q^{-1})}{(1-\chi(\varpi)\alpha^{-1}\beta^{-1}q^{-1})}\\
%&=
%\frac{(1-\alpha^2 q^{-1})(1-\alpha^{-2}q^{-1})(1-\beta^2 q^{-1})(1-\beta^{-2}q^{-1})}
%     {(1+q^{-1})^2(1-\alpha^2)(1-\beta^2)}I,
\end{align*}
%where
%\begin{align*}
%I
%=&
%\frac{\alpha^2\beta^2}{1-\chi(\varpi)\alpha\beta q^{-1}}
%-
%\frac{\alpha^2}{1-\chi(\varpi)\alpha\beta^{-1}q^{-1}}
%-
%\frac{\beta^2}{1-\chi(\varpi)\alpha^{-1}\beta q^{-1}}
%+
%\frac{1}{1-\chi(\varpi)\alpha^{-1}\beta^{-1}q^{-1}}\\
%=&
%\frac{(1-q^{-2})(1-\alpha^2)(1-\beta^2)}
%     {(1-\chi(\varpi)\alpha\beta q^{-1})(1-\chi(\varpi)\alpha\beta^{-1}q^{-1})
%     (1-\chi(\varpi)\alpha^{-1}\beta q^{-1})(1-\chi(\varpi)\alpha^{-1}\beta^{-1}q^{-1})}.
%\end{align*}
%It follows that $I(\itPi,\bft)$ is equal to
%\[
%(1+q)^{-1}\cdot\frac{(1-q^{-1})}{(1+q^{-1})}\cdot
%\frac{(1-\alpha^2 q^{-1})(1-\alpha^{-2}q^{-1})(1-\beta^2 q^{-1})(1-\beta^{-2}q^{-1})}
%     {(1-\chi(\varpi)\alpha\beta q^{-1})(1-\chi(\varpi)\alpha\beta^{-1}q^{-1})
%      (1-\chi(\varpi)\alpha^{-1}\beta q^{-1})(1-\chi(\varpi)\alpha^{-1}\beta^{-1}q^{-1})}.
%\]

This completes the proof.
\end{proof}

\begin{prop}\label{P:(2,1) with 1 spherical}
Let $E=K\x F$ and $\pi$ be a spherical representation of $\G(F)$.
\begin{itemize}
\item[(1)]
If $K$ is ramified over $F$ and $\pi'$ is spherical, then we have
\[
I^*\left(\itPi,\bft\right)
=
1.
\] 
\item[(2)]
If $\pi'$ is unramified special, then we have
\[
I^*\left(\itPi,\bft\right)
=
\begin{cases}
q^{-1}(1+q^{-1})^{-2}(1+q^{-2})\quad&\text{if $\cF$ is unramified over $F$},\\
q^{-1}(1+q^{-1})^{-1}\quad&\text{if $\cF$ is ramified over $F$}.
\end{cases}
\]
\end{itemize}
\end{prop}

\begin{proof}
Let
\[
\pi
=
{\rm Ind}_{B(F)}^{\G(F)}\left(|\cdot|^{\lambda}_{F}\bt|\cdot|^{-\lambda}_{F}\right),\quad
\beta=|\varpi|_{F}^\lambda.
\]
We begin with $(1)$. Let
\[
\pi'
=
{\rm Ind}_{B(\cF)}^{\G(\cF)}\left(|\cdot|^{\lambda'}_{\cF}\bt|\cdot|^{-\lambda'}_{\cF}\right),
\quad 
\alpha=|\varpi_{\cF}|_{\cF}^{\lambda'}.
\]
%Put
%\[
%A_1
%=
%\frac{1}{1+q^{-1}}\cdot\frac{1-\alpha^{-2}q^{-1}}{1-\alpha^{-2}};
%\quad
%A_2
%=
%\frac{1}{1+q^{-1}}\cdot\frac{1-\alpha^{2}q^{-1}}{1-\alpha^{2}},
%\]
%and
%\[
%B_1
%=
%\frac{1}{1+q^{-1}}\cdot\frac{1-\beta^{-2} q^{-1}}{1-\beta^{-2}};
%\quad
%B_2
%=
%\frac{1}{1+q^{-1}}\cdot\frac{1-\beta^{2} q^{-1}}{1-\beta^{2}}.
%\]
For a non-negative integer $n$, let $X_n$ be the image of 
\[
\G(\cO_F)
\begin{pmatrix}
\varpi^n&0\\0&1
\end{pmatrix}
\G(\cO_F)
\]
in $F^{\x}\backslash\G(F)$. Note that
\[
\vol(X_n,dh)
=
\begin{cases}
1\quad&\text{if $n=0$},\\
q^n(1+q^{-1})\quad&\text{if $n\geq 1$}.
\end{cases}
\]
By \lmref{L:macdonald formula}, we have
\begin{align*}
I\left(\itPi,\bft\right)
&=
\sum_{n=0}^\infty
\Phi_{\pi'}
\left( 
\begin{pmatrix}
\varpi^n&0\\0&1
\end{pmatrix}
\right)
\Phi_\pi
\left( 
\begin{pmatrix}
\varpi^n&0\\0&1
\end{pmatrix}
\right) 
\vol(X_n,dh)\\
%&=
%A_1B_1\frac{(1+\alpha^2\beta q^{-3/2})}{(1-\alpha^2\beta q^{-1/2})}
%+
%A_2B_1\frac{(1+\alpha^{-2}\beta q^{-3/2})}{(1-\alpha^{-2}\beta q^{-1/2})}\\
%&+
%A_1B_2\frac{(1+\alpha^2\beta^{-1} q^{-3/2})}{(1-\alpha^2\beta^{-1} q^{-1/2})} 
%+
%A_2B_2\frac{(1+\alpha^{-2}\beta^{-1} q^{-3/2})}{(1-\alpha^{-2}\beta^{-1} q^{-1/2})}\\
&=
\frac
{(1-\alpha^2 q^{-1})(1-\alpha^{-2} q^{-1})(1+\beta q^{-1/2})(1+\beta^{-1} q^{-1/2})}
{(1-\alpha^2\beta q^{-1/2})(1-\alpha^{-2}\beta q^{-1/2})(1-\alpha^2\beta^{-1}q^{-1/2})
(1-\alpha^{-2}\beta^{-1}q^{-1/2})}.
\end{align*}
Recall that the $L$-factor is given by
\begin{align*}
L(s,\itPi,r)&= (1-\alpha\beta p^{-s})^{-1}(1-\alpha\beta^{-1}p^{-s})^{-1}(1-\beta p^{-s})^{-1}\\
&\x (1-\beta^{-1}p^{-s})^{-1}(1-\alpha^{-1}\beta p^{-s})^{-1}(1-\alpha^{-1}\beta^{-1}p^{-s})^{-1}.
\end{align*}
This shows $(1)$.

Now we consider $(2)$.
Let $\pi'={\rm St}_{\cF}\ot\chi'$ for some unramified quadratic character 
$\chi'$ of $\cF^{\x}$. Suppose $\cF$ is unramified over $F$. By definition
\[
I\left(\itPi,\bft\right)
=
\int_{F^{\x}\backslash\G(F)}
\Phi_{\pi'}(h)\Phi_{\pi}(h)dh.
\]
Applying \eqref{E:integral formula for iwahori}, \lmref{L:macdonald formula} and
\lmref{L:mc for unramified special}, 
\begin{align*}
I\left(\itPi,\bft\right)
&=
(1+q)^{-1}
\sum_{n=-\infty}^\infty
\chi'(\varpi)^n
\Phi_{\pi}
\left( 
\begin{pmatrix}
\varpi^n&0\\0&1
\end{pmatrix}
\right) 
\left\lbrace 
q^{-|n|}-q^{-|n-1|}
\right\rbrace\\
%&=
%q^{-1}\frac{(1-q^{-1})}{(1+q^{-1})^2}
%\left\lbrace 
%\frac{(1-\alpha^{-2}q^{-1})(1-\chi'(\varpi)\alpha q^{-1/2})}
%     {(1-\alpha^{-2})(1-\chi'(\varpi)\alpha q^{-3/2})}
%+
%\frac{(1-\alpha^2 q^{-1})(1-\chi'(\varpi)\alpha^{-1}q^{-1/2})}
%{(1-\alpha^2)(1-\chi'(\varpi)\alpha^{-1}q^{-3/2})}
%\right\rbrace\\
&=
\frac{(1-q^{-1})(1+q^{-2})}{(1+q^{-1})}
\cdot 
\frac{(1-\chi'(\varpi)\alpha q^{-1/2})(1-\chi'(\varpi)\alpha^{-1}q^{-1/2})}
     {(1-\chi'(\varpi)\alpha q^{-3/2})(1-\chi'(\varpi)\alpha^{-1}q^{-3/2})}.
\end{align*}
Suppose $\cF$ is ramified over $F$. Similar calculations shows
\[
I\left(\itPi,\bft\right)
=
q^{-1}\frac{(1-q^{-1})}{(1+q^{-1})}\cdot
\frac{(1-\alpha^2 q^{-1})((1-\alpha^{-2} q^{-1}))}
{(1-\alpha^2 q^{-3/2})(1-\alpha^{-2} q^{-3/2})}.
\]
Finally, if $\cF/F$ is unramified, we have
\begin{align*}
L(s,\itPi,r)
&=
(1+\chi'(\varpi)\alpha q^{-s})^{-1}
(1-\chi'(\varpi)\alpha q^{-s-1})^{-1}\\
&\times(1+\chi'(\varpi)\alpha^{-1}q^{-s})^{-1}
(1-\chi'(\varpi)\alpha^{-1}q^{-s-1})^{-1},
\end{align*}
while if $\cF/F$ is ramified, 
\begin{align*}
L(s,\itPi,r)
&=
(1-\alpha q^{-s-1})^{-1}(1-\alpha^{-1}q^{-s-1})^{-1}. 
\end{align*}
This shows $(2)$ and our proof is complete.
\end{proof}

\begin{prop}\label{P:(2,1) with 1 unramified special}
Let $E=K\x F$ and $\pi={\rm St}_F\ot\chi$, where $\chi$ is a unramified quadratic character of 
$F^{\x}$.
\begin{itemize}
\item[(1)]
If $\pi'$ is sperical, $\chi(\varpi)=-1$ and $K$ is ramified over $F$, then we have
\[
I^*\left(\itPi,\bft\right)
=
2q^{-1}(1+q^{-1})^{-1}.
\]
\item[(2)]
If $\pi'={\rm St}_{\cF}\ot\chi'$, where $\chi'$ is a unramified quadratic character of $K^{\x}$,
then we have
\[
I^*(\itPi,\bft)
=
\begin{cases}
2q^{-1}(1+q^{-1})^{-1}(1+q^{-2})&\text{if $\cF$ is unramified over $F$ and $\chi'\chi(\varpi)=1$},\\
q^{-1}&\text{if $\cF$ is ramified over $F$}.
\end{cases}
\]
\end{itemize}
\end{prop}

\begin{proof}
We first consider $(1)$. By definition,
\[
I\left(\itPi,\bft\right)
=
\int_{F^{\x}\backslash\G(F)}
\Phi(h)dh,
\]
where
\[
\Phi(h)
=
\Phi_{\pi'}
\left(h;
\begin{pmatrix}
\varpi^{-1}_{\cF}&0\\0&1
\end{pmatrix}
\right) 
\Phi_{\pi}(h),\quad h\in\G(F).
\]
By \eqref{E:integral formula for iwahori}, \lmref{L:macdonald formula} and
\lmref{L:mc for unramified special}, we have 
\begin{align*}
I\left(\itPi,\bft\right)
&=
(1+q)^{-1}(1-\chi(\varpi))\sum_{n=-\infty}^{\infty}\chi(\varpi)^n
\Phi_{\pi'}
\left( 
\begin{pmatrix}
\varpi^n&0\\0&1
\end{pmatrix}
\right)\\
&=
2q^{-1}\frac{(1-q^{-1})}{(1+q^{-1})^2}\cdot
\frac{(1+\alpha^2 q^{-1})(1+\alpha^{-2}q^{-1})}
{(1-\alpha^2 q^{-1})(1-\alpha^{-2}q^{-1})}. 
\end{align*}
Notice that
\[
L(s,\itPi,r)
=
(1-\chi(\varpi)\alpha^2q^{-s-1/2})^{-1}(1-\chi(\varpi)\alpha^{-2}q^{-s-1/2})^{-1}
(1-\chi(\varpi)q^{-s-1/2})^{-1}.
\]
This shows $(1)$.

Now we consider $(2)$. By definition,
\[
I\left(\itPi,\bft\right)
=
\int_{F^{\x}\backslash\G(F)}
\Phi(h)dh,
\]
where
\[
\Phi(h)
=
\Phi_{\pi'}(h)\Phi_{\pi}(h),\quad h\in\G(F).
\]
Suppose $\cF$ is umramified over $F$.
Applying \eqref{E:integral formula for iwahori}, \lmref{L:mc for unramified special}, we find that 
\begin{align*}
I\left(\itPi,\bft\right)
&=
(1+q)^{-1}
\left\lbrace 
\sum_{n=-\infty}^\infty
\Phi
\left( 
\begin{pmatrix}
\varpi^n&0\\0&1
\end{pmatrix}
\right)q^{|n|}
+
\sum_{n=-\infty}^\infty
\Phi
\left(
\pMX 0110\begin{pmatrix}
\varpi^n&0\\0&1
\end{pmatrix}
\right)q^{|n-1|}
\right\rbrace\\
%&=
%(1+q)^{-1}(1+\chi'\chi(\varpi))\sum_{n=-\infty}^\infty
%\chi'\chi(\varpi)^{|n|}q^{-2|n|}\\
&=
(1+q)^{-1}(1+\chi'\chi(\varpi))
\frac{(1+\chi'\chi(\varpi)q^{-2})}{(1-\chi'\chi(\varpi)q^{-2})}.
\end{align*}
When $\cF$ is ramified over $F$, a similar calculation shows that
\[
I^*\left(\itPi,\bft\right)
=
q^{-1}
\frac{(1+\chi'\chi(\varpi)q^{-1})}
{(1-\chi'\chi(\varpi)q^{-2})}.
\]
Note that the $L$-factors are
\[
L(s,\itPi,r)
=
\begin{cases}
(1-\chi'\chi(\varpi)q^{-s-3/2})^{-1}(1-q^{-2s-1})^{-1}&\text{if $K/F$ is unramified},\\
(1-\chi(\varpi)q^{-s-3/2})^{-1}(1-\chi(\varpi)q^{-s-1/2})^{-1}&\text{if $K/F$ is ramified}.
\end{cases}
\]
This proves the proposition.
\end{proof}

\begin{prop}\label{P:(3,0) case}
Let $E$ is a field. 
\begin{itemize}
\item[(1)]
If $E$ is ramified over $F$ and $\itPi$ is spherical, then we have
\[
I^*\left(\itPi,\bft\right)
=
1
\]
\item[(2)]
If $\itPi={\rm St}_E\ot\chi$, where $\chi$ is the non-trivial 
unramified quadratic character of $E^{\x}$, then we have
\[
I\left(\itPi,\bft\right)
=
\begin{cases}
2q^{-1}(1+q^{-1})^{-1}(1-q^{-1}+q^{-2})&\text{if $E/F$ is unramified},\\
2q^{-1}(1+q^{-1})^{-1}&\text{if $E/F$ is ramified}.
\end{cases}
\]
\end{itemize}
\end{prop} 

\begin{proof}
Suppose $\itPi$ is spherical and $E/F$ is ramified. Let 
\[
\itPi
=
{\rm Ind}_{B(E)}^{\G(E)}\left(|\cdot|^{\lambda}_E\bt|\cdot|^{-\lambda}_E\right),\quad
\alpha=|\varpi_E|^{\lambda}_E.
\]
For a non-negative integer $n$, let $X_n$ be the image of 
\[
\G(\cO_F)
\begin{pmatrix}
\varpi^n&0\\0&1
\end{pmatrix}
\G(\cO_F)
\]
in $F^{\x}\backslash\G(F)$. Note that
\[
{\rm Vol}(X_n,dh)
=
\begin{cases}
1\quad&\text{if $n=0$},\\
q^n(1+q^{-1})\quad&\text{if $n\geq 1$}.
\end{cases}
\]
Applying \lmref{L:macdonald formula},
\[
I\left(\itPi,\bft\right)
=
\sum_{n=0}^\infty
\Phi_{\itPi}
\left( 
\begin{pmatrix}
\varpi^n&0\\0&1
\end{pmatrix}
\right) 
{\rm Vol}(X_n,dh)
=
\frac{(1-q^{-1})(1+\alpha q^{-1/2})(1+\alpha^{-1}q^{-1/2})}
{(1-\alpha^3 q^{-1/2})(1-\alpha^{-3}q^{-1/2})}.
\]
Notice that 
\begin{align*}
L(s,\Pi,r)
=
(1-\alpha^{3}p^{-s})^{-1}(1-\alpha p^{-s})^{-1}(1-\alpha^{-1}p^{-s})^{-1}(1-\alpha^{-3}p^{-s})^{-1}.
\end{align*}
This proves $(2)$.

Suppose $\itPi={\rm St}_E\ot\chi$, where $\chi$ is the non-trivial 
unramified quadratic character of $E^{\x}$. If $E/F$ is unramified,
\[
I\left(\itPi,\bft\right)
=
\int_{F^{\x}\backslash\G(F)}
\Phi_{\itPi}(h)dh.
\]
By \eqref{E:integral formula for iwahori} and \lmref{L:mc for unramified special}, we obtain
\begin{align*}
I\left(\itPi,\bft\right)
&=
(1+q)^{-1}
\left\lbrace 
\sum_{n=-\infty}^\infty
\Phi_{\itPi}
\left( 
\begin{pmatrix}
\varpi^n&0\\0&1
\end{pmatrix}
\right)q^{|n|}
+
\sum_{n=-\infty}^\infty
\Phi_{\itPi}
\left( 
\pMX 0110 \begin{pmatrix}
\varpi^n&0\\0&1
\end{pmatrix}
\right)q^{|n-1|}
\right\rbrace\\  
%&=
%(1+q)^{-1}(1-\chi(\varpi)) 
%\sum_{n=-\infty}^\infty
%\chi(\varpi)^n q^{-2|n|}\\
&=
(1+q)^{-1}(1-\chi(\varpi)) 
\frac{(1+\chi(\varpi)q^{-2})}{(1-\chi(\varpi)q^{-2})}. 
\end{align*}
When $E/F$ is ramified, similar calculations show
\[
I^*(\itPi,\bft)
=
\frac{2q^{-1}}{(1+q^{-1})}.
\]
On the other hand, the $L$-factors are 
\begin{align*}
L(s,\Pi,r)
&= \left \{ \begin{array}{ll}
(1-\chi(\varpi)q^{-s-3/2})^{-1}(1+\chi(\varpi) q^{-s-1/2}+q^{-2s-1})^{-1} & {\mbox{if }} E/F 
\mbox{ is} \mbox{ unramified }, \\
(1-\chi(\varpi)q^{-s-3/2})^{-1} & {\mbox{if }} E/F \mbox{ is} \mbox{ ramified }.
\end{array} \right.
\end{align*}
This completes the proof.
\end{proof}

\section{The calculation of local zeta integral: (II)}
\label{S:local zeta integral: division algebra}
The purpose of this section is to compute the normalized zeta integral $I^*(\itPi,\bft)$ in \defref{D: definition of zeta integral} when $D$ is a division algebra over $F$.

\subsection{Haar measures}\label{SS:measure for division algebra}
Haar measures on $F$ and $F^{\x}$ are the same as in \subsecref{SS:measure for matrix algebra}.
We describe the choice of Haar measures on $D^{\x}(F)$. When $F=\R$, let $dh$ be the Haar measure 
on $D^{\x}(\R)$ such that ${\rm Vol}(D^{\x}(\R)/\R^{\x},dh/d^{\x}t)=1$, where 
$d^{\x}t=|t|^{-1}_{\R}dt$ and $dt$ is the usual Lebesgue measure on $\R$. When $F$ is 
non-archimedean, let $\cO_D$ be its maximal compact 
subring. Then $dh$ is chosen so that ${\rm Vol}\left(\cO_D^{\x},dh\right)=1$.

In any cases, the measure on the quotient space $F^{\x}\backslash D^{\x}(F)$ is the unique 
quotient measure induced from the measure $dh$ on $D^{\x}(F)$ and the measure $d^{\x}x$ on $F^{\x}$.

\subsection{Embeddings}\label{SS:embeddings}
We fix various embeddings in this section. Following results depend on these embeddings. 
When $F=\R$, we embedded $D(\R)$ in ${\rm M}_2(\C)$ in the usual way. More precisely, we let
\[
D(\R)=\mathbf{H}
=
\stt{\pMX{a}{b}{-\b{b}}{\b{a}}\in {\rm M}_2(\C)}.
\]
When $F$ is non-archimedean and $E=K\x F$, we have $D(E)={\rm M}_2(K)\x D(F)$ and
we fix an embedding:
\[
\iota: D(F)\rightarrow {\rm M}_2(\cF),
\]
so that 
\[
\iota(D(F))
=
\stt{
\begin{pmatrix}
\alpha&\beta\\ \omega\b{\beta}&\b{\alpha}
\end{pmatrix}
\mid
\alpha, \beta\in \cF},
\]
where $x\mapsto\b{x}$ is the non-trivial Galois action on $x\in \cF$, and $\omega$ is either
$\varpi$ or a unit $u$ such that $F(\sqrt{u})$ is the unramified extension over $F$, 
according to $\cF$ is unramified or ramified over $F$. We then identify $D(F)$ with
its image under the embedding $\iota$. 
The maximal order $\cO_D$ in $D(F)$ is then 
\[
\stt{\begin{pmatrix}
\alpha&\beta\\ \omega\b{\beta}&\b{\alpha}
\end{pmatrix}
\mid
\alpha, \beta\in\cO_{\cF}}.
\] 
Let
\[
\varpi_D
=
\begin{pmatrix}
0&1\\\varpi&0
\end{pmatrix}
\quad\text{or}\quad
\varpi_D
=
\begin{pmatrix}
\varpi_{\cF}&0\\0&-\varpi_{\cF}
\end{pmatrix},
\]
according to $\cF$ is unramified or ramified over $F$. We have
\begin{equation}\label{E:decomposition of division modulo center}
F^{\x}\backslash D^{\x}(F)
=
\left( 
\cO_F^{\x}\backslash\cO_D^{\x}
\right) 
\sqcup
\varpi_D
\left( 
\cO_F^{\x}\backslash\cO_D^{\x}
\right).
\end{equation}
Note that 
\[
{\rm Vol}(\cO_F^{\x}\backslash\cO_D^{\x}, dh)=1,
\] 
according to our choice of measures. 

\subsection{The archimedean case}
\label{SS:balanced case}
In this case, we have following realizations
\[
\left(\pi_j,V_{\pi_j}\right)
=
\left(\rho_{k_j},\cL_{k_j}(\C)\right)
\] 
for $j=1,2,3$, where
\[
\cL_{k_j}(\C)=\bigoplus_{n_j=0}^{k_j-2}\C\cdot X_j^{n_j}Y_j^{k_j-2-n_j}
\quad\text{with}\quad
\rho_{k_j}(g)P(X_j,Y_j)=P((X_j,Y_j)g){\rm det(g)}^{-k_j/2-1},
\]
for $g\in D^{\x}(\R)$ and $P(X_j,Y_j)\in\cL_{k_j-2}(\C)$. 
The representation space of $\itPi$ is given by
\begin{equation}\label{E:representation space for balanced case}
V_{\itPi}=\cL_{k_1}(\C)\ot\cL_{k_2}(\C)\ot\cL_{k_3}(\C).
\end{equation}
The new line $V^{\rm new}_{\itPi}$ in this case is the one dimensional subspace fixed by $D^\x(\R)$. Let $\bfP_{\ul{k}}$ be the distinguished vector in $V_{\itPi}^{\rm new}$ defined by
\begin{equation}\label{E:invariant vector}
\bfP_{\ul{k}}
=
\det\begin{pmatrix}
X_1&X_2\\
Y_1&Y_2
\end{pmatrix}^{k_3^*}\otimes \det\begin{pmatrix}
X_2&X_3\\
Y_2&Y_3
\end{pmatrix}^{k_1^*}\otimes \det\begin{pmatrix}
X_3&X_1\\
Y_3&Y_1
\end{pmatrix}^{k_2^*}
\end{equation} 
where  $k_3^*=(k_1+k_2-k_3-2)/2$, $k_1^*=(k_2+k_3-k_1-2)/2$ and $k_2^*=(k_1+k_3-k_2-2)/2$. 
Its clear that $\bfP_{\ul{k}}$ is non-zero and invariant by $D^\x(\R)$. Therefore, we have
\[
V_{\itPi}^{\rm new}=\C\cdot\bfP_{\ul{k}}.
\] 

Let $\langle\,,\,\rangle_{k_j}$ be the $D^{\x}(\R)$-invariant bilinear pairing on 
$\cL_{k_j-2}(\C)$ defined by 
\begin{equation}\label{E:pairing for balanced case}
\langle X_j^{n_j}Y_j^{k_j-2-n_j}, X_j^{m_j}Y_j^{k_j-2-m_j}\rangle_{k_j}
=
\begin{cases}
(-1)^{n_j}\begin{pmatrix}k_j-2\\n_j\end{pmatrix}^{-1}\quad&\text{if $n_j+m_j=k_j-2$},\\
%\begin{pmatrix}k_j\\n_j\end{pmatrix}^{-1}\quad&\text{if $n_j+m_j=k_j-2$},\\
0\quad&\text{if $n_j+m_j\neq k_j-2$},
\end{cases}
\end{equation}
for $0\leq n_j, m_j\leq k_j-2$.  Let $\langle\,,\,\rangle_{\ul{k}}$ be the $D^{\x}(E)$-invariant pairing on $V_\itPi$ given by
\begin{equation}\label{E:pairing for balanced case 2}
\langle\,,\,\rangle_{\ul{k}}
=
\langle\,,\,\rangle_{k_1}
\ot
\langle\,,\,\rangle_{k_2}
\ot
\langle\,,\,\rangle_{k_3}.
\end{equation}
In this case, the normalized local zeta integral $I^*(\itPi,\bft)$ in \defref{D: definition of zeta integral} is equal to 
%\begin{defn}\label{D:local period integral for the balabced case}
%In the case $D$ is the division algebra over $\R$, we define the local zeta integral by
\beq\label{E:local period integral for the balabced case}
I^*(\itPi,\bft)
=
\frac{\zeta_F(2)}{\zeta_E(2)}\cdot
\frac{L(1,\itPi',{\rm Ad})}{L(1/2,\itPi', r)}\cdot
\langle\bfP_{\ul{k}},\bfP_{\ul{k}}\rangle_{\ul{k}},
\eeq
where $\itPi'$ is the Jacquet-Langlands lift of $\itPi$ to $\G(\R)$.
%\end{defn}
We proceed to compute the value
$\langle\bfP_{\ul{k}},\bfP_{\ul{k}}\rangle_{\ul{k}}$. Let $\ell$ be the linear map
\[
\ell:V_{\itPi}\rightarrow V_{\itPi}^{D^\x(\R)}=V_{\itPi}^{\rm new},
\quad v\mapsto \ell(v)=\int_{\R^{\x}\backslash D^{\x}(\R)}\itPi(h)v\,dh.
\]
Since $\ell(\bfP_{\ul{k}})=\bfP_{\ul{k}}\neq 0$, we have $\ell\neq 0$ and hence surjective.
We have the following equality
\begin{equation}\label{E:observation}
\langle\bfP_{\ul{k}}, \bfP_{\ul{k}}\rangle_{\ul{k}}
\cdot
\langle\ell(v_1), \ell(v_2)\rangle_{\ul{k}}
=
\langle v_1, \bfP_{\ul{k}}\rangle_{\ul{k}}
\cdot
\langle v_2, \bfP_{\ul{k}}\rangle_{\ul{k}}
\end{equation}
for every $v_1, v_2\in\cL(\C)$.

\begin{prop}\label{P:archimedean balanced case}
We have \[
I^*\left(\itPi, \bft\right)
=
\frac{(k_1-1)(k_2-1)(k_3-1)}{4\pi^2}.
\]
\end{prop}

\begin{proof} Note that the $L$-factor is given by
\begin{align*}
L(s,\itPi,r)
&=
\zeta_{\C}(s+(k_1+k_2+k_3-3)/2))\zeta_{\C}(s+(-k_1+k_2+k_3-1)/2)\\
&\x\zeta_{\C}(s+(k_1-k_2+k_3-1)/2)\zeta_{\C}(s+(k_1+k_2-k_3-1)/2)).
\end{align*}
In view of \eqref{E:local period integral for the balabced case}, it suffices to show that 
\[
\langle\bfP_{\ul{k}}, \bfP_{\ul{k}}\rangle_{\ul{k}}
=
\frac{\Gamma(k_1^*+k_2^*+k_3^*+2)\Gamma(k_1^*+1)\Gamma(k_2^*+1)\Gamma(k_3^*+1)}
     {\Gamma(k_1^*+k_2^*+1)\Gamma(k_1^*+k_3^*+1)\Gamma(k_2^*+k_3^*+1)}.
\]
By direct computation, we have
\[
\begin{aligned}
\bfP_{\ul{k}}
=
\sum_{n_1=0}^{k^*_1}\sum_{n_2=0}^{k^*_2}\sum_{n_3=0}^{k^*_3}
&\begin{pmatrix}k^*_1\\n_1\end{pmatrix}
\begin{pmatrix}k^*_2\\n_2\end{pmatrix}
\begin{pmatrix}k^*_3\\n_3\end{pmatrix}
(-1)^{(k^*_1+k^*_2+k^*_3)-(n_1+n_2+n_3)}\\
&X_1^{k^*_2-n_2+n_3} Y_1^{k^*_3+n_2-n_3}\ot
X_2^{k^*_3+n_1-n_3} Y_2^{k^*_1-n_1+n_3}\ot
X_3^{k^*_1-n_1+n_2} Y_3^{k^*_2+n_1-n_2}.
\end{aligned}
\]
The coefficient in front of the vector 
$v_1:=X_1^{k_1-2}\otimes Y_2^{k_2-2}\otimes X_3^{k_1^*}Y_3^{k^*_2}$
in the expression of $\bfP_{\ul{k}}$ is equal to $(-1)^{k^*_1+k^*_2}$. 
On the other hand, the coefficient in front of the vector
$v_2:=Y_1^{k_1-2}\ot X_2^{k_2-2}\otimes X_3^{k_2^*}Y_3^{k^*_1}$ is $(-1)^{k^*_3}$.
It follows that 
\begin{equation}\label{E:value}
\langle v_1, \bfP_{\ul{k}}\rangle_{\ul{k}}
\cdot
\langle v_2, \bfP_{\ul{k}}\rangle_{\ul{k}}
=
(-1)^{k^*_1+k^*_2+k^*_3}\cdot\langle v_1,v_2\rangle_{\ul{k}}^2
=
(-1)^{k^*_1+k^*_2+k^*_3}
\begin{pmatrix}
k^*_1+k^*_2\\k_1^*
\end{pmatrix}^{-2}.
\end{equation}
On the other hand, we have
\[
\langle\ell(v_1),\ell(v_2)\rangle_{\ul{k}}
=
\int_{\R^{\x}\backslash D^{\x}(\R)}
\langle\itPi(h)v_1, v_2\rangle_{\ul{k}}dh.
\]
Note that
\[
\R^{\x}\backslash D^{\x}(\R)\cong \stt{\pm 1}\backslash\SU(2),
\]
We parametrize $u=\begin{pmatrix}\alpha&\beta\\-\b{\beta}&\b{\alpha}\end{pmatrix}\in\SU(2)$ 
by setting $\alpha=\text{cos}\,\theta\cdot e^{i\varphi}$ and 
$\beta=\text{sin}\,\theta\cdot e^{i\chi}$ with $0\leq \theta\leq\pi/2$ and 
$0\leq \varphi, \chi\leq 2\pi$. For $\Phi\in\text{L}^1(\SU(2))$, we have
\begin{equation}\label{E:int G}
\int_{\SU(2)}\Phi(u)\,du
=
\frac{1}{2\pi^2}\int_0^{2\pi}\int_0^{2\pi}\int_0^{\frac{\pi}{2}}\,\,
\Phi(\theta,\varphi,\chi)\cdot\text{sin}\,2\theta\,\,d\theta\,d\varphi\,d\chi.
\end{equation}
Our choice of the Haar measure on $\R^{\x}\backslash D^{\x}(\R)$ 
implies the total volume of $\SU(2)$ is equal to 2.\\
Let $u=\begin{pmatrix}\alpha&\beta\\-\b{\beta}&\b{\alpha}\end{pmatrix}\in\SU(2)$. 
By \eqref{E:int G}, we have
\begin{align*}
\int_{\R^{\x}\backslash D^{\x}(\R)}
\langle\itPi(h)v_1, v_2\rangle_{\ul{k}}dh
&=
\frac{1}{2}\int_{\SU(2)}
\langle\itPi(u)v_1, v_2\rangle_{\ul{k}}du\\
&=
(-1)^{k^*_1+k^*_2+k^*_3}
\begin{pmatrix}
k^*_1+k^*_2\\k^*_1
\end{pmatrix}^{-1}
\sum_{j=0}^{k^*_1}
\begin{pmatrix}
k^*_1\\j
\end{pmatrix}
\begin{pmatrix}
k^*_2\\j
\end{pmatrix}
\frac{(-1)^j}{2}\int_{\SU(2)}
|\alpha|^{k^*_1+k^*_1+k^*_3-j}_{\C}|\beta|^j_{\C}du\\
&=
(-1)^{k^*_1+k_1-2}(k^*_1+k^*_2+k^*_3+1)^{-1}
\begin{pmatrix}k^*_1+k^*_2\\k^*_1\end{pmatrix}^{-1}
\sum_{j=0}^{k_1^*}(-1)^j
\frac{\begin{pmatrix}k_1^*\\j\end{pmatrix}\begin{pmatrix}k_2^*\\j\end{pmatrix}}
{\begin{pmatrix}k^*_1+k^*_2+k^*_3\\j\end{pmatrix}}. 
\end{align*}
Using \eqref{E:observation}, \eqref{E:value} and the equation above, we obtain
\begin{align*}
\langle\bfP_{\ul{k}}, \bfP_{\ul{k}}\rangle_{\ul{k}}^{-1}
&=
(k^*_1+k^*_2+k^*_3+1)^{-1}
\begin{pmatrix}k^*_1+k^*_2\\k^*_1\end{pmatrix}
\sum_{j=0}^{k_1^*}(-1)^j
\frac{\begin{pmatrix}k_1^*\\j\end{pmatrix}\begin{pmatrix}k_2^*\\j\end{pmatrix}}
{\begin{pmatrix}k^*_1+k^*_2+k^*_3\\j\end{pmatrix}}\\
&=
(k^*_1+k^*_2+k^*_3+1)^{-1}
\begin{pmatrix}k^*_1+k^*_2\\k^*_1\end{pmatrix}
\begin{pmatrix}k_1^*+k_2^*+k_3^*\\k_2^*+k_3^*\end{pmatrix}^{-1}
\sum_{j=0}^n(-1)^j
\begin{pmatrix}k_2^*\\j\end{pmatrix}
\begin{pmatrix}k_1^*+k_2^*+k_3^*-j\\k_2^*+k_3^*\end{pmatrix}\\
&=
\frac{\Gamma(k_1^*+k_2^*+1)\Gamma(k_1^*+k_3^*+1)\Gamma(k_2^*+k_3^*+1)}
{\Gamma(k_1^*+k_2^*+k_3^*+2)\Gamma(k_1^*+1)\Gamma(k_2^*+1)\Gamma(k_3^*+1)}.
\end{align*}
The last equality follows from \lmref{L:combintorial identity 3} below. This completes the proof of \propref{P:archimedean balanced case}.
\end{proof}
\begin{lm}\label{L:combintorial identity 3}
Let $a,b$ and $n$ be non-negative integers. Suppose $a\geq n$. Then we have 
\begin{equation}\label{E:combintorial identity 3}
\sum_{j=0}^n(-1)^j\begin{pmatrix}a\\j\end{pmatrix}\begin{pmatrix}a+b+n-j\\a+b\end{pmatrix}
 =\begin{pmatrix}b+n\\b\end{pmatrix}.
\end{equation}
\end{lm}
\begin{proof}
Consider the function
\[
f(X)
=
\sum_{j=0}^a(-1)^j\begin{pmatrix}a\\j\end{pmatrix}(1+X)^{a+b+n-j}
\]
where $X$ is a variable. 
Since we have assumed $a\geq n$, the coefficient of the term $X^{a+b}$ in $f(X)$ is the
right hand side of \eqref{E:combintorial identity 3}. 
On the other hand, we have
\begin{align*}
f(X)
&=
\sum_{j=0}^a(-1)^j\begin{pmatrix}a\\j\end{pmatrix}(1+X)^{a+b+n-j}\\
&=
(1+X)^{a+b+n}\sum_{j=0}^a(-1)^j\begin{pmatrix}a\\j\end{pmatrix}(1+X)^{-j}\\
&=
(1+X)^{a+b+n}(1-(1+X)^{-1})^a=X^a(1+X)^{b+n}.
\end{align*}
Its clear that the coefficient of the term $X^{a+b}$ in $X^a(1+X)^{b+n}$ is equal to the 
left hand side of \eqref{E:combintorial identity 3}. This finishes the proof of lemma.
\end{proof}
\subsection{The non-archimedean case}
Let $F$ be a non-archimedean local field and $D$ is the quaternion division algebra over $F$. Recall that we have assumed 
\[
{\rm Hom}_{D^{\x}}(\itPi,\C)\neq\stt{0}.
\]
According to the results of Prasad \cite{Prasad1990}, \cite{Prasad1992} and our assumption on 
$\itPi$, this happens precisely for the cases being considered in the following proposition.
 
\begin{prop}\label{E:local integral divison: non-archimedean case}
Let $\nu_D:D^{\x}\to\mathbb{G}_m$ be the reduced norm of $D$.
\begin{itemize}
\item[(1)] 
Let $E=F\x F\x F$. If $\pi_j=\chi_j\circ\nu_D$, where $\chi_j$ is a unramified quadratic 
character of $F^{\x}$ with $\chi_1\chi_2\chi_3(\varpi)=1$. Then we have
\[
I^*(\itPi,\bft)
=
2(1-q^{-1})^2
\] 
\item[(2)] 
Let $E=K\x F$ and $\pi=\chi\circ\nu_D$, where $\chi$ is a unramified quadratic character of $F^{\x}$.
Then we have
\[
I^*(\itPi,\bft)
=
\begin{cases}
1&\text{if $\pi'$ is spherical and $K/F$ is unramified},\\
2&\text{if $\pi'$ is spherical, $\chi(\varpi)=1$ and $K/F$ is ramified}\\
2(1+q^{-2})&\text{if $\pi'={\rm St}_K\ot\chi'$ with $\chi'\chi(\varpi)=-1$}.
\end{cases}
\] 
Here $\chi'$ is a unramified quadratic character of $K^{\x}$.
\item[(3)]
Let $E$ be a field. If $\itPi$ is the trivial character of $D^{\x}(E)$, then we have
\[
I^*(\itPi,\bft)
=
\begin{cases}
2(1+q^{-1}+q^{-2})&\text{if $E/F$ is unramified},\\
2&\text{if $E/F$ is ramified}.
\end{cases}
\]
\end{itemize}
Here $I^*\left(\itPi,\bft\right)$ is the local zeta integral in  \defref{D: definition of zeta integral}. 
\end{prop}

\begin{proof}
We first treat $(1)$. Since $\chi_1\chi_2\chi_3(\varpi)=1$, we have
\[
I\left(\itPi,\bft\right)
=
\int_{F^{\x}\backslash D^{\x}(F)}\chi_1\chi_2\chi_3(\nu_D(h))dh
=
{\rm Vol}(F^{\x}\backslash D^{\x}(F), dh)
=
2.
\]
The $L$-factor is 
\[
L(s,\itPi',r)
=
(1-\chi_1\chi_2\chi_3(\varpi) q^{-s-1/2})^{-2}(1-\chi_1\chi_2\chi_3(\varpi) q^{-s-3/2})^{-1},
\]
where $\itPi'$ is the Jacquet-Langlands lift of $\itPi$ to $\G(F)$. This shows $(1)$.

We proceed to show $(2)$. Suppose $\pi'$ is spherical. then by \lmref{L:macdonald formula} and 
\eqref{E:decomposition of division modulo center}, we find that
\begin{align*}
I\left(\itPi,\bft\right)
&=
\int_{F^{\x}\backslash D^{\x}(F)}
\Phi_{\pi'}(h)\pi(h)dh
=
1
+
\Phi_{\pi'}(\varpi_D)\chi(\varpi)\\
&=
\begin{cases}
(1+q^{-2})^{-1}(1+\chi(\varpi)\alpha q^{-1})(1+\chi(\varpi)\alpha^{-1}q^{-1})
\quad&\text{if $K/F$ is unramified},\\
2
\quad&\text{if $K/F$ is ramified}.    
\end{cases}
\end{align*}

Suppose $\pi'={\rm St}_K\ot\chi'$. In this case,
\[
I\left(\itPi,\bft\right)
=
\int_{F^{\x}\backslash D^{\x}(F)}
\Phi_{\pi'}(h)\pi(h)dh
=
1+\Phi_{\pi'}(\varpi_D)\chi(\varpi)
=
(1-\chi'\chi(\varpi))
=
2.
\]
Here we use \lmref{L:mc for unramified special} and the observation that 
$\cO_D^{\x}$ is contained in the Iwahori subgroup of $\G(\cF)$.
The $L$-factors are given in \propref{P:(2,1) with 1 spherical} and
\propref{P:(2,1) with 1 unramified special}. This shows $(2)$.

For the case $(3)$, we have
\[
I(\itPi,\bft)
=
{\rm Vol}\left(F^{\x}\backslash D^{\x}(F), dh\right)=2.
\]
The $L$-factors are given in \propref{P:(3,0) case}. This completes the proof.
\end{proof}
  
\section{Explicit central value formulae and algebraicity for triple product}
\label{S:explicit central value formula}
The purpose of this section is to give explicit central value formulae for the triple product 
$L$-functions by combining Ichino's formula \cite[Theorem 1.1 and Remark 1.3]{Ichino2008} with the local calculations in the previous sections. We use these formulae to prove the algebraicity of the central values. 

Since the work of  Garrett \cite{Garrett1987}, special values of triple product $L$-functions have been studied extensively by many people such as Orloff \cite{Orloff1987}, Satoh \cite{Satoh1987},
Harris and Kudla \cite{HarrisKudla1991}, Garrett and Harris \cite{GarrettHarris1993},
Gross and Kudla \cite{GrossKudla1992}, Bocherer and Schulze-Pillot \cite{BochererSchulze1996},
Furusawa and Morimoto \cite{Furusawa2014}, \cite{Furusawa2016}. 
%{\color{red}However, all these results were 
%concerned with the special values of triple product $L$-functions attached to three elliptic 
%newforms, or three Hilbert newforms \cite[page 652]{Shimura1978}. In this section, we 
%will prove the algebraicity of the central critical value for the triple product $L$-functions 
%attached to a Hilbert modular new form of a real quadratic extension over $\Q$ and an 
%elliptic new form, or a Hilbert newform of a real cubic extension over $\Q$}.

\subsection{Notation}
We fix some notations here. If $F$ is a number field, let $\cO_F$ be its ring of integers,
$\cD_F$ be its absolute discriminant, and $h_F$ be its class number. 
Let $\A$ be the ring of adeles of $\Q$ and $\wh{\Z}=\prod_p\Z_p$ be the profinite completion of $\Z$. 
We will denote by $v$ a place of $\Q$ and by $p$ a finite prime of $\Q$.
If $R$ is a $\Q$-algebra, let $\A_R=\A\ot_\Q R$ and $R_v=R\ot_\Q \Q_v$. 
For an abelian group $M$, let $\wh{M}=M\ot_\Z\wh{\Z}$.

We fix an additive character $\psi=\prod_v\psi_{v}:\Q\backslash\A\to\C^{\x}$ defined by 
$\psi_{\infty}(x)=e^{2\pi \sqrt{-1}x}$ for $x \in \R$, and $\psi_{p}(x)=e^{-2\pi \sqrt{-1}x}$ for 
$x \in \Z[p^{-1}]$. 
\subsection{Modular forms and Automoprhic forms}
\label{SS:modular forms and automorphic forms}
We briefly review the definitions of modular forms and automorphic forms on certain quaternion 
algebras, and we write down an explicit correspondence between them. We follow the exposition of 
\cite[section 1]{Shimura1981}, but with some modifications, so that it will be suitable for our
application here. 

We first introduce some notations.
Let $d\geq 1$ be an integer and $\frak{H}^d$ be the $d$-fold product of the upper half complex 
plane $\frak{H}$. Let ${\rm GL}^+_2(\R)$ be the identity connect component of $\G(\R)$.
If $d=1$, we let $h\in{\rm GL}^+_2(\R)$ acting on $z\in\frak{H}$ and
we define the factor $J(h,z)$ by 
\begin{align*}
h\cdot z&=\frac{az+b}{cz+d},\\
J(h,z)&={\rm det}(g)^{-\frac{1}{2}}(cz+d)
\quad h=\pMX abcd.
\end{align*}
In general, we let ${\rm GL}^+_2(\R)^d$ acting on $\frak{H}^d$ component-wise. If 
$\ul{k}=(k_1,\ldots,k_d)\in\Z^d$, we put 
\[
J(h,z)^{\ul{k}}
=
\prod_{j=1}^d j(h_j, z_j)^{k_j}
\quad
\text{for}
\quad
h=(h_1,\ldots,h_d)\in{\rm GL}^+_2(\R)^d
\quad
\text{and}
\quad 
z=(z_1,\ldots,z_d)\in\frak{H}^d.
\]

Let $C^\infty(\frak{H})$ be the space of $\C$-valued smooth functions on $\frak{H}$. Let $k$ be an 
integer. Recall the Maass-Shimura differential operators 
$\delta_k$ and $\varepsilon$ on $C^\infty(\frak{H})$ are given by
\[
\delta_k
=
\frac{1}{2\pi \sqrt{-1}}\left (\frac{\partial}{\partial z}+\frac{k}{2\sqrt{-1}y} \right)
\quad\text{and}\quad
\varepsilon
=
-\frac{1}{2\pi \sqrt{-1}}y^2\frac{\partial}{\partial\b{z}}
\quad y={\rm Im}(z)
\]
(\cf \cite[page 310]{Hida1993}). If $m\geq 0$ is an integer, we put
$\delta_k^m=\delta_{k+2m-2}\cdots\delta_{k+2}\delta_k$. In general, if 
$\ul{k}=(k_1,\ldots,k_d),\,\ul{m}=(m_1,\ldots,m_d)\in\Z^d$ with $m_j\geq 0$ for 
$1\leq j\leq d$, we let $\delta_{\ul{k}}^{\ul{m}}$ and $\varepsilon^{\ul{m}}$ be given
by
\[
\delta_{\ul{k}}^{\ul{m}}=(\delta_{k_1}^{m_1},\ldots,\delta_{k_d}^{m_d})
\quad\text{and}\quad
\varepsilon^{\ul{m}}=(\varepsilon^{m_1},\ldots,\varepsilon^{m_d}),
\]
and acting on $f\in C^\infty(\frak{H}^d)$ coordinate-wise.

Let $F$ be a totally real number field over $\Q$ with degree $d=[F:\Q]$. 
Let $\Sigma_F:={\rm Hom}_\Q(F,\C)$ and $\frak{H}^{\Sigma_F}$ be the $d$-fold product of $\frak{H}$. 
Let $D$ be a quaternion algebra over $F$. Let $G=D^{\x}$ viewed as an algebraic group defined over 
$F$. For any $F$-algebra $L$, $G(L)=(D\ot_F L)^\x$. We assume $D$ is either totally indefinite or totally definite. In other words, we assume either  
$G(F_\infty)\cong\G(\R)^{\Sigma_F}$ or $G(F_\infty)\cong(\mathbf{H}^{\x})^{\Sigma_F}$, where $\bfH$ is the Hamiltonian quaternion algebra.

\subsubsection{The totally indefinite case}
Let $\ul{k}=(k_\sigma)_{\sigma\in\Sigma_F}, \ul{m}=(m_\sigma)_{\sigma\in\Sigma_F}
\in\Z^{\Sigma_F}$ with $k_\sigma>0$ and $m_\sigma\geq 0$ for all $\sigma\in\Sigma_F$. The zero and 
the identity element $\Z^{\Sigma_F}$ will be denoted by $\ul{0}$ and $\ul{1}$, respectively. 
Let $U\subset G(\wh{F})$ be an open compact subgroup. We assume $\nu_D(U)=\wh{\cO}^{\x}_F$, where 
$\nu_D$ is the reduced norm of $D$ and we extend it to a map on $D\ot_F\wh{F}$ in an obvious way.

Denote by $\cN_{\ul{k}}^{[\ul{m}]}(D,F;U)$ the space of functions 
$f:\frak{H}^{\Sigma_F}\x G(\wh{F})\to\C$ such that $f(z,ahu)=f(z,h)$ for 
$z\in\frak{H}^{\Sigma_F}$ and $(a,h,u)\in\wh{F}^{\x}\x G(\wh{F})\x U$.
Also, for each $h\in G(\wh{F})$, the function 
$f_h(z):=f(z,h)\in C^{\infty}(\frak{H}^{\Sigma_F})$ is slowly 
increasing and $\varepsilon^{\ul{m}+\ul{1}}f_h=0$. 
Finally, it satisfies the following automorphy condition:
\begin{equation}\label{E:automorphy}
f_h(\gamma\cdot z)J(\gamma,z)^{-\ul{k}}=f_h(z),
\quad
\gamma\in G(F)\cap\left(G^+(F_\infty)\x hUh^{-1}\right),
\end{equation}
where $G^+(F_\infty)$ is the
identity connect component of $G(F_\infty)$. We put
$\cN_{\ul{k}}(D,F;U)=\cup_{\ul{m}}\cN_{\ul{k}}^{[\ul{m}]}(D,F;U)$. 
Notice that if $f\in\cN_{\ul{k}}(D,F;U)$, then 
$\delta_{\ul{k}}^{\ul{m}}f\in\cN_{\ul{k}+2\ul{m}}(D,F;U)$ 
(Cf. \cite[page 312]{Hida1993}). Assume $D={\rm M}_2$ is the matrix algebra. 
Let $\frak{n}\subset\cO_F$ be an ideal. Put
\[
K_0(\wh{\frak{n}})
=
\stt{\pMX abcd\in\G(\wh{\cO}_F)\mid c\in\wh{\frak{n}}}.
\]
Then $\cN_{\ul{k}}^{[\ul{0}]}({\rm M}_2,F;K_0(\wh{\frak{n}}))
=\cM_{\ul{k}}({\rm M}_2,F;K_0(\wh{\frak{n}}))$ is the space of holomorphic Hilbert modular forms of 
$F$ of weight $\ul{k}$ and level $\frak{n}$. Let $\cS_{\ul{k}}({\rm M}_2,F;K_0(\wh{\frak{n}}))$ 
be the subspace of holomorphic cusp forms in $\cM_{\ul{k}}({\rm M}_2,F;K_0(\wh{\frak{n}}))$.  

We also define a subspace of automorphic forms on $G(\A_F)$ as follows. 
Let $\ul{k}$ and $U$ be as above. We identify $U$ and $G(F_\infty)$ with subgroups 
of $G(\A_F)$ in an obvious way. Let $\cA_{\ul{k}}(D,F;U)$ be the space of automorphic forms 
$\mathbf{f}:G(\A_F)\to\C$  (Cf. \cite[section 4]{BorelJacquet1979}) such that 
\begin{align*}
\mathbf{f}(a\gamma hk(\ul{\theta})u)
&=
\mathbf{f}(h)e^{\sqrt{-1}\ul{k}\cdot\ul{\theta}},
\quad 
\ul{k}\cdot\ul{\theta}=\sum_{\sigma\in\Sigma_F}k_\sigma\theta_\sigma,
\\
(a\in\A^{\x}_F,\,\,\gamma\in G(F),\,\,\ul{\theta}=(\theta_\sigma)_{\sigma\in\Sigma_F},\,\,
k(\ul{\theta})&=(k(\theta_\sigma))_{\sigma\in\Sigma_F},\,\,
k(\theta_\sigma)
=
\pMX{{\rm cos}\theta_\sigma}{{\rm sin}\theta_\sigma}
{-{\rm sin}\theta_\sigma}{{\rm cos}\theta_\sigma},\,\, 
u\in U).
\end{align*}
Denote by $\cA^0_{\ul{k}}(D,F;U)$ the subspace of cusp forms in $\cA_{\ul{k}}(D,F;U)$ 
(if there is no such, we let $\cA^0_{\ul{k}}(D,F;U)$ be the whole space).
Suppose $F=\Q$. Let $\t{V}_{\pm}:\cA_k(D,F;U)\to\cA_{k\pm 2}(D,F;U)$ be the normalized weight 
raising/lowing elements (\cite[page 165]{JLbook}) given by
\[
\t{V}_{\pm}
=
-\frac{1}{8\pi}
\left(
\begin{bmatrix} 1&0\\0&-1 \end{bmatrix}\otimes 1\pm
\begin{bmatrix} 0&1\\1&0 \end{bmatrix}\otimes \sqrt{-1}
\right)\in{\rm Lie}(\G(\R))\ot_\R\C.
\]
In general, we have 
$\t{V}_{\pm}^{\ul{m}}:\cA_{\ul{k}}(D,F;U)\to\cA_{\ul{k}\pm 2\ul{m}}(D,F;U)$, where 
$\t{V}_{\pm}^{\ul{m}}=(\t{V}_\pm^{m_\sigma})_{\sigma\in\Sigma_F}$ acts on the archimedean
component of $\mathbf{f}\in\cA_{\ul{k}}(D,F;U)$ coordinate-wisely.

We write down an explicit correspondence between the spaces $\cN_{\ul{k}}(D,F;U)$ and 
$\cA_{\ul{k}}(D,F;U)$. Fix a set of representatives $\stt{x_1,\cdots,x_h}$ for the double cosets
$G(F)\backslash G(\A_F)/G^+(F_\infty)U$. Then 
\[
G(\A_F)=\amalg_{j=1}^h G(F)x_jG^+(F_\infty)U
\]
is a disjoint union. We may assume every archimedean component of $x_j$ is one for $1\leq j\leq r$,
and we regard $x_j$ as elements in $G(\wh{F})$. For each $f\in 
\cN_{\ul{k}}(D,F;U)$, we define $\mathit{\Phi}(f)\in\cA_{\ul{k}}(D,F;U)$ \emph{the adelic lift} of $f$ by the formulae  
\begin{align*}
\mathit{\Phi}(f)(\gamma x_j h_\infty u)
&=
f_{x_j}(h_\infty\cdot\mathbf{i})J(h_\infty,\mathbf{i})^{-\ul{k}},
\quad
\mathbf{i}=(\sqrt{-1},\cdots,\sqrt{-1})\in\frak{H}^{\Sigma_F},
\\
(&\gamma\in G(F),\,\,h_\infty\in G^+(F_\infty),\,\,u\in U,\, 
\,1\leq j\leq h).
\end{align*}
Conversely, we can recover $f$ form $\it{\Phi}(f)$ by setting
\[
f(z,h)
=
\it{\Phi}(f)(h_\infty h)J(h_\infty,\mathbf{i})^{\ul{k}},
\quad 
h_\infty\in G^+(F_\infty)
\,\,\text{with}\,\,
h_\infty\cdot\mathbf{i}=z.
\]
The weight raising/lowering operators are the adelic version of the Maass-Shimura differential
operators $\delta_{\ul{k}}^{\ul{m}}$ and $\varepsilon^{\ul{m}}$ on the space of automorphic
forms. More precisely, one check that
\begin{equation}\label{E:raising relation}
\t{V}_+^{\ul{k}}\mathit{\Phi}(f)
=
\mathit{\Phi}(\delta_{\ul{k}}^{\ul{m}}f)
\quad\text{and}\quad
\t{V}_-^{\ul{m}}\mathit{\Phi}(f)
=
\mathit{\Phi}(\varepsilon^{\ul{m}} f).
\end{equation}
In particular, $f$ is holomorphic if and only if $\t{V}_-^{\ul{1}}\mathit{\Phi}(f)=0$.

\subsubsection{$D$ is totally definite}
Let $\ul{k}=(k_\sigma)_{\sigma\in\Sigma_F}$ and $U$ be as above. We assume 
$k_\sigma\geq 2$ for all $\sigma\in\Sigma_F$. We identify $G(F_\infty)$ with 
$(\mathbf{H}^{\x})^{\Sigma_F}\subset\G(\C)^{\Sigma_F}$. Let $(\rho_{k_\sigma},\cL_{k_\sigma}(\C))$
be the $(k_\sigma-1)$-dimensional irreducible representation of $\mathbf{H}^{\x}$, and 
$\langle\cdot,\cdot\rangle_{k_\sigma}$ be the bilinear pairing on 
$\cL_{k_\sigma}(\C)$ defined in \subsecref{SS:balanced case}, respectively. We form an irreducible
representation $(\rho_{\ul{k}},\cL_{\ul{k}}(\C))$ of $G(F_\infty)$ by setting 
\[
\rho_{\ul{k}}
=
\bt_{\sigma\in\Sigma_F}\rho_{k_\sigma}
\quad\text{and}\quad
\cL_{\ul{k}}(\C)
=
\ot_{\sigma\in\Sigma_F}\cL_{k_\sigma}(\C).
\]
Then
$\langle\cdot,\cdot\rangle_{\ul{k}}=\ot_{\sigma\in\Sigma_F}\langle\cdot,\cdot\rangle_{k_\sigma}$
defines a bilinear pairing on $\cL_{\ul{k}}(\C)$.

Let $\cM_{\ul{k}}(D,F;U)$ be the space of $\cL_{\ul{k}}(\C)$-valued atomorphic forms of type 
$\rho_{\ul{k}}$, which consists of functions $f:G(\A_F)\to\cL_{\ul{k}}(\C)$ such that
\begin{align*}
f(a\gamma h h_\infty u)&=\rho_{\ul{k}}(h_\infty)^{-1}f(h),\\
(a\in\A_F^{\x},\,\,\gamma\in G(F),\,\,&h_\infty\in G(F_\infty),\,\,u\in U)
\end{align*}

Let $\cA(G(\A_F))$ be the space of $\C$-valued automorphic forms on $G(\A_F)$ 
(Cf. \cite[section 4]{BorelJacquet1979}). For $\mathbf{v}\in\cL_{\ul{k}}(\C)$ and 
$f\in\cM_{\ul{k}}(D,F;U)$, we define a function 
$\mathit{\Phi}(\mathbf{v}\ot f):G(F)\backslash G(\A_F)\to\C$ by
\[
\mathit{\Phi}(\mathbf{v}\ot f)(h)
=
\langle\mathbf{v},f(h)\rangle_{\ul{k}}.
\]
Then the map $\mathbf{v}\ot f\mapsto \mathit{\Phi}(\mathbf{v}\ot f)$ gives rise to a 
$G(F_\infty)$-equivalent morphism $\cL_{\ul{k}}(\C)\to\cA(G(\A_F))$ for every 
$f\in\cM_{\ul{k}}(D,F;U)$. Let $\cA_{\ul{k}}(D,F;U)$ the subspace of $\cA(G(\A_F))$, consisting of
functions $\mathit{\Phi}(\mathbf{v}\ot f):G(\A_F)\to\C$ for $\mathbf{v}\in\cL_{\ul{k}}(\C)$ and 
$f\in\cM_{\ul{k}}(D,F;U)$.

More generally, suppose $F=F_1\x\cdots\x F_r$, where $F_j$ are totally real number fields. Let 
$D$ be a quaternion $\Q$-algebra and put $D_{F_j}=D\ot_\Q F_j$, $D_F=D\ot_\Q F$. Let 
$U_j\subset G_j(\wh{F}_j)$ be open compact subgroups, where $G_j:=D^{\x}_{F_j}$ viewed as an 
algebraic group defined over $F_j$. Let $\ul{k}_j\in\Z^{\Sigma_{F_j}}$ be sets of positive integers.
Put $U=(U_1,\ldots,U_r)$ and $\ul{k}=(\ul{k}_1,\ldots,\ul{k}_r)$. If $D$ is definite, we define
\[
\cM_{\ul{k}}(D,F;U)
=
\ot_{j=1}^r\cM_{\ul{k}_j}(D_{F_j},F_j;U_j)
\quad\text{and}\quad
\cA_{\ul{k}}(D,F;U)
=
\ot_{j=1}^r\cA_{\ul{k}_j}(D_{F_j},F_j;U_j).
\]
If $D$ is indefinite, similar definitions apply to the spaces $\cN_{\ul{k}}^{[\ul{m}]}(D,F;U)$ and 
$\cA_{\ul{k}}(D,F;U)$.

\subsection{Global settings}\label{SS:global settings}
Let $E$ be an \etale cubic $\Q$-algebra. Then $E$ is (i) $\Q\x\Q\x\Q$ three copies of $\Q$, or (ii)
$F\x\Q$, where $F$ is a quadratic extension of $\Q$, or (iii) $E$ is a field.
Let $\cO_E$ be the maximal order in $E$ and let $\cD_E$ be the absolute discriminant of $E$. Put
\begin{equation}\label{E:constant little c}
c
=
\begin{cases}
3\quad\text{if}\,\,E=\Q\x\Q\x\Q,\\
2\quad\text{if}\,\,E=F\x\Q,\\
1\quad\text{if}\,\,E\,\,\text{is a cubic extension of $\Q$}.
\end{cases} 
\end{equation}
Here $F$ is a quadratic extension over $\Q$. We assume
\begin{equation}\label{E:totally real}
E_\infty=E\ot_{\Q}\R\cong\R\x\R\x\R.
\end{equation}
In particular, $F$ is a real quadratic extension over $\Q$, and $E$ is a real cubic extension over
$\Q$ if it is a field.  

Let $\frak{n}\subset\cO_E$ be an ideal. We have
$\frak{n}=(N_1\Z,N_2\Z,N_3\Z)$ or $\frak{n}=(\frak{n}_F,N\Z)$ according to $E=\Q\x\Q\x\Q$ or 
$E=F\x\Q$, respectively. Here $N_j, N\,(j=1,2,3)$ are positive integers and $\frak{n}_F$ is an 
ideal of $\cO_F$. Let $\ul{k}=(k_1,k_2,k_3)$ be a triple of positive even integers with 
$k_j\geq 2$ for $j=1,2,3$. We put
\begin{equation}\label{E:w}
w=k_1+k_2+k_3-3.
\end{equation}
Let $f_E\in\cM_{\ul{k}}({\rm M}_2,E;K_0(\wh{\frak{n}}))$ be a normalized Hilbert newform of weight 
$\ul{k}$ and level $K_0(\wh{\frak{n}})$ (\cf. \cite[page 652]{Shimura1978}). 
More precisely, if $E=\Q\x\Q\x\Q$, then $f_E=f_1\ot f_2\ot f_3$, where 
$f_j\in\cS_{k_j}(M_1,\Q;K_0(N_j\wh{\Z}))$ is a normalized newform
of weight $k_j$ and level $K_0(N_j\wh{\Z})$. On the other hand, if $E=F\x\Q$, then 
$f_E=g_F\ot f$, where $g_F\in\cS_{(k_1,k_2)}({\rm M}_2,F;K_0(\wh{\frak{n}}_F))$ is a 
normalized Hilbert newform of weight $(k_1,k_2)$ and level $K_0(\wh{\frak{n}}_F)$, and 
$f\in\cS_{k_3}({\rm M}_2,\Q;K_0(N\wh{\Z}))$ is a normalized newform of weight $k_3$ and level 
$K_0(N\wh{\Z})$. Let $\mathbf{f}_E=\mathit{\Phi}(f_E)$ be its adelic lift to 
$\cA_{\ul{k}}({\rm M}_2,E;K_0(\wh{\frak{n}}))$. 
%We assume $f_E$ is $normalized$ in the 
%following sense. Let $\psi_E=\psi \circ {\rm tr_{\Q}^{E}}$ and $\delta_E \in \wh{E}^{\times}$ be 
%a finite idele such that $\delta_{E,\frak{p}}^{-1} \in E_{\frak{p}}^{\times}$ is a generator of 
%the conductor of $\psi_{E,{\frak{p}}}$ for all prime ideals $\frak{p}$ of $E$. 
%We choose $f_E$ so that
%\begin{align}\label{E:normalization of f_E}
%\int_{E \backslash \A_{E}}
%{\bf f}_{E}
%\left(
%\begin{pmatrix} 1&x\\0&1\end{pmatrix}
%\begin{pmatrix} \delta_{E}^{-1}& 0 \\ 0 &1 \end{pmatrix}
%\right)
%\overline{\psi_{E}(x)}dx &=e^{-6\pi},
%\quad
%{\rm Vol}(E\backslash\A_E,dx)=1.
%\end{align}
Let $\itPi$ be the unitary irreducible cuspidal automorphic representation of $\G(\A_E)$ generated 
by $\mathbf{f}_E$. By the tensor product theorem \cite{Flath1979}, $\itPi\cong\ot'_v\itPi_v$, where
$\itPi_v$ are irreducible admissible representations of $\G(E_v)$. 
We define the $L$-function and $\epsilon$-factor associated to $\itPi$ and $r$ as product of local 
$L$-factors and $\epsilon$-factors. That is, we put
\begin{align*}
L(s,\itPi,r)=\prod_{v}L(s,\itPi_v,r_v)
\quad 
\text{and}
\quad
\epsilon(s,\itPi,r,\psi)=\prod_{v}\epsilon(s,\itPi_v,r_v,\psi_v).
\end{align*}
Note that $L(s,\itPi,r )$ is holomorphic at $s=1/2$.

Ichino's formula relates the period integrals of triple products of certain automorphic forms 
on quaternion algebras along the diagonal cycles and the central values of triple $L$-functions.
To describe the choice of the quaternion algebra, we define the local root number 
$\epsilon(\itPi_v)\in\stt{\pm}$ associated to $\itPi_v$ for each place $v$ by the following 
condition
\[
\epsilon(\itPi_v)=1 \Leftrightarrow {\rm Hom}_*(\itPi_v,\C)\neq\stt{0},
\]  
where $*=\G(\Q_p)$ or $\left(\mathfrak{g},K\right)$ according to $v=p$ or $v=\infty$, respectively.
%By the results of \cite{Prasad1990}, \cite{Prasad1992} and \cite{Loke2001},
%we see that if $\epsilon(\itPi_v)=-1$, then the Jacquet-Langlands lift $\itPi_v'$ of $\itPi_v$ to 
%$(B_v\ot E_v)^{\x}$ is non-zero, and ${\rm Hom}_{B^{\x}_v}(\itPi_v',\C)\neq\stt{0}$. Here $B_v$ is
%the quaternion division algebra over $\Q_v$.

In the following, we assume the global root number $\epsilon(\itPi)$ associated to $\itPi$ is 
equal to $1$. Namely, we assume 
\begin{equation}\label{E:global root number}
\epsilon\left(\itPi\right):=\prod_v\epsilon(\itPi_v)=1.
\end{equation}
Notice that $\epsilon(\itPi_v)=1$ for almost all $v$ by the results of 
\cite[Theorem 1.2]{Prasad1990} and \cite[Theorem B]{Prasad1992}.
By this assumption, there is a unique quaternion $\Q$-algebra $D$ such that $D_v$ is 
the division $\Q_v$-algebra if and only if $\epsilon(\itPi_v)=-1$. 
%We viewed $D^{\x}$ as an algebraic group defined over $\Q$. 
Applying \cite[Theorem 1.2]{Prasad1990} and 
\cite[Theorem B]{Prasad1992}, we see that the Jacquet-Langlands lift 
$\itPi^{D}=\ot'_v\itPi_v^D$ of $\itPi$ to $D^{\times}(\A_E)$ exists, where $\itPi_v^D$ is a unitary 
irreducible admissible representation of $D^{\x}(E_v)$.  Moreover, by the way we chose 
$D$, the following local root number condition is satisfied:
\begin{align}\label{E:local root number condition}
\epsilon(\itPi_v)
=
\left \{ \begin{array} {ll}  1 & \mbox{if $D_v$ is the matrix algebra} , \\
-1 & \mbox{if $D_v$ is the division algebra}.
\end{array} \right .
\end{align}
Let $\Sigma_D$ be the ramification set of $D$ and $\Sigma_D^{(\infty)}\subset\Sigma_D$ be the subset 
without the infinite place. For each $v\notin\Sigma_D$, we fix an
isomorphism $\iota_v:{\rm M}_2(\Q_v)\cong D\ot_{\Q}\Q_v$ once and for all. Let $\cO_D$ be the maximal 
order of $D$ such that $D\ot_{\Z}\Z_p=\iota_p({\rm M}_2(\Z_p))$ for all $p\notin\Sigma_D$.
If $R$ is a $\Q$-algebra, we put $D(R):=D\ot_\Q R$. We introduce following three sets of places of 
$\Q$:
\begin{align}
\begin{split}
\Sigma_3
&=\stt{v\mid E_v\cong\Q_v\x\Q_v\x\Q_v},\\
\Sigma_2
&=\stt{v\mid E_v\cong K_v\x\Q_v,\,\text{for some quadratic extension $K_v$ of $\Q_v$}}, \\
\Sigma_1
&=\stt{v\mid \text{$E_v$ is a cubic extension of $\Q_v$}}.
\end{split}
\end{align}
Note that by our assumption, we have $\infty\in\Sigma_3$. Also for every $p\notin\Sigma_D$, 
the map $\iota_p$ induces isomorphisms $D(E_p)\cong{\rm M}_2(E_p)$ and 
$\cO_D\ot_\Z \cO_{E_p}\cong{\rm M}_2(\cO_{E_p})$, where $\cO_{E_p}$ is the maximal order of $E_p$.
For $v\notin\Sigma_2\cap\Sigma_D$, the canonical diagonal embedding $\Q_v \hookrightarrow E_v$ 
induces a diagonal embedding $D_v \hookrightarrow D(E_v)$. On the other hand, for each 
$p\in\Sigma_2\cap\Sigma_D$, we choose an isomorphism $D(K_p)\cong{\rm M}_2(K_p)$ so that the
embedding $D_p\hookrightarrow D(E_p)\cong{\rm M}_2(K_p)\x D_p$ is the identity map in the 
second coordinate, and is given by the one in \subsecref{SS:embeddings} for the first coordinate. 
In any case, we identify $D_v$ as subalgebras of $D(E_v)$ via these embeddings. 
Suppose $E$ is a field, we note that the finite ramification sets $\Sigma^{(\infty)}_{D(F)}$ and 
$\Sigma^{(\infty)}_{D(E)}$ of $D(F)$ and $D(E)$ are given by  
\begin{align*}
\Sigma^{(\infty)}_{D(F)}
&=
\stt{\frak{p}\subset\cO_F\,\,\text{prime ideal}
\mid 
\text{$\frak{p}$ divides $p$ for some $p\in\Sigma_3\cap\Sigma_D$}},\\
\Sigma^{(\infty)}_{D(E)}
&=
\stt{\frak{p}\subset\cO_E\,\,\text{prime ideal}
\mid 
\text{$\frak{p}$ divides $p$ for some $p\in(\Sigma_1\cup\Sigma_3)\cap\Sigma_D$}}.
\end{align*}
We put
\begin{equation}\label{E:N^_}
N^-=\prod_{p\in\Sigma^{(\infty)}_D}p
\quad
\text{and}
\quad
\frak{N}^-_F=\prod_{\frak{p}\in\Sigma^{(\infty)}_{D(F)}}\frak{p}
\quad
\text{and}
\quad
\frak{N}^-_E=\prod_{\frak{p}\in\Sigma^{(\infty)}_{D(E)}}\frak{p}.
\end{equation}

Recall that $\frak{n}$ is an ideal in $\cO_E$ and $\wh{\frak{n}}=\prod_p\frak{n}_p$ is the closure 
of $\frak{n}$ in $\wh{E}$. In the following, we further assume that 
\begin{equation}\label{E:square-free}
\text{$\frak{n}$ is square-free}.
\end{equation}
More precisely, we assume $N_1, N_2$ and $N_3$ are square-free integers if $E=\Q\x\Q\x\Q$ and 
$\frak{n}_F\subset\cO_F$, $N\in\Z$ are square-free if $E=F\x\Q$. 
Let
\begin{equation}\label{E:M}
M=\prod_{p\mid N_{\Q}^{E}(\frak{n})}p.
\end{equation}
If $L>0$ is an integer coprime to $N^-$, we denote by $R'_L$ the standard Eichler order of level $L$
contained in $\cO_D$.  Similar notation is used to indicate the standard Eichler orders of $D(F)$ 
and $D(E)$. 
We define the order $R_{\itPi^D}$ of $D(E)$ by
\[
R_{\itPi^D}
=
\begin{cases}
R'_{N_1/N^-}\x R'_{N_2/N^-}\x R'_{N_3/N^-}
\quad&\text{if $E=\Q\x\Q\x\Q$},\\
R'_{\frak{n}_F/\frak{N}_F^-}\x R'_{N/N^-}
\quad&\text{if $E=F\x\Q$},\\
R'_{\frak{n}/\frak{N}^-_E}
\quad&\text{if $E$ is a field}.
\end{cases}
\]
We mention that the divisibility of each ideals appeared in the definition of $R_{\itPi^D}$ 
follows from the results of \cite{Prasad1990} and \cite{Prasad1992}. We also define an order
$R_{M/N^-}$ of $D$, which is a twist of the standard Eichler order $R'_{M/N^-}$.  More precisely, 
for $p$ such that $p\in\Sigma_{E,2}$ with $p\mid\cD_E$ and 
$\frak{n}\cO_{E_p}=\varpi_{K_p}\cO_{K_p}\x\Z_p$, 
we require
\[
R_{M/N^-}\ot_\Z \Z_p=\pMX p001 K_0(p)\pMX{p^{-1}}{0}{0}{1}.
\]
Notice that these are precisely the places $p$ so that $E_p=K_p\x\Q_p$ with $K_p/\Q_p$ is unramified, 
and $\itPi_p^D=\itPi_p=\pi'_p\bt\pi_p$ where $\pi'_p$ (resp. $\pi_p$) is a special 
(unramified) representation of $\G(K_p)$ (resp. $\G(\Q_p)$).

To describe our formula, we need a notation.
Let $\nu(\itPi)$ be the number of prime $p$ such that
\begin{itemize}
\item
$p\in\Sigma_3$, $\itPi_p=\pi_{1,p}\bt\pi_{2,p}\bt\pi_{3,p}$ and $\pi_{j,p}$ are special 
representations of $\G(\Q_p)$ for $j=1,2,3$.
\item
$p\in\Sigma_2$, $\itPi_p=\pi'_p\bt\pi_p$ and $\pi_p'$ (resp. $\pi_p$) is a special representation
of $\G(K_p)$ (resp. $\G(\Q_p)$).
\item
$p\in\Sigma_2$, $K_p/\Q_p$ is ramified, $\itPi_p=\pi'_p\bt\pi_p$ and $\pi_p'$ (resp. $\pi_p$) 
is a unramified representation (resp. special representation) of $\G(K_p)$ (resp. $\G(\Q_p)$).
\item
$p\in\Sigma_1$ and $\itPi_p$ is a special representation of $\G(E_p)$.
\end{itemize}
%We remain that in following, we always keep the assumptions \eqref{E:totally real},
%\eqref{E:global root number} and \eqref{E:square-free}.
\subsection{Unbalanced case}
Assume $\epsilon(\itPi_\infty)=1$ in this section. We assume without loss of generality that 
$k_3={\rm max}\stt{k_1,k_2,k_3}$. Then $\epsilon(\itPi_\infty)=1$ implies $k_3\geq k_1+k_2$.
In this case, we have 
\[
D^{\x}(E_\infty)=\G(\R)\x\G(\R)\x\G(\R)
\quad
\text{and}
\quad
\itPi^D_\infty=\itPi_\infty,
\]
is the discrete series representation of $D^{\x}(E_\infty)$ of minimal weight $\ul{k}$ and trivial
central character. Let $\cA(D^{\x}(\A_E))$ be the space of $\C$-valued automorphic forms on 
$D^{\x}(\A_E)$ and let $\cA(D^{\x}(\A_E))_{\itPi^D}$ be the underlying space of $\itPi^D$ in
$\cA(D^{\x}(\A_E))$. Put
\[
\cA_{\ul{k}}(D,E;\wh{R}^{\x}_{\itPi^D})[\itPi^D]
=
\cA_{\ul{k}}(D,E;\wh{R}^{\x}_{\itPi^D})\cap\cA(D^{\x}(\A_E))_{\itPi^D}.
\]
By the multiplicity one theorem and the theory of newform, we have
\[
\cA_{\ul{k}}(D,E;\wh{R}^{\x}_{\itPi^D})[\itPi^D]=\C\,\bff^D_E,
\]
for some non-zero element $\bff^D_E\in\itPi^D$. We normalize $\bff^D_E$ in the following way.
Consider $(\bff^D_E)^*(h)=\ol{\bff^D_E(h\tau_\infty)}$ for $h\in D^{\x}(\A_E)$, where
\[
\tau_\infty=\left(\pMX{-1}{0}{0}{1},\pMX{-1}{0}{0}{1},\pMX{-1}{0}{0}{1}\right)\in D^{\x}(E_\infty).
\]
Since $\itPi^D$ is unitary and self-contragredient, we have 
$\b{\itPi}^D\cong\t{\itPi}^D\cong\itPi^D$, where $\b{\itPi}^D$ is the conjugate representation of 
$\itPi^D$. The multiplicity then one theorem implies $(\bff_E^D)^*\in\itPi^D$. 
By the theory of newform, there exists a non-zero constant 
$\alpha$ such that $\bff^D_E(h)=\alpha\cdot\ol{\bff^D_E(h\tau_\infty)}$ for all $h\in D^{\x}(\A_E)$.
Since $((\bff^D_E)^*)^*=\bff^D_E$, we see that $\alpha\b{\alpha}=1$. As $\alpha$ can be written as
$\b{\beta}/\beta$ for some $\beta\in\C^{\x}$, we may assume 
\begin{equation}\label{E:normalization of f^D_E}
\bff^D_E(h\tau_\infty)=\ol{\bff^D_E(h)}
\quad
\text{for all}
\quad
h\in D^{\x}(\A_E).
\end{equation}
Of course this normalization is only up to a non-zero real number, but its enough for us.

Let $f^D_E\in\cN_{\ul{k}}(D,E;\wh{R}^{\x}_{\itPi^D})$ so that $\mathit{\Phi}(f_E^D)=\bff^D_E$.
We define the norm $\langle f^D_E, f^D_E\rangle_{\wh{R}^{\x}_{\itPi^D}}$ of $f^D_E$ as follows. 
Fix a set of representatives $\stt{x_1,\cdots,x_r}$ for the double cosets 
$D^{\x}(E)\backslash D^{\x}(\A_E)/D^{\x}(E_\infty)^+\wh{R}^{\x}_{\itPi^D}$, where
$D^{\x}(E_\infty)^+$ is the three-fold product of ${\rm GL}_2^+(\R)$. We may assume every 
archimedean component of $x_j$ is one for $1\leq j\leq r$. Let 
\[
\Gamma_j=D^{\x}(E)\cap\left(D^{\x}(E_\infty)^+\x x_j\wh{R}^{\x}_{\itPi^D}x_j^{-1}\right),
\quad
1\leq j\leq r.
\]
The functions $f^D_{E,x_j}:\frak{H}^3\to\C$ satisfy the automorphy condition \eqref{E:automorphy} 
for $\gamma\in\Gamma_j$. We define
\[
\langle f^D_E,f^D_E\rangle_{\wh{R}^{\x}_{\itPi^D}}
=
\sum_{j=1}^r
\int_{\Gamma_j\backslash\frak{H}^3}
|f^D_{E,x_j}(z)|^2{\rm Im}(z)^{\ul{k}}d\mu(z),
\quad
z=(z_1,z_2,z_3)\in\frak{H}^3,\,\,
{\rm Im}(z)^{\ul{k}}=\prod_{\ell=1}^3{\rm Im}(z_\ell)^{k_\ell}.
\]
The measure $d\mu(z)$ on $\frak{H}^3$ is given by 
\[
d\mu(z)=\prod_{\ell=1}^3 y_\ell^{-2}dx_\ell dy_\ell
\quad\quad
(z_\ell=x_\ell+iy_\ell,\quad 1\leq\ell\leq 3),
\]
where $dx_\ell$ and $dy_\ell$ are the usual Lebesgue measures on $\R$. Clearly, 
$\langle f^D_E,f^D_E\rangle_{\wh{R}^{\x}_{\itPi^D}}$ is independent of the choice of the set 
$\stt{x_1,\cdots,x_r}$. Similarly, we can define the norm 
$\langle f_E, f_E\rangle_{K_0(\wh{\frak{n}})}$ of $f_E$.

On the other hand, the Petersson norms of $\mathbf{f}_E$ and $\mathbf{f}^D_E$ are given by 
\[
\int_{\A_E^{\x}\G(E)\backslash \G(\A_E)}
|\mathbf{f}_E(h)|^2 dh
\quad
\text{and}
\quad
\int_{\A_E^{\x}D^{\x}(E)\backslash D^{\x}(\A_E)}
|\mathbf{f}^D_E(h)|^2 dh,
\]
where $dh$ are the Tamagawa measures on $\A_E^{\x}\backslash\G(\A_E)$ and 
$\A_E^{\x}\backslash D^{\x}(\A_E)$, respectively.
By \cite[Lemma 6.1 and Lemma 6.3]{IchinoPrasanna}, we have
\begin{align}\label{E:norm relation}
\begin{split}
\langle f_E, f_E\rangle_{K_0(\wh{\frak{n}})} 
&= 
h_E\left[\GL_2(\wh{\cO}_E): K_0(\wh{\frak{n}})\right]
\cD_E^{3/2} \zeta_E(2) 
\int_{\A_E^{\x}\GL_2^{\x}(E)\backslash\GL_2^{\x}(\A_E)} 
|{\bf f}_E(h)|^2dh,\\
\langle f_E^D, f^D_E\rangle_{\wh{R}_{\itPi^D}^{\x}}
&= 
h_E\left[\wh{\cO}_{D(E)}^{\x}:\wh{R}_{\itPi^D}^{\x}\right] 
\prod_{p \mid N^-}(p-1) 
\prod_{p \in \Sigma_1 \cap \Sigma_D}(p-1)^2
\prod_{p \in \Sigma_3 \cap \Sigma_D,p^3 \Vert M}(p^2+p+1)\\
&\x 
\cD_E^{3/2}\zeta_E(2) 
\int_{\A_E^{\x}D^{\x}(E)\backslash D^{\x}(\A_E)} 
|{\bf f}_E^D(h)|^2dh.
\end{split}
\end{align}
Here $h_E:={}^{\sharp}\left (E^{\times}\backslash \A_E^{\x}/E_{\infty}^{\x}\wh{\cO}_E^{\x}\right)$
is the class number of $E$. We mention that
\[
dh=
8\prod_{p\mid N^-}(p-1)^{-1}
\prod_{p\in\Sigma_3\cap\Sigma_D}(p-1)^{-2}
\prod_{p \in \Sigma_1\cap\Sigma_D,p^3\Vert M}(p^2+p+1)^{-1}
\cD_E^{-3/2}\zeta_E(2)^{-1} 
\prod_{v}dh_v,
\]
where $dh$ is the Tamagawa measures on $\A_E^{\x}\backslash D^{\x}(\A_E)$ and $dh_v$ is the Haar 
measure on $E_v^{\x} \backslash D^{\x}(E_v)$ defined in \subsecref{SS:measure for matrix algebra} 
and \subsecref{SS:measure for division algebra} for each place $v$ of $\Q$.

\begin{lm}\label{L:adelic Petersson inner product for GL_2}
We have
\begin{align*}
\langle f_E, f_E\rangle_{K_0(\wh{\frak{n}})} 
&=
2^{-k_1-k_2-k_3+c-3} h_E\cD_EN_{\Q}^{E}(\frak{n})\cdot L(1,\itPi,{\rm Ad}), 
\end{align*}
where $c$ is given by \eqref{E:constant little c}.
\end{lm}
\begin{proof}
By specializing the formula in \cite[Proposition 6]{Wald1985}, we have
\begin{align*}
\int_{\A_E^{\times}\GL_2(E) \backslash \GL_2(\A_E)} |{\bf f}_E(h)|^2 dh
&=
2^{-k_1-k_2-k_3+c-3}\zeta_{E}(2)^{-1}\cD_E^{-1/2}N_{\Q}^{E}(\frak{n})
\left[\GL_2(\wh{\cO}_E):K_0(\wh{\frak{n}})\right]^{-1} L(1,\itPi,{\rm Ad}). 
\end{align*}
The lemma follows form combining this with the equation \eqref{E:norm relation}.
\end{proof}

For each place $v$, let ${\bf t}_v\in D^{\x}(E_v)$ be the element defined in 
\subsecref{SS:raising element} for $\itPi_v^D$ and put $\bft=\ot_v\bft_v$, 
$\h{\bft}=\otimes_p{\bf t}_p$. 
Recall $N^- = \prod_{p \in \Sigma_D}p$ and $M=\prod_{p \mid N_{\Q}^{E}(\frak{n})}p$. Let
\[
\Gamma^D_{M/N^-}
=D^{\x}(\Q)\cap\left(D^{\x}(\R)^+\x\wh{R}^{\x}_{M/N^-}\right)
\subset{\rm SL}_2(\R),
\]
which is a Fuchsian group of the first kind. Remember that $k_3\geq k_1+k_2$. Set
\[
2m=k_3-k_1-k_2.
\] 

Recall that $\nu(\itPi)$ is the non-negative integer defined in the last paragraph of 
\S\ref{SS:global settings}. 
\begin{thm}\label{T:central value formula unbalanced case}
Suppose $f^D_E$ is normalized as \eqref{E:normalization of f^D_E}.
\begin{itemize}
\item[(1)]
We have
\begin{align*}
&\left(\int_{\Gamma^D_{M/N^-}\backslash \frak{H}}(1\ot\delta_{k_2}^m\ot 1)
f_E^D((z,z,-\ol{z}),\h{\bft})y^{k_3-2}dxdy\right)^2
=
2^{-2k_3-1+\nu(\itPi)}M\cD_E^{-1/2}
\frac{\langle f_E^D,f^D_E\rangle_{\wh{R}_{\itPi^D}^{\times}}}
{\langle f_E, f_E\rangle_{K_0(\wh{\frak{n}})}}
L\left(\frac{1}{2},\itPi,r\right).
\end{align*}
%where $z=x+iy$ with $dx,dy$ the usual Lebesgue measures on $\R$, and 
%\begin{align*}
%\nu(\itPi)
%&=
%{}^{\sharp}\{p \in \Sigma_1 \cup \Sigma_2,p^3 \Vert N_{\Q}^{E}(\frak{n}) \}
%+
%{}^{\sharp}\{p \in \Sigma_3 ,p \mid N_{\Q}^{E}(\frak{n}) \} 
%+ 
%{}^{\sharp}\{p \in \Sigma_2,p \Vert N_{\Q}^{E}(\frak{n}),\epsilon_{2,p}^{(2)}=1,
%\cF_p / \Q_p \mbox{ is ramified} \}.
%\end{align*}
\item[(2)]
The central value is non-negative, that is
\[
L\left(\frac{1}{2},\itPi,r\right)\geq 0.
\]
\end{itemize}
\end{thm}

\begin{proof}
By the normalization \eqref{E:normalization of f^D_E}, we have
\[
\int_{\A^\x_E D^\x(E)\backslash D^\x(\A_E)}
\bff^D_E(h)\bff^D_E(h\tau_\infty)dh
=
\int_{\A^\x_E D^\x(E)\backslash D^\x(\A_E)}
\left|\bff^D_E(h)\right|^2dh.
\]
On the other hand, since ${\rm Ad}\tau_{\R}(V_+)=V_+-2\sqrt{-1}I_2$ and $\itPi^D$ has trivial 
central character, we also have
\begin{align*}
\int_{\A^{\x}D^{\x}\backslash D^{\x}(\A)}\itPi_{\infty}^D(\bft_\infty)\bff^D_E(h\h{\bft})dh
&=
\int_{\A^{\x}D^{\x}\backslash D^{\x}(\A)}
\itPi_{\infty}^D(\bft_\infty\tau_\infty)\ol{\bff^D_E(h\h{\bft}})dh\\
&=
\ol{\int_{\A^{\x}D^{\x}\backslash D^{\x}(\A)}
\itPi_{\infty}^D(\bft_\infty)\bff^D_E(h\h{\bft})dh}.
\end{align*}
By Ichino's formula \cite[Theorem 1.1 and Remark 1.3]{Ichino2008})and the choices of Haar measures
in \subsecref{SS:measure for matrix algebra} and \subsecref{SS:measure for division algebra}, we 
find that 
\begin{align*}
\frac{\left \vert
\int_{\A^{\x}D^{\x}\backslash D^{\x}(\A)}
\itPi^D(\bft)\bff^D_E(h)dh
\right \vert^2}
{\int_{\A^\x_E D^\x(E)\backslash D^\x(\A_E)}
\left|\bff^D_E(h)\right|^2dh}
&=
\frac{\left(
\int_{\A^{\x}D^{\x}\backslash D^{\x}(\A)}
\itPi^D(\bft)\bff^D_E(h)dh\right)^2}
{\int_{\A_E^{\x}D^{\x}(E) \backslash D^{\x}(\A_E)} {\bf f}_E^D(h){\bf f}_E(h\tau_\infty)dh}\\
&=
2^{1-c}\prod_{p \mid N^-}(p-1)^{-1}\cdot \frac{\zeta_E(2)}{\zeta_{\Q}(2)^2}\cdot
\frac{L\left(1/2,\itPi,r \right)}{L(1,\itPi,{\rm Ad})}
\cdot \prod_{v}I^*(\itPi_v^D,\bft_v),
\end{align*}
where $c$ is given by \eqref{E:constant little c}. Since $L(1,\itPi,{\rm Ad})>0$ by
\lmref{L:adelic Petersson inner product for GL_2} and $I^*(\itPi_v,\bft_v)>0$ for all $v$ by our
results in the previous sections, we see immediately that assertion $(2)$ holds.  

To drive our formula, we note that from \eqref{E:raising relation} and the definition of 
$\bft_\infty$, the function $\itPi^D({\bf t}){\bf f}_E^D$ is the adelic lift of the automorphic
function 
\begin{align*}
((z_1,z_2,z_3),h) &\mapsto (1\ot\delta_{k_2}^m\ot 1)
f_E^D((z_1,z_2,-\overline{z_3}),h\hat{\bf t}),
\quad
((z_1,z_2,z_3),h)\in\frak{H}^3\x\GL_2(\wh{E}).
\end{align*}  
Applying lemmas 6.1 and 6.3 in \cite{IchinoPrasanna}, we obtain 
\begin{align*}
\left(
\int_{\A^{\x}D^{\x}\backslash D^{\x}(\A)}
\itPi^D(\bft)\bff^D_E(h)dh
\right)^2
&=
\zeta_{\Q}(2)^{-2}\prod_{p\mid M / N^-}(1+p)^{-2}
\prod_{p \mid N^-}(p-1)^{-2}\\
&\x
\left(
\int_{\Gamma^D_{M/N^-}\backslash\frak{H}}(1\ot\delta_{k_2}^m\ot 1)
f_E^D((z,z,-\b{z}), \hat{\bf t})y^{k_3-2}dxdy  
\right)^2.
\end{align*}
The theorem follows from
combining this with \lmref{L:adelic Petersson inner product for GL_2}
and our results for $I^*(\itPi_v^D,\bft_v)$ in \secref{S:local zeta integral: matrix algebra} and
\secref{S:local zeta integral: division algebra}.
\end{proof}

Now we apply \thmref{T:central value formula unbalanced case} to prove the algebraicity of the
central critical values of the triple product $L$-functions in certain cases. We keep the notations
and assumptions in \subsecref{SS:global settings}. We further assume that $E$ is not a field.
To unify our statements, let $K=\Q\x\Q$ or $K=F$. Let $\cD_K$ be its absolute discriminant. 
We have $E=K\x\Q$. Let $\frak{n}_K=(N_1\Z,N_2\Z)$, $N_3=N$ and $g_K=f_1\ot f_2$, $f_3=f$ when 
$K=\Q\x\Q$. In any case, we have
\[
f_E=g_K\ot f
\quad
\text{and}
\quad
\frak{n}=(\frak{n}_K,N\Z).
\]
We define the motivic $L$-function and its associated completed 
$L$-function for ${g}_{\cF}\otimes f$ by
\begin{align*}
L(s,{g}_{\cF}\otimes f) 
= 
\prod_{p}L\left( s-\frac{w}{2},\itPi_p,r_p \right) 
\quad
\text{and}
\quad 
\Lambda(s,{g}_{\cF}\otimes f) 
= 
L\left ( s-\frac{w}{2},\itPi,r\right ).
\end{align*}
Recall that $w$ is given by \eqref{E:w}. Define the Petersson norm of $f$ by
\[
\langle f,f\rangle_{\Gamma_0(N)}
= 
\int_{\Gamma_0(N)\backslash\frak{H}}|f(z)|^2y^{k_3-2}dxdy,
\quad
(z=x+iy).
\]
Here $dx, dy$ are the usual Lebesgue measures on $\R$.

\begin{cor}\label{C:algebraicity for triple, unbalanced case}
Assume either $\epsilon(\itPi)=-1$, or $\epsilon(\itPi_v)=1$ for all $v$. Then for every
$\sigma\in{\rm Aut}(\C)$, we have
\begin{align*}
\left(
\frac{L((w+1)/2,g_{\cF}\otimes f)}
{\cD_{\cF}^{1/2}\pi^{2k_3}\langle f,f\rangle_{\Gamma_0(N)}^2}\right)^{\sigma}
=
\frac{L((w+1)/2,g_{\cF}^{\sigma}\otimes f^{\sigma})}
{\cD_{\cF}^{1/2}\pi^{2k_3}\langle f^{\sigma},f^{\sigma}\rangle_{\Gamma_0(N)}^2}.
\end{align*}
\end{cor}

\begin{proof}
First note that we have $\epsilon(\itPi^\sigma)=\epsilon(\itPi)$ for all $\sigma\in{\rm Aut}(\C)$. 
In fact, this follows from $\itPi^\sigma_\infty\cong\itPi_\infty$ and  
\lmref{L:root number is galois invariant}. Also, if $\epsilon(\itPi)=-1$, then by the results of 
\cite{HarrisKudla2004} and \cite{Prasad2008}, we have 
\[
L\left( \frac{w+1}{2},g_{\cF}\otimes f\right)=0,
\]
(see also the remark after this corollary).
As $\epsilon(\itPi^\sigma)=\epsilon(\itPi)=-1$, we conclude that
\[
L\left( \frac{w+1}{2},g_{\cF}^{\sigma}\otimes f^{\sigma}\right)
=
L\left( \frac{w+1}{2},g_{\cF}\otimes f\right) 
=
0.
\]

Now we assume $\epsilon(\itPi_v)=1$ for all $v$. Then $D={\rm M}_2$.
Let $\iota:\frak{H}\rightarrow\frak{H}^2$ be the diagonal embedding $z\mapsto (z,z)$. 
The $\GL_2(\Q_p)$ component of ${\bf t}_p \in \GL_2(\cF_p)\times \GL_2(\Q_p)$ is equal to $1$ 
for all $p$. Thus we may view $\hat{\bf t}$ as an element in $\GL_2(\wh{\cF})$.
Note that $(1\otimes \delta_{k_2}^{m})\rho(\hat{\bf t})g_{\cF}$
is a nearly holomorphic Hilbert modular form over $\cF$ of weight $(k_1,k_2+2m)$, 
where $\rho$ denote the right translation of $\GL_2(\widehat{K})$. 
Let
\[
\iota^*((1\ot\delta_{k_2}^{m})\rho(\hat{\bft})g_K)(z)
=
(1\ot\delta_{k_2}^{m})\rho(\h{\bft})g_K((z,z),1)
\] 
be its pullback along $\iota$ at the identity cusp. Then it is a 
nearly holomorphic modular form of weight $k_3$ and level $\Gamma_0(M)$.
We consider the period integral 
$\langle \iota^*((1\ot\delta_{k_2}^{m})g_K),f\rangle$ defined by
\begin{align*}
\langle \iota^*((1\ot\delta_{k_2}^{m})\rho(\hat{\bf t})g_K),f\rangle 
=
\int_{\Gamma_0(M)\backslash\frak{H}}
\iota^*((1\ot\delta_{k_2}^{m})\rho(\h{\bft})g_K)(z)\ol{f(z)}y^{k_3-2}dxdy,
\end{align*}
where $z= x+iy$ and $dx, dy$ are the usual Lebesgue measures on $\R$. Let $\sigma\in{\rm Aut}(\C)$. 
By our normalization of $g_K$, we have 
\[
(\iota^*((1\otimes \delta_{k_2}^{m})
\rho(\hat{\bf t})g_{\cF}))^{\sigma} = \iota^*((1\otimes \delta_{k_2}^{m})
\rho(\hat{\bf t})g^{\sigma}_{\cF}).
\]
Since $\iota^*((1\otimes \delta_{k_2}^{m})\rho(\hat{\bf t})g_{\cF})$ 
is nearly holomorphic and $f$ is a newform, by \cite[Theorem 4]{Sturm1980} and \cite{Shimura1976},
we have
\begin{align*}
\left(\frac{\langle\iota^*((1\otimes \delta_{k_2}^{m})
\rho(\hat{\bf t})g_K),f\rangle}{\langle f,f\rangle_{\Gamma_0(N)}}\right)^{\sigma}
=
\frac{\langle(\iota^*((1\ot\delta_{k_2}^{m})\rho(\h{\bft})g_K^{\sigma})),f^{\sigma}\rangle}
{\langle f^{\sigma},f^{\sigma}\rangle_{\Gamma_0(N)}}.
\end{align*}
In particular, we have $\langle \iota^*((1\ot\delta_{k_2}^{m})
\rho(\hat{\bf t})g_K),f \rangle\in\R$.
Note that 
\[
(1\ot\delta_{k_2}^m\ot 1)\rho(\h{\bft})f_E ((z,z,-\ol{z}),1)
= 
\iota^*((1\ot\delta_{k_2}^{m})\rho(\h{\bft})g_K)(\tau)\ol{f(z)}.
\]
By Theorem \ref{T:central value formula unbalanced case}, we have
\begin{align*}
&\frac{\langle \iota^*((1\ot\delta_{k_2}^{m})
\rho(\h{\bft})g_K),f\rangle^2}{\langle f,f\rangle_{\Gamma_0(N)}^2}
=
2^{-2k_3-1+\nu(\Pi)}M\frac{\Lambda((w+1)/2,g_K\ot f)}
{\cD_{\cF}^{1/2}\langle f,f\rangle_{\Gamma_0(N)}^2}.
\end{align*}
Applying $\sigma$ on both sides and note that $\epsilon(\itPi_v^\sigma)=\epsilon(\itPi_v)$,
$I^*(\itPi^\sigma_v,\bft_v)=I^*(\itPi_v,\bft_v)$ for all $v$, as well as 
$\nu(\itPi)=\nu(\itPi^\sigma)$. The corollary follows from applying our central value formula
to the left hand side again.
\end{proof}

\begin{Remark}
We mention that there are three ways to define the triple
$L$-function: $(1)$ from the Galois side, as we did in this paper; $(2)$ by the Langlands-Shahidi
method; $(3)$ by the theory of local zeta integrals \cite{PSR1987}, \cite{Ikeda1992}. However, for
questions of vanishing or nonvanishing at $1/2$, it makes no difference which definition of the 
$L$-function we choose, since any two of them are only different from a finite number of local 
$L$-factors, which we know are all non-vanishing at $1/2$.
\end{Remark}

\subsection{Balanced case}
Assume $\epsilon(\itPi_\infty)=-1$ in this section. 
%This happens precisely when the triplet 
%$\ul{k}=(k_1,k_2,k_3)$ satisfies $k_i<k_j+k_\ell$ for $\stt{i,j,\ell}=\stt{1,2,3}$.
We have
\[
D^{\x}(E_\infty)=\mathbf{H}^{\x}\x\mathbf{H}^{\x}\x\mathbf{H}^{\x}
\quad
\text{and}
\quad
(\itPi^D_{\infty},V_{\itPi_\infty^D})
=
(\rho_{\ul{k}},\cL_{\ul{k}}(\C)).
\]
Let $\cA(D^{\x}(\A_E))_{\itPi^D}$ be the underlying space of $\itPi^D$ in $\cA(D^{\x}(\A_E))$ and 
put
\[
\cA_{\ul{k}}(D,E;\wh{R}^{\x}_{\itPi^D})[\itPi^D]
=
\cA_{\ul{k}}(D,E;\wh{R}^{\x}_{\itPi^D})\cap\cA(D^{\x}(\A_E))_{\itPi^D}.
\]
By the multiplicity one theorem and the theory of newform, there exists a unique (up to constants) 
non-zero element $f^D_E\in\cM_{\ul{k}}(D,E;\wh{R}^{\x}_{\itPi^D})$ such that the map 
$\mathbf{v}\ot f^D_E\mapsto \mathit{\Phi}(\mathbf{v}\ot f^D_E)$ defines a
$D^{\x}(E_\infty)$-isomorphism form $\cL_{\ul{k}}(\C)$ onto 
$\cA_{\ul{k}}(D,E;\wh{R}^{\x}_{\itPi^D})[\itPi^D]$. Let $\mathbf{P}_{\ul{k}}\in\cL_{\ul{k}}(\C)$ 
be the $\mathbf{H}^{\x}$-fixed element given by \eqref{E:invariant vector}. We put
$\mathbf{f}^D_E=\mathit{\Phi}(\mathbf{P}_{\ul{k}}\ot f^D_E)$. Then its immediately form the 
definition that $\mathbf{f}_E^D$ is right $\mathbf{H}^{\x}$-invariant.

To state our central value formula for the balanced case, we need some notations.
Let ${\rm Cl}(R_{\itPi^D})$ and ${\rm Cl}(R_{M/N^-})$ be sets of representatives of
\[
\wh{E}^{\x} D^{\x}(E)\backslash D^{\x}(\widehat{E})/\widehat{R}_{\itPi^D}^{\x}
\quad
\text{and} 
\quad
\wh{\Q}^{\x} D^{\x}(\Q)\backslash D^{\x}(\widehat{\Q})/\widehat{R}_{M/N^-}^{\x},
\] 
respectively. Let $\Gamma_{\alpha}$ be finite sets defined by
\[
\left(D^{\x}(E)\cap\wh{E}^{\x}\,\alpha\,\wh{R}^{\x}_{\itPi^D}\,\alpha^{-1}\right)/E^{\x} 
\quad
\text{or}
\quad
\left(D^{\x}(\Q)\cap\wh{\Q}^{\x}\,\alpha\,\wh{R}^{\x}_{M/N^-}\,\alpha^{-1}\right)/\Q^{\x},
\]
according to $\alpha\in{\rm Cl}(R_{\itPi^D})$ or $\alpha\in{\rm Cl}(R_{M/N^-})$, respectively.
We put 
\[
\langle f^D_E,f^D_E\rangle_{\wh{R}^{\x}_{\itPi^D}}
=
\sum_{\alpha\in{\rm Cl}(R_{\itPi^D})}
\frac{1}{{}^\sharp\Gamma_\alpha}
\langle f^D_E(\alpha),f^D_E(\alpha)\rangle_{\ul{k}}.
\]
For each place $v$, let ${\bf t}_v\in D^{\x}(E_v)$ be the element defined in 
\subsecref{SS:raising element} for $\itPi_v^D$ and put $\bft=\ot_v\bft_v$. Recall that 
$M=\prod_{p\mid N_\Q^E(\frak{n})}p$ and that $\nu(\itPi)$ is the non-negative integer defined in
the last paragraph of \S\ref{SS:global settings}. 

\begin{thm}\label{T:central value formula for balanced case}
\begin{itemize}
\item[(1)]
We have
\[
\left(\sum_{\alpha\in{\rm Cl}(R_{M/N^-})}
\frac{1}{{}^\sharp\Gamma_\alpha}
\langle f^D_E(\alpha\bft),\bfP_{\ul{k}}\rangle_{\ul{k}}\right)^2
=
2^{-(k_1+k_2+k_3+1)+\nu(\itPi)}M\cD_E^{-1/2}
\frac{\langle f^D_E,f^D_E\rangle_{\wh{R}^{\x}_{\itPi^D}}}
{\langle f_E,f_E\rangle_{K_0(\wh{\mf{n}})}}
L\left(\frac{1}{2},\itPi,r\right).
\]                              
%where
%\[
%\nu(\itPi)
%=
%{}^\sharp\{p \in \Sigma_1 \cup \Sigma_2,p^3 \Vert N_{\Q}^{E}(\frak{n}) \}
%+
%{}^\sharp\{p \in \Sigma_3 ,p \mid N_{\Q}^{E}(\frak{n}) \}|
%+ 
%{}^\sharp\{p \in \Sigma_2,p \Vert N_{\Q}^{E}(\frak{n}),\epsilon_{2,p}^{(2)}
%=
%1,\cF_p / \Q_p \mbox{ is ramified} \}.
%\] 
\item[(2)]
The central value is non-negative, that is
\[
L\left(\frac{1}{2},\itPi,r\right)\geq 0.
\]
\end{itemize}                                
\end{thm}

\begin{proof}
$(1)$ By Lemmas 6.1 and 6.3 in \cite{IchinoPrasanna}, we have
\[
\left(\sum_{\alpha\in{\rm Cl}(R_{M/N^-})}
\frac{1}{{}^\sharp\Gamma_\alpha}
\langle f^D_E(\alpha\bft),\bfP_{\ul{k}}\rangle_{\ul{k}}\right)
=
\frac{1}{24}\prod_{p\mid M/ N^-}(1+p)\prod_{p\mid N^-}(p-1)
\int_{\A^{\x}D^{\x}(\Q)\backslash D^{\x}(\A)} \bff^D_E(h\bft)dh,
\]
where $dh$ is the Tamagawa measure on $\A^{\x}\backslash D^{\x}(\A)$. On the other hand, applying
same lemmas, we obtain
\begin{align*}
\langle f^D_E,f_E^D\rangle_{\wh{R}^{\x}_{\itPi^D}}
&= 
2^{-6}\pi^{-3}h_E\cD_E^{3/2}\left[\wh{\cO}_{D(E)}^{\x}:\wh{R}_{\itPi^D}^{\x}\right]
\prod_{p\mid N^-}(p-1)
\prod_{p\in\Sigma_1\cap\Sigma_D}(p-1)^2
\prod_{p\in\Sigma_3\cap\Sigma_D,p^3\Vert M}(p^2+p+1)\\
&\x
\zeta_E(2)\int_{\A^{\x}_E D^{\x}(E)\backslash D^{\x}(\A_E)}
\langle f_E^D(h),f_E^D(h)\rangle_{\ul{k}}dh,
\end{align*}
where $dh$ is the Tamagawa measure on $\A^{\x}_E\backslash D^{\x}(\A_E)$.
Schur's orthogonal relation implies
\[
\int_{\A^{\x}_E D^{\x}(E)\backslash D^{\x}(\A_E)}
\mathbf{f}^D_E(h)\mathbf{f}^D_E(h)dh
=
\frac{\langle\bfP_{\ul{k}},\bfP_{\ul{k}}\rangle_{\ul{k}}}
{(k_1-1)(k_2-1)(k_3-1)}
\int_{\A^{\x}_E D^{\x}(E)\backslash D^{\x}(\A_E)}
\langle f_E^D(h),f_E^D(h)\rangle_{\ul{k}}dh.
\]
The measure $dh$ on the RHS of the equation above is also the Tamagawa measure on 
$\A^{\x}_E\backslash D^{\x}(\A_E)$.
By Ichino's formula \cite[Theorem 1.1 and Remark 1.3 ]{Ichino2008} and the choices of Haar measures
in \subsecref{SS:measure for matrix algebra} and \subsecref{SS:measure for division algebra}, we 
find that
\[
\frac{\left(\int_{\A^{\x}D^{\x}(\Q)\backslash D^{\x}(\A)} \bff^D_E(h\bft)dh\right)^2}
{\int_{\A^{\x}_E D^{\x}(E)\backslash D^{\x}(\A_E)}\mathbf{f}^D_E(h)\mathbf{f}^D_E(h)dh}
=
2^{3-c}3^{-1}\prod_{p\mid N^-}(p-1)^{-1}
\cdot
\frac{\zeta_E(2)}{\zeta_\Q(2)}\cdot
\frac{L(1/2,\itPi, r)}{L(1,\itPi,{\rm Ad})}\cdot
\prod_{v}
I^*(\itPi^D_v,\bft_v).
\]
Here the constant $c$ is given by \eqref{E:constant little c}. The central value formula follows
from the equations above together with 
\lmref{L:adelic Petersson inner product for GL_2} and the results for $I^*(\itPi_v^D,\bft_v)$ in
\secref{S:local zeta integral: matrix algebra} and \secref{S:local zeta integral: division algebra}.

To prove $(2)$, it suffices to show that the ratio
\[
\frac{\left(\int_{\A^{\x}D^{\x}(\Q)\backslash D^{\x}(\A)}\bff^D_E(h\bft)dh\right)^2}
{\int_{\A^{\x}_E D^{\x}(E)\backslash D^{\x}(\A_E)}
\langle f_E^D(h),f_E^D(h)\rangle_{\ul{k}}dh}
\]
is non-negative. To do this we consider $(f^D_E)^*(h)=\ol{f^D_E(h\tau_\infty)}$ for 
$h\in D^{\x}(\A_E)$, where
\[
\tau_\infty
=
\left(\begin{pmatrix}
0&1\\-1&0
\end{pmatrix},
\begin{pmatrix}
0&1\\-1&0
\end{pmatrix},
\begin{pmatrix}
0&1\\-1&0
\end{pmatrix}\right)
\in D^{\x}(E_\infty).
\]
The function $(f^D_E)^*$ satisfy the same conditions as $f^D_E$. By the uniqueness, there exists a 
non-zero constant $\alpha$ such that $f^D_E(h)=\alpha\cdot\ol{f^D_E(h\tau_\infty)}$ for all 
$h\in D^{\x}(\A_E)$. On one hand, we have
\begin{align*}
\int_{\A^{\x}D^{\x}(\Q)\backslash D^{\x}(\A)} 
\langle f_E^D(h\bft),\bfP_{\ul{k}}\rangle_{\ul{k}}dh
&=
\alpha\int_{\A^{\x}D^{\x}(\Q)\backslash D^{\x}(\A)} 
\langle\ol{f_E^D(h\bft\tau)},\bfP_{\ul{k}}\rangle_{\ul{k}}dh\\
&=
\alpha\cdot
\ol{\int_{\A^{\x}D^{\x}(\Q)\backslash D^{\x}(\A)} 
\langle f_E^D(h\bft),\bfP_{\ul{k}}\rangle_{\ul{k}}dh}.
\end{align*}
On the other hand, recall that 
\[
\cH_{\ul{k}}(v,w)
=
\langle v,\itPi^D_\infty(\tau_\infty)\b{w}\rangle_{\ul{k}}\quad v,w\in\cL(\C),
\]
defines an $D^{\x}(E_\infty)$-invariant $Hermitian$ pairing on $V_{\itPi^D_\infty}$. 
We have
\begin{align*}
\int_{\A^{\x}_E D^{\x}(E)\backslash D^{\x}(\A_E)}
\langle f_E^D(h),f_E^D(h)\rangle_{\ul{k}}dh
&=
\int_{\A^{\x}_E D^{\x}(E)\backslash D^{\x}(\A_E)}
\langle f^D_E(h),f^D_E(h)\rangle_{\ul{k}}dh\\
&=
\alpha\int_{\A^{\x}_E D^{\x}(E)\backslash D^{\x}(\A_E)}
\langle f^D_E(h),\ol{f^D_E(h\tau)}\rangle_{\ul{k}}dh\\
&=
\alpha\int_{\A^{\x}_E D^{\x}(E)\backslash D^{\x}(\A_E)}
\cH_{\ul{k}}(f^D_E(h),f^D_E(h))dh.
\end{align*}
This finishes the proof.
\end{proof}

Define the motivic $L$-function for $f_E$ by
\[
L(s,f_E,r)
=
\prod_{p}L\left(s-\frac{w}{2},\itPi_p,r_p\right).
\]
Recall that $w$ is given by \eqref{E:w}.
We have the following corollary, which proves the Deligne's conjecture for the central critical
value of $L(s,f_E,r)$ in the balanced range.

\begin{cor}\label{C:algebraicity for triple for balanced case}
Let $\sigma\in{\rm Aut}(\C)$. We have
\[
\left(
\frac{L\left((w+1)/2,f_E,r\right)}
{\cD_E^{1/2}\pi^{w+2}\langle f_E, f_E\rangle_{K_0(\wh{\mf{n}})}}
\right)^\sigma
=
\frac{L\left((w+1)/2,f^\sigma_E,r\right)}
{\cD_E^{1/2}\pi^{w+2}\langle f^\sigma_E,f^\sigma_E\rangle_{K_0(\wh{\mf{n}})}}. 
\]
\end{cor}

\begin{proof}
If $\epsilon(\itPi)=-1$, then by the same argument in 
\corref{C:algebraicity for triple, unbalanced case}, we have
\[
L\left(\frac{w+1}{2},f_E^\sigma,r\right)
=
L\left(\frac{w+1}{2},f_E,r\right)
=
0,
\]
for every $\sigma\in{\rm Aut}(\C)$. 

Assume $\epsilon(\itPi)=1$ and put
\[
\langle f^D_E,\bfP_{\ul{k}}\rangle_{\ul{k}}
=
\left(\sum_{\alpha\in{\rm Cl}(R_{M/N^-})}
\frac{1}{{}^\sharp\Gamma_\alpha}
\langle f^D_E(\alpha\bft),\bfP_{\ul{k}}\rangle_{\ul{k}}\right)^2.
\] 
By \cite[Lemme II.1.1]{Wald1985B}, we have
\[
\left( 
\frac{\langle f^D_E,\bfP_{\ul{k}}\rangle_{\ul{k}}}
{\langle f^D_E,f^D_E\rangle_{\wh{R}^{\x}_{\itPi^D}}}
\right)^\sigma 
=
\frac{\langle (f^\sigma_E)^D,\bfP_{\ul{k}}\rangle_{\ul{k}}}
{\langle (f^\sigma_E)^D,(f^\sigma_E)^D\rangle_{\wh{R}^{\x}_{\itPi^D}}}
\]
for all $\sigma\in{\rm Aut}(\C)$.  The rest of the proof is similar to that in the last paragraph
of \corref{C:algebraicity for triple, unbalanced case}.
\end{proof}
%The only thing we should mention is that 
%\[
%L(1,\Pi,{\rm Ad})
%=
%2^{k_1+k_2+k_3+3-c}N_{\Q}^{E}(\mf{n})^{-1}h_E^{-1}\cD_E^{-1}
%\|f_E\|_{K_0(\mf{n})},
%\]
%according to our normalization \eqref{E:normalization of f_E}, where $c$ is given by
%\eqref{E:constant little c}.
%The only thing we should mention is that 
%\[
%L(1,\Pi,{\rm Ad})
%=
%2^{k_1+k_2+k_3+3-c}N_{\Q}^{E}(\mf{n})^{-1}h_E^{-1}\cD_E^{-1}
%\|f_E\|_{K_0(\mf{n})},
%\]
%according to our normalization \eqref{E:normalization of f_E}, where $c$ is given by
%\eqref{E:constant little c}.
\section{Applications}
In this section, we prove our main results
of this paper. Let $N_1,N_2$ be positive square free integers, and $\kappa',\kappa$ be positive even integers. 
Put $w=2\kappa+\kappa'-3$. 
Let $N={\rm gcd}(N_1,N_2)$ and $M={\rm lcm}(N_1,N_2)$. Let $f\in S_{\kappa'}(\Gamma_0(N_1))$ and 
$g\in S_{\kappa}(\Gamma_0(N_2))$ be normalized elliptic newforms and ${\bf f}$ and ${\bf g}$ be 
the adelic lifts of $f$ and $g$, respectively. Let $\tau=\ot'_v\tau_v$ and $\pi=\ot'_v \pi_v$ be 
the irreducible unitary cuspidal automorphic representations of $\GL_2(\A)$ generated by 
${\bf f}$ and ${\bf g}$, respectively.

If $F'$ is a cyclic extension of $\Q$ with prime degree, we let $\pi_{F'}$ be the base change lift
of $\pi$ to $\G(\A_{F'})$.  Since $N_2$ is assumed to be square-free, $\pi_{F'}$ is a unitary
irreducible cuspidal automorphic representation of $\G(\A_{F'})$ with trivial central character
\cite{Arthur1989book}.

We define the motivic $L$-function and its associated completed $L$-function for
 ${\rm Sym}^2(g)\otimes f$ by
\begin{align*}
L(s,{\rm Sym}^2(g)\otimes f)
&= \prod_{p}L\left(s-\frac{w}{2},{\rm Sym}^2(\pi_p )\ot\tau_p 
\right), & & \Lambda (s,{\rm Sym}^2(g)\ot f )
=  
\prod_{v}L\left(s-\frac{w}{2},{\rm Sym}^2(\pi_v)\ot\tau_v \right).
\end{align*}
Note that $L(s,{\rm Sym}^2(g)\otimes f)$ is holomorphic at $s=(w+1)/2$.
\begin{cor}\label{C:algebraicity for GL_2 times GL_3 unbalanced case}
Assume $\kappa' \geq 2\kappa$. Let $\epsilon=(-1)^{\kappa'/2-1}$. 
For $\sigma \in {\rm Aut}(\C)$, we have
\begin{align*}
\left( 
\frac{L((w+1)/2,{\rm Sym}^2(g)\otimes f)}{\pi^{3\kappa'/2}(\sqrt{-1})^{\kappa'/2-1}
\langle f,f\rangle \Omega_f^{\epsilon} }\right )^{\sigma}
&=
\frac{L((w+1)/2,{\rm Sym}^2(g^{\sigma})\otimes f^{\sigma})}{\pi^{3\kappa'/2}
(\sqrt{-1})^{\kappa'/2-1}
\langle f^{\sigma},f^{\sigma}\rangle \Omega_{f^{\sigma}}^{\epsilon} }.
\end{align*}
Here $\Omega_f^{\pm}$ are the periods of $f$ defined by Shimura in \cite{Shimura1977}.
\end{cor}

\begin{proof}
Define $\Xi$ to be the set of real quadratic extensions $K / \Q$ such that 
\begin{itemize}
\item
$D_K$ is prime to $M$. 
\item
%$\chi_{K,p}(p)=-\left (\frac{D_K}{p} \right )$ for $p \mid N$,  
$\left(\frac{D_K}{p}\right)=-1$ for $p\mid N$.
\item 
$\left (\frac{D_K}{p} \right ) =1$ for $p\mid M/N_2$.
\end{itemize}
Here $D_K$ is the discriminant of $K /\Q$. Let $K \in \Xi$ and 
$\chi_{\cF}=\otimes_v \chi_{\cF,v} : \cF^{\times} \backslash \A_{\cF}^{\times} \rightarrow \C$ 
be the idele class character associated to $\cF$ by class field theory. 
Put $\itPi=\pi_K\bt\tau$. By the properties of $K$ , we have $\epsilon(\itPi_v)=1$ for all $v$. 
On the other hand, by the results of \cite{Prasad1990} and \cite{Prasad1992}, we have
\[
\epsilon(\itPi_v)
=
\epsilon \left(\frac{1}{2}, \itPi_v,r_v,\psi_v  \right )\chi_{\cF,v}(-1),
\]
for all place $v$. In particular, $\epsilon \left (1/2,\itPi,r,\psi \right)=1$ and the 
matrix algebra ${\rm M}_2$ is the unique quaternion algebra over $\Q$ satisfying 
(\ref{E:local root number condition}). 
We see from the factorization
$\epsilon(s,\itPi,r,\psi)=\epsilon(s,{\rm Sym}^2(\pi)\ot\tau,\psi)\epsilon(s,\tau\ot\chi_K,\psi)$
that 
\[
\epsilon\left (\frac{1}{2},{\rm Sym}^2(\pi)\ot\tau,\psi\right)
=
\epsilon\left (\frac{1}{2},\tau\ot\chi_K,\psi\right ).
\]
If $\epsilon \left (1/2,\tau\ot\chi_{\cF},\psi\right)=-1$, 
then $\epsilon \left (1/2,{\rm Sym}^2(\pi)\ot\tau,\psi \right)=-1$. On the other hand, by 
\corref{C:epsilons foctor of symmetric square and cubic are galois invariant}, we also have
$\epsilon \left (1/2,{\rm Sym}^2(\pi^\sigma)\otimes\tau^\sigma,\psi \right)=-1$.
Therefore
\[
L\left( \frac{w+1}{2},{\rm Sym}^2(g^{\sigma})\otimes f^{\sigma} \right)
=
L\left( \frac{w+1}{2},{\rm Sym}^2(g)\otimes f \right) 
=
0,
\]
for all $\sigma \in {\rm Aut}(\C)$ by functional equation. Otherwise, by the nonvanishing theorem of \cite{FriedbergHoffstein1995}, there exists $K' \in \Xi$ such that $L\left (\kappa'/2,f\otimes \chi_{K'}\right )\neq 0$. Let $\sigma \in {\rm Aut}(\C)$. 
By \cite{Shimura1977}, we have
\begin{align*}
\left( 
\frac{L(\kappa'/2,f\otimes \chi_{K'})}{D_{K'}^{1/2}\pi^{\kappa'/2}(\sqrt{-1})^{\kappa'/2}
\Omega_f^{-\epsilon}}\right )^{\sigma}
&=
\frac{L(\kappa'/2,f^{\sigma}\ot\chi_{{K'}})}{D_{K'}^{1/2}\pi^{\kappa'/2} 
(\sqrt{-1})^{\kappa'/2}\Omega_{f^{\sigma}}^{-\epsilon}},\\
\left( 
\frac{\langle f,f  \rangle }{(\sqrt{-1})^{\kappa'-1}\Omega_f^+\Omega_f^-}\right )^{\sigma}
&=
\frac{\langle f^{\sigma},f^{\sigma}\rangle }{(\sqrt{-1})^{\kappa'-1}
\Omega_{f^{\sigma}}^+\Omega_{f^{\sigma}}^-}.
\end{align*}
Let $g_{K'}$ be the normalized Hilbert modular newform associated to $\pi_{K'}$, the base change lift of $\pi$ to $\GL_2(\A_{K'})$. 
By \corref{C:algebraicity for triple, unbalanced case}, we have
\begin{align*}
\left(
\frac{L((w+1)/2,g_{K'}\otimes f)}
{D_{K'}^{1/2}\pi^{2\kappa'}\langle f,f\rangle^2}\right)^{\sigma}
=
\frac{L((w+1)/2,g_{K'}^{\sigma}\otimes f^{\sigma})}
{D_{K'}^{1/2}\pi^{2\kappa'}\langle f^{\sigma},f^{\sigma}\rangle^2}.
\end{align*}
Note that $g_{K'}^{\sigma} = (g^{\sigma})_{K'}$.
Now the corollary follows from combining these equations
with the following factorization 
\[
L \left (\frac{w+1}{2},g_{K'}\otimes f \right )
=
L \left ( \frac{w+1}{2}, {\rm Sym}^2(g)\otimes f\right )L
\left ( \frac{\kappa'}{2},f\otimes \chi_{{K'}}\right ).
\]
This completes the proof.
\end{proof}

Define the Petersson norm of $g$ by
\[
\langle g,g  \rangle 
= 
\int_{\Gamma_0(N_2) \backslash \frak{H}}|g(\tau)|^2y^{\kappa-2}d\tau.
\]

\begin{cor}\label{C:algebraicity for GL_2 times GL_3 balanced case}
Assume $2\kappa>\kappa'$ and $N_1>1$. Let $\epsilon=(-1)^{\kappa'/2-1}$. 
For $\sigma\in{\rm Aut}(\C)$, we have
\[
\left( 
\frac{L((w+1)/2,{\rm Sym}^2(g)\ot f)}
{\pi^{2\kappa+\kappa'/2-1} (\sqrt{-1})^{\kappa'/2-1}\langle g,g\rangle^2\Omega^{\epsilon}_f}      
\right)^\sigma
=
\frac{L((w+1)/2,{\rm Sym}^2(g^\sigma)\ot f^\sigma)}
     {\pi^{2\kappa+\kappa'/2-1}(\sqrt{-1})^{\kappa'/2-1}\langle g^\sigma,g^\sigma\rangle^2
     \Omega^\epsilon_{f^{\sigma}}}.      
\]
Here $\Omega^{\pm}_f$ are the periods of $f$ defined by Shimura in \cite{Shimura1977}.
\end{cor}
\begin{proof}
Since $N_1>1$, by the non-vanishing results of \cite{FriedbergHoffstein1995}, we can choice a 
real quadratic field $\cF$ with fundamental discriminant $\cD>0$ such that
$L\left(\kappa'/2,f\otimes \chi_{\cD}\right)\neq 0$, where $\chi_{\cD}$ is the Dirichlet character
associated to $\cF/\Q$ by class field theory.
Let $g_{\cF}$ be the normalized Hilbert modular newform associated to $\pi_{\cF}$ and $\bfg_{\cF}\in\pi_{\cF}$
be its adelic lift. By equation \eqref{E:norm relation}, the Petersson norm of $g_{\cF}$ is given by
\[
\langle g_{\cF}, g_{\cF}\rangle 
= 
h_{\cF}\left[ \GL_2(\widehat{\mathcal{O}}_{\cF}) : K_0(N_2\cO_{\cF})\right]
\cD_{\cF}^{3/2} \zeta_{\cF}(2)
\int_{\A_{\cF}^{\times}\GL_2^{\times}(\cF)\backslash\GL_2^{\times}(\A_{\cF})} 
|\bfg_{\cF}(h)|^2 dh,
\]
where $h_{\cF}$ is the class number of $\cF$ and $dh$ is the Tamagawa measure on 
$\A_{\cF}^{\x}\backslash\G(\A_{\cF})$.
We have 
\[
\left( 
\frac{\langle g_{\cF},g_{\cF}\rangle}{\langle g,g\rangle^2}
\right)^\sigma
=
\frac{\langle (g^\sigma)_{\cF},(g^\sigma)_{\cF}\rangle}
     {\langle g^\sigma,g^\sigma\rangle^2}. 
\]
This equality follows from combining the factorization 
\[
L(1,\pi_{\cF},{\rm Ad})
=
L(1,\pi,{\rm Ad})L(1,\pi,{\rm Ad},\chi),
\]
and a result of Sturm \cite{Sturm1989}. The rest of proof is similar to that of 
\corref{C:algebraicity for GL_2 times GL_3 unbalanced case} except we use 
\corref{C:algebraicity for triple for balanced case} here instead. 
\end{proof}

We consider the case when $E$ is a cubic Galois extension over $\Q$.
Under some assumptions, we prove Deligne's conjecture for the central critical value of 
$L(s,{\rm Sym}^3(f))$, where $L(s,{\rm Sym}^3(f))$ is the motivic $L$-function for ${\rm Sym}^3(f)$
defined by
\[
L(s,{\rm Sym}^3(f))
=
\prod_{p}L\left(s-\frac{w}{2},{\rm Sym}^3(\tau_p)\right).
\]
Here $w=3\kappa'-3$.

\begin{cor}\label{C:algebraicity for symmetric cubic}
Assume $N_1>1$ and there exist a cubic Dirichlet character $\chi$ such that 
$L\left(\frac{\kappa'}{2} ,f\ot\chi\right)\neq 0$. For $\sigma \in {\rm Aut}(\C)$, we have
\begin{align*}
\left(
\frac{L((w+1)/2,{\rm Sym}^3(f))}{\pi^{2\kappa'-1}(\sqrt{-1})^{\kappa'}
\langle f,f\rangle (\Omega_f^{\epsilon})^2} 
\right )^{\sigma}
=
\frac{L((w+1)/2,{\rm Sym}^3(f^{\sigma}))}{\pi^{2\kappa'-1}(\sqrt{-1})^{\kappa'}
\langle f^{\sigma},f^{\sigma}\rangle (\Omega_{f^{\sigma}}^{\epsilon})^2}.
\end{align*}

\end{cor}

\begin{proof}
The argument is similar to that of \corref{C:algebraicity for GL_2 times GL_3 unbalanced case} 
and \corref{C:algebraicity for GL_2 times GL_3 balanced case}.
Let $E$ be the cubic Galois extension of $\Q$ associated to $\chi$ by global class filed theory, 
and $\chi_E$ be a idele class character associated to $E/\Q$. 
Let $f_{E}$ be the normalized Hilbert modular newform associated to $\pi_{E}$ and
$\langle f_E,f_E\rangle=\|f_E\|_{K_0(N_1\cO_E)}$ be the Petersson norm of $f_E$.
The factorization 
\[
L(1,\pi_E,{\rm Ad})
=
L(1,\pi,{\rm Ad})
L(1,\pi,{\rm Ad},\chi_E)
L(1,\pi,{\rm Ad},\b{\chi}_E),
\]
toghther with Sturm's result \cite{Sturm1980} yield
\[
\left( 
\frac{\langle f,f\rangle^3}{\langle f_E,f_E\rangle}
\right)^\sigma
=
\frac{\langle f^\sigma,f^\sigma\rangle^3}{\langle (f^\sigma)_E,(f^\sigma)_E\rangle}. 
\]
Using again \cite{Shimura1977}, we have
\begin{align*}
\left( 
\frac{L(\kappa'/2,f\otimes \chi)}
{G(\chi)\pi^{\kappa'/2}(\sqrt{-1})^{\kappa'/2}\Omega^{-\epsilon}_f}
\right)^\sigma
&=
\frac{L(\kappa'/2,f^\sigma\otimes\chi)}
{G(\chi^\sigma)\pi^{\kappa'/2} (\sqrt{-1})^{\kappa'/2}\Omega^{-\epsilon}_{f^\sigma}},\\
\left( 
\frac{L(\kappa'/2,f\ot\b{\chi})}
{G(\b{\chi})\pi^{\kappa'/2}(\sqrt{-1})^{\kappa'/2}\Omega^{-\epsilon}_f}
\right)^\sigma
&=
\frac{L(\kappa'/2,f^\sigma\ot\b{\chi})}
{G(\b{\chi}^\sigma)\pi^{\kappa'/2} (\sqrt{-1})^{\kappa'/2}\Omega^{-\epsilon}_{f^\sigma}},\\
\left(   
\frac{\langle f,f\rangle}{(\sqrt{-1})^{\kappa'-1}\Omega_f^+\Omega_f^-}
\right)^\sigma 
&=
\frac{\langle f^\sigma,f^\sigma\rangle}
{(\sqrt{-1})^{\kappa'-1}\Omega_{f^\sigma}^+\Omega_{f^\sigma}^-} .
\end{align*}
Here $G(\chi)$ (resp. $G(\overline{\chi})$) is the Gauss sum associated to $\chi$ 
(resp. $\overline{\chi}$) defined in \cite{Shimura1977}.  
Notice that since the Hecke field of $f$ is totally real, we have
\[
L\left(\frac{\kappa'}{2},f\ot\b{\chi}\right)
=
\ol{L\left(\frac{\kappa'}{2},f\ot\chi\right)}
\neq 0.
\]
Also, as $E/\Q$ is Galois, $\cD_E$ is a square. The corollary then follows from these equations
together with \corref{C:algebraicity for triple for balanced case} and the factorization  
\[
L\left(\frac{w+1}{2},f_E,r\right)
=
L\left(\frac{w+1}{2},{\rm Sym}^3(f)\right)
L\left(\frac{\kappa'}{2},f\ot\chi\right)
L\left(\frac{\kappa'}{2},f\ot\b{\chi}\right).  
\]
This finishes the proof. 
\end{proof}

\section*{Appdendix : Root numbers and Deligne's periods}
The appendix consists of two parts. In the first part, we explain that the various local root
numbers are invariant under the Galois action. In the second part, we compute the Deligne's 
period of the motive associated to ${\rm Sym}^2(g)\ot f$.  

\subsection*{Root numbers}
Let $F$ be a non-archimedean local field of characteristic zero. Let $E$ be an \etale cubic 
algebra over $F$. Let $D$ be the quaternion division algebra over $F$. The definition of the local 
root numbers in \subsecref{SS:global settings} is valid in more general settings. More precisely,
let $\itPi$ be an irreducible admissible generic representation of $\G(E)$ whose 
central character is trivial on $F^{\x}$. Define $\epsilon(\itPi)\in\stt{\pm 1}$ by the 
following condition
\[
\epsilon(\itPi)=1 \Leftrightarrow {\rm Hom}_{\G(F)}(\itPi,\C)\neq\stt{0}.
\]  
We call $\epsilon(\itPi)$ the (local) root number associated to $\itPi$. We can also define the 
local root number for the archimedean case as we did in the same section, but in terms of the 
category of $\left(\mathfrak{g},K\right)$-modules. The results of Prasad
\cite{Prasad1990}, \cite{Prasad1992} imply that if $\epsilon(\itPi)=-1$, then the 
Jacquet-Langlands lift $\itPi'$ of $\itPi$ to $D^{\x}(E)$ is non-zero, and 
${\rm Hom}_{D^{\x}}(\itPi',\C)\neq\stt{0}$.

Let $\sigma\in{\rm Aut}(\C)$ and $(\pi,V)$ be a representation of a group $G$.
Following \cite[section 1]{Wald1985B}, we define a representation 
$\pi^\sigma$ of $G$ as follows. Let $V'$ be another $\C$-linear space with a
$\sigma$-linear isomorphism $t':V\to V'$. We define
\[
\pi^\sigma(g)=t'\circ\pi(g)\circ t'^{-1},\quad g\in G.
\]
If $\pi=\chi$ is a character, then $\chi^\sigma=\sigma\left(\chi\right)$.

Notice that $\itPi^\sigma$ is an irreducible admissible generic representation of $\G(E)$ with 
central character $\omega^\sigma_\itPi$, where $\omega_\itPi$ is the central character of $\itPi$.
In particular, $\omega_\itPi$ is trivial on $F^{\x}$ if and only if $\omega^\sigma_\itPi$ is.

\begin{lma}\label{L:root number is galois invariant}
For every $\sigma\in{\rm Aut}(\C)$, we have
\[
\epsilon\left(\itPi^\sigma\right)=\epsilon\left(\itPi\right).
\]
\end{lma}
\begin{proof}
This follows immediately from the definition. Indeed, we have a $\sigma$-linear isomorphism,
\[
{\rm Hom}_{\G(F)}\left(\itPi,\C\right)\stackrel{\sim}{\longrightarrow}
{\rm Hom}_{\G(F)}\left(\itPi^\sigma,\C\right),
\]
defined by $\ell\mapsto\ell':=\sigma\circ\ell\circ t'^{-1}$. This finishes the proof.
\end{proof} 

We have a corollary.
\begin{cora}\label{C:epsilons foctor of symmetric square and cubic are galois invariant}
Let $\psi$ be a non-trivial additive character of $F$.
Let $\pi$ and $\tau$ be two irreducible admissible generic 
representations of $\G(F)$ with central character $\omega_\pi$ and $\omega_\tau$, respectively. 
Let $\sigma\in{\rm Aut}(\C)$.
\begin{itemize}
\item[(1)] Suppose $\omega_\pi^2\cdot\omega_\tau=1$. Then 
$\epsilon\left(1/2,{\rm Sym}^2(\pi)\ot\tau,\psi\right)\in\stt{\pm 1}$ is independent
of $\psi$, and we have
\[
\epsilon\left(\frac{1}{2},{\rm Sym}^2(\pi^\sigma)\ot\tau^\sigma\right)
=
\epsilon\left(\frac{1}{2},{\rm Sym}^2({}\pi)\ot\tau\right).
\]
\item[(2)] Suppose $\omega^3_\pi=1$. Then 
$\epsilon\left(1/2,{\rm Sym}^3(\pi),\psi\right)\in\stt{\pm 1}$ is 
independent of $\psi$, and we have
\[
\epsilon\left(\frac{1}{2},{\rm Sym}^3(\pi^\sigma)\right)
=
\epsilon\left(\frac{1}{2},{\rm Sym}^3(\pi)\right).
\]
\end{itemize}
\end{cora}
\begin{proof}
We only prove $(1)$ since the proof of $(2)$ is similar. Let $\itPi=\pi\bt\pi\bt\tau$. By the results
of \cite[section 8]{Prasad1990}, \cite[Theorem 1.2]{WTG2008} and \cite[Theorem 4.4.1]{Rama2000},
$\epsilon\left(1/2,\itPi,r,\psi\right)\in\stt{\pm 1}$ is independent of $\psi$ and 
\[
\epsilon(\itPi)=\epsilon\left(\frac{1}{2},\itPi,r\right).
\]
Since $\tau\ot\chi_{\pi}$ is self-dual, we have 
$\epsilon\left(1/2,\tau\ot\omega_\pi,\psi\right)\in\stt{\pm 1}$ is independent of $\psi$. 
By the factorization 
\[
\epsilon\left(s,\itPi,r,\psi\right)
=
\epsilon\left(s,\tau\ot\chi\omega_{\pi},\psi\right)
\epsilon\left(s,{\rm Sym}^2(\pi)\ot\tau,\psi\right),
\]
we see that 
$\epsilon\left(1/2,{\rm Sym}^2(\pi)\ot\tau,\psi\right)\in\stt{\pm 1}$ is also independent of $\psi$.

By the lemma A, we only need to show
\[
\epsilon\left(\frac{1}{2},\tau^\sigma\ot\omega^\sigma_\pi\right)
=
\epsilon\left(\frac{1}{2},\tau\ot\omega_\pi\right).
\] 
But this is a result of 
\cite[Proposition I.2.5]{Wald1985B}, which said
\begin{align*}
\epsilon\left(\frac{1}{2},\tau^\sigma\ot\omega^\sigma_\pi\right)
=
\epsilon\left(\frac{1}{2},\tau^\sigma\ot\omega^\sigma_\pi,\psi^\sigma\right)
=
\sigma\left(\epsilon\left(\frac{1}{2},\tau\ot\omega_\pi,\psi\right)\right)
=
\epsilon\left(\frac{1}{2},\tau\ot\omega_\pi\right),
\end{align*}
where $\psi^\sigma=\sigma\circ\psi$. This completes the proof.
\end{proof}

\subsection*{Deligne's periods}
Notations being the same as in the previous section. In \cite{Yoshida2001}, H. Yoshida 
define fundamental periods of a pure motive over $\Q$ whose construction including Deligne's periods. 
In particular, Yoshida give a formula for Deligne's periods of the tensor product of two pure motives 
over $\Q$ in terms of the fundamental periods of the two motives. 
Specializing the formula of Yoshida, C. Bhagwat give a more explicit formula in 
\cite{Bhagwat2014} for pure motive whose nonzero Hodge numbers are one.
In this section we use formula in \cite{Bhagwat2014} to compute Deligne's periods of the motive 
associated to ${\rm Sym}^2(g)\otimes f$. It turns out that there are no  fundamental periods 
other than Deligne's periods in our case. Let $M(f)$ and $M(g)$ be the motives over $\Q$ with 
coefficients in $\Q(f)$ and $\Q(g)$, respectively. For their construction, see 
\cite{Scholl1990}. We consider the symmetric square ${\rm Sym}^2M(g)$ 
(resp. the symmetric cube ${\rm Sym}^3M(f)$) of the motive $M(g)$ 
(resp. $M(f)$). We follow \cite{Deligne1979} and \cite{Yoshida2001} for the conventions 
and notations. All motives below have coefficients in $\Q(f,g)$, and we write $\sim$ for the 
equivalence relation defined by $\Q(f,g)^{\times}$.

In \cite{Bhagwat2014}, the exponent of $(c^+(M')c^-(M'))$ in Theorem 3.2 should be $a_{k'}^*-k$ 
in stead of $a_{k'}^*-k-1$.

\begin{propa}\label{P:Deligne's periods}
We have
\begin{align*}
c^{\pm}({\rm Sym}^2M(g)\otimes M(f))
&= 
\left \{ \begin{array}{ll} (2\pi\sqrt{-1})^{3-3\kappa}(\sqrt{-1})^{1-\kappa'}
\langle f,f\rangle \Omega_f^{\pm} & \mbox{ if }\kappa' \geq 2\kappa,\\
(2\pi\sqrt{-1})^{2-\kappa-\kappa'}
\langle g,g\rangle^2
\Omega_f^{\pm} & \mbox{ if }2\kappa>\kappa'. 
\end{array} \right .\\
c^{\pm}({\rm Sym}^3M(f))
&=
(2\pi\sqrt{-1})^{1-\kappa'}(\sqrt{-1})^{1-\kappa'}\langle f,f\rangle (\Omega_f^{\pm})^2.
\end{align*}
\end{propa}

\begin{proof}
Put $M=M(f)$, $M'={\rm Sym}^2M(g)$, and $N=M\otimes M'$. 
By Prop. 7.7 in \cite{Deligne1979}, we have
\begin{align*}
\mathsf{w}(M')&=2\kappa-2 ,&  d^+(M')&=2 , & d^{-}(M')&=1,\\
c^+(M')&=c^+(M(g))c^-(M(g))\delta (M(g)), & c^+(M')&=c^+(M(g))c^-(M(g)).
\end{align*}
Since $L_{\infty}(M',s)
=
\zeta_{\C}(s)\zeta_{\R}(s-\kappa+2)$. We have
\begin{align*}
H_B(M')\otimes_{\Q}\C 
&= 
H^{0,2\kappa-2}(M')\oplus H^{\kappa-1,\kappa-1}(M')\oplus H^{2\kappa-2,0}(M'),\\
H_{DR}(M')
&=
F^{0}(M')\supsetneq F^{\kappa-1}(M') \supsetneq F^{2\kappa-2}(M') \supsetneq \{0\}.
\end{align*}
It is well known that
\begin{align*}
\mathsf{w}(M)
&=
\kappa'-1 , & d^{+}(M)&=1, & d^{-}(M)
&=1,\\
c^{\pm}(M)
&=\Omega_f^{\pm}, & \delta(M)&=(2\pi \sqrt{-1})^{1-\kappa'} , & c^+(M)c^-(M)
&=(\sqrt{-1})^{1-\kappa'}\langle f,f\rangle . \\
\end{align*}
The Hodge decomposition and the Hodge filtration are given by
\begin{align*}
H_B(M)\otimes_{\Q}\C 
&= 
H^{0,\kappa'-1}(M)\oplus H^{\kappa'-1,0}(M) , \\
H_{DR}(M)
&=
F^{0}(M) \supsetneq F^{\kappa'-1}(M) \supsetneq \{0\}.
\end{align*}
For the motive $N$, we have $\mathsf{w}(N)=2\kappa+\kappa'-3=w$, and $d^{\pm}(N)=3$. 
(Since $d^{\pm}(N)=d^+(M)d^+(M')+d^-(M)d^-(M')=3$.) Following the notation in \S 3 
of \cite{Bhagwat2014}, we have
\begin{align*}
p_1&=0 , & p_2&=\kappa'-1 , \\
q_1&=0, & q_2 &=\kappa-1, & q_3&=2\kappa-2, \\
k&=1 , & k' &=1, & \epsilon(M')&=1.
\end{align*}
Note that $\mathcal{P}=\mathcal{P}'=\emptyset.$
\\Assume $\kappa' \geq 2\kappa$. Then we have  
\[
L_{\infty}(N,s)
=
\zeta_{\C}(s)\zeta_{\C}(s-(\kappa-1))\zeta_{\C}(s-(2\kappa-2)).
\]
Therefore, 
\begin{align*}
H_B(N)\otimes_{\Q}\C 
&= 
H^{0,\omega}(N) \oplus H^{\kappa-1,2\kappa'-2}(N) \oplus H^{2\kappa-2,\kappa'-1}(N) 
\oplus H^{\kappa'-1,2\kappa-2}(N)\oplus H^{2\kappa'-2,\kappa-1}(N) \oplus H^{\omega,0}(N),\\
H_{DR}(N)
&=
F^{0}(N) \supsetneq F^{\kappa-1}(N) \supsetneq F^{2\kappa-2}(N) \supsetneq F^{\kappa'-1}(N) 
\supsetneq F^{2\kappa'-2}(N)  \supsetneq F^{\omega}(N) \supsetneq \{0\}.
\end{align*}
In the notation of \S\S 2.2 in \cite{Bhagwat2014}, we have $k_0=3$ and $r_3=2\kappa-2$. 
Thus $a_1=3$, $a_2=0$, $a_1^*=1$, and $a_3^*=1$. Note that $\check{M'}=M'(2\kappa-2)$. 
Therefore, by equation  (5.1.7) in \S 5 of \cite{Deligne1979}, we have
\begin{align*}
\delta(M') &\sim c^+(M')c^-(\check{M'})^{-1}\\
&=c^+(M')c^-(M')^{-1}(2\pi \sqrt{-1})^{-(2\kappa-2)d^-(M')}\\
&=\delta(M(g)) (2\pi \sqrt{-1})^{-(2\kappa-2)d^-(M')}\\
&=(2\pi \sqrt{-1})^{3-3\kappa}.
\end{align*}
By Theorem 3.2 in \cite{Bhagwat2014}, we have
\begin{align*}
c^{\pm}(N)&=c^{\pm}(M)\delta(M')(c^+(M)c^-(M))\\
&=c^{\pm}(M)(2\pi \sqrt{-1})^{3-3\kappa} (c^+(M)c^-(M))\\
&=(2\pi\sqrt{-1})^{3-3\kappa}(\sqrt{-1})^{1-\kappa'}\langle f,f\rangle \Omega_f^{\pm}.
\end{align*}
\\Assume $2\kappa > \kappa'$. Then we have  
\[
L_{\infty}(N,s)
=
\zeta_{\C}(s)\zeta_{\C}(s-(\kappa-1))\zeta_{\C}(s-(\kappa'-1)).
\]
If $\kappa<\kappa'$, then
\begin{align*}
H_B(N)\otimes_{\Q}\C 
&= 
H^{0,\omega}(N) \oplus H^{\kappa-1,\kappa+\kappa'-2}(N) \oplus H^{\kappa'-1,2\kappa-2}(N) 
\oplus H^{2\kappa-2,\kappa'-1}(N)\oplus H^{\kappa+\kappa'-2,\kappa-1}(N) \oplus H^{\omega,0}(N),\\
H_{DR}(N)
&=
F^{0}(N) \supsetneq F^{\kappa-1}(N) \supsetneq F^{\kappa'-1}(N) \supsetneq F^{2\kappa-2}(N) 
\supsetneq F^{\kappa+\kappa'-2}(N)  \supsetneq F^{\omega}(N) \supsetneq \{0\}.
\end{align*}
In this case, we have $k_0=3$ and $r_3=\kappa'-1$.
\\If $\kappa>\kappa'$, then
\begin{align*}
H_B(N)\otimes_{\Q}\C 
&= 
H^{0,\omega}(N) \oplus H^{\kappa'-1,2\kappa-2}(N) \oplus H^{\kappa-1,\kappa+\kappa'-2}(N) 
\oplus H^{\kappa+\kappa'-2,\kappa-1}(N)\oplus H^{2\kappa-2,\kappa'-1}(N) \oplus H^{\omega,0}(N),\\
H_{DR}(N)
&=
F^{0}(N) \supsetneq F^{\kappa'-1}(N) \supsetneq F^{\kappa-1}(N) \supsetneq F^{\kappa+\kappa'-2}(N) 
\supsetneq F^{2\kappa-2}(N)  \supsetneq F^{\omega}(N) \supsetneq \{0\}.
\end{align*}
In this case, we have $k_0=3$ and $r_3=\kappa-1$.
\\If $\kappa=\kappa'$, then
\begin{align*}
H_B(N)\otimes_{\Q}\C 
&= 
H^{0,\omega}(N)\oplus H^{\kappa-1,2\kappa-2}(N)\oplus H^{2\kappa-2,\kappa-1}(N)
\oplus H^{\omega,0}(N),\\
H_{DR}(N)
&=
F^0(N) \supsetneq F^{\kappa-1}(N) \supsetneq F^{2\kappa-2}(N) \supsetneq F^{\omega}(N) 
\supsetneq \{0\}.
\end{align*}
In this case, we have $k_0=2$ and $r_2=\kappa-1$.\\
In all cases, we have $a_1=2$, $a_2=1$, $a_1^*=2$, and $a_3^*=0$.
By Theorem 3.2 in \cite{Bhagwat2014}, we have
\begin{align*}
c^{\pm}(N)
&=c^{\pm}(M)\delta(M)(c^+(M)c^-(M))\\
&\sim \Omega_f^{\pm}(2\pi \sqrt{-1})^{1-\kappa'}(c^+(M(g)c^-(M(g)))^2\delta(M(g))\\
&\sim \Omega_f^{\pm}(2\pi \sqrt{-1})^{1-\kappa'}(\sqrt{-1})^{2-2\kappa}
\langle g,g\rangle^2 (2\pi \sqrt{-1})^{1-\kappa}\\
&\sim (2\pi \sqrt{-1})^{2-\kappa-\kappa'}\langle g,g\rangle^2\Omega_f^{\pm}.
\end{align*}
The formula for $c^{\pm}({\rm Sym}^3M(f))$ follows from Prop. 7.7 in 
\cite{Deligne1979}. This completes the proof.
\end{proof}

\bibliographystyle{alpha}
\bibliography{ref}
\end{document}

%% file: amssymbol.tex
\usepackage{amsmath}
\usepackage{amscd,amsthm,amssymb,amsfonts}
\usepackage{mathrsfs}
\usepackage{dsfont}
\usepackage{stmaryrd}
\usepackage{euscript}
\usepackage{expdlist}
\usepackage{enumerate}

\theoremstyle{plain}
\newtheorem{thm}{Theorem}[section]

\newtheorem*{thma}{Theorem A}
\newtheorem*{thmb}{Theorem B}

\newtheorem*{thm*}{Theorem}

\newtheorem{lm}[thm]{Lemma}
\newtheorem{cor}[thm]{Corollary}
\newtheorem*{cor*}{Corollary}
\newtheorem{prop}[thm]{Proposition}
\newtheorem*{conj*}{Conjecture}

\theoremstyle{remark}

\theoremstyle{definition}
\newtheorem*{defn*}{Definition}
\newtheorem{Remark}[thm]{Remark}
\newtheorem{I_Remark*}{Remark}
\newtheorem{defn}[thm]{Definition}

\newcommand{\nc}{\newcommand}

\newcommand{\beq}{\begin{equation}}
\newcommand{\eeq}{\end{equation}}
\newcommand{\bpmx}{\begin{pmatrix}}
\newcommand{\epmx}{\end{pmatrix}}
\newcommand{\bbmx}{\begin{bmatrix}}
\newcommand{\ebmx}{\end{bmatrix}}
\newcommand{\wh}{\widehat}
\newcommand{\wtd}{\widetilde}

\newcommand{\beqcd}[1]{\begin{equation*}\label{#1}\tag{#1}}
\newcommand{\eeqcd}{\end{equation*}}

\numberwithin{equation}{section}

\def\parref#1{\ref{#1}}
\def\thmref#1{Theorem~\parref{#1}}

\def\propref#1{Proposition~\parref{#1}}
\def\corref#1{Corollary~\parref{#1}}     
\def\secref#1{\S\parref{#1}}

\def\lmref#1{Lemma~\parref{#1}}
\def\subsecref#1{\S\parref{#1}}
\def\defref#1{Definition~\parref{#1}}

\def\makeop#1{\expandafter\def\csname#1\endcsname
  {\mathop{\rm #1}\nolimits}\ignorespaces}
\makeop{Hom}   \makeop{End}   \makeop{Aut}   %\makeop{Isom}
\makeop{Pic} \makeop{Gal}       \makeop{Div} \makeop{Lie}
\makeop{PGL}   \makeop{Corr} \makeop{PSL} \makeop{sgn} \makeop{Spf}
 \makeop{Tr} \makeop{Nr} \makeop{Fr} \makeop{disc}
\makeop{Proj} \makeop{supp} \makeop{ker}   \makeop{Im} \makeop{dom}
\makeop{coker} \makeop{Stab} \makeop{SO} \makeop{SL} \makeop{SL}
\makeop{Cl}    \makeop{cond} \makeop{Br} \makeop{inv} \makeop{rank}
\makeop{id}    \makeop{Fil} \makeop{Frac}  \makeop{GL} \makeop{SU}
\makeop{Trd}   \makeop{Sp} \makeop{Tr}    \makeop{Trd} \makeop{Res}
\makeop{ind} \makeop{depth} \makeop{Tr} \makeop{st} \makeop{Ad}
\makeop{Int} \makeop{tr}    \makeop{Sym} \makeop{can} \makeop{SO}
\makeop{torsion} \makeop{GSp} \makeop{Tor}\makeop{Ker} \makeop{rec}
\makeop{Ind} \makeop{Coker}
 \makeop{vol} \makeop{Ext} \makeop{gr} \makeop{ad}
 \makeop{Gr}\makeop{corank} \makeop{Ann}
\makeop{Hol} %Holomorphic
\makeop{Fitt} \makeop{Mp} \makeop{CAP}

\DeclareMathAlphabet{\mathpzc}{OT1}{pzc}{m}{it}
\def\makebb#1{\expandafter\def
  \csname bb#1\endcsname{{\mathbb{#1}}}\ignorespaces}
\def\makebf#1{\expandafter\def\csname bf#1\endcsname{{\bf
      #1}}\ignorespaces}
\def\makegr#1{\expandafter\def
  \csname gr#1\endcsname{{\mathfrak{#1}}}\ignorespaces}
\def\makescr#1{\expandafter\def
  \csname scr#1\endcsname{{\EuScript{#1}}}\ignorespaces}
\def\makecal#1{\expandafter\def\csname cal#1\endcsname{{\mathcal
      #1}}\ignorespaces}
\def\doLetters#1{#1A #1B #1C #1D #1E #1F #1G #1H #1I #1J #1K #1L #1M
                 #1N #1O #1P #1Q #1R #1S #1T #1U #1V #1W #1X #1Y #1Z}
\def\doletters#1{#1a #1b #1c #1d #1e #1f #1g #1h #1i #1j #1k #1l #1m
                 #1n #1o #1p #1q #1r #1s #1t #1u #1v #1w #1x #1y #1z}
\doLetters\makebb   \doLetters\makecal  \doLetters\makebf
\doLetters\makescr
\doletters\makebf   \doLetters\makegr   \doletters\makegr

\normalsize

\makeop{Ram} \makeop{Rep} \makeop{mass}

\makeop{Bl}

\def\cA{{\mathcal A}}  %automorphic forms
\def\cB{\EuScript B}
\def\cD{\mathcal D}

\def\cF{{\mathcal F}}  %Hida family

\def\cL{{\mathcal L}}
\def\cH{{\mathcal H}}

\def\cJ{\mathcal J}
\def\cM{\mathcal M}

\def\cO{\mathcal O}
\def\cS{{\mathcal S}}
\def\cf{{\mathcal f}}
\def\cW{{\mathcal W}}

\def\cJ{\mathcal J}
\def\cN{\mathcal N}

\def\cU{\mathcal U}

\def\bfH{\mathbf H}

\def\bff{\mathbf f}

\def\bbI{\mathbb I}

\newcommand{\Z}{\mathbf Z}
\newcommand{\Q}{\mathbf Q}
\newcommand{\R}{\mathbf R}
\newcommand{\C}{\mathbf C}
\newcommand{\A}{\mathbf A}    % for adele

\def\fraka{{\mathfrak a}}

\def\etale{{\'{e}tale }}
\def\ot{\otimes}

\def\ol{\overline}  \nc{\opp}{\mathrm{opp}} \nc{\ul}{\underline}

\def\XYmatrix{\xymatrix@M=8pt} % make \xymatrix not too cluttered
\def\ncmd{\newcommand}
\ncmd{\xysubset}[1][r]{\ar@<-2.5pt>@{^(-}[#1]\ar@<2.5pt>@{_(-}[#1]}
\ncmd{\XYmatrixc}[1]{\vcenter{\XYmatrix{#1}}}
\ncmd{\xyto}[1][r]{\ar@{->}[#1]}
\ncmd{\xyinj}[1][r]{\ar@{^(->}[#1]}
\ncmd{\xysurj}[1][r]{\ar@{->>}[#1]}
\ncmd{\xyline}[1][r]{\ar@{-}[#1]}
\ncmd{\xydotsto}[1][r]{\ar@{.>}[#1]}
\ncmd{\xydots}[1][r]{\ar@{.}[#1]}
\ncmd{\xyleadsto}[1][r]{\ar@{~>}[#1]}
\ncmd{\xyeq}[1][r]{\ar@{=}[#1]} \ncmd{\xyequal}[1][r]{\ar@{=}[#1]}
\ncmd{\xyequals}[1][r]{\ar@{=}[#1]}
\ncmd{\xymapsto}[1][r]{l\ar@{|->}[#1]}\ncmd{\xyimplies}[1][r]{\ar@{=>}[#1]}
\ncmd{\xyiso}{\ar[r]_-{\sim}}
\def\injxy{\ar@{^(->}}

\newcommand{\pMX}[4]{\begin{pmatrix}
{#1}& {#2}\\
{#3}&{#4}\end{pmatrix} }

\newcommand{\seesaw}[4]{{#1}\ar@{-}[rd]\ar@{-}[d]&{#2}\ar@{-}[d]\\
{#3}\ar@{-}[ru]&{#4}}

\def\cf{\mbox{{\it cf.} }}

\def\x{{\times}}

\newcommand\stt[1]{\left\{#1\right\}}